\documentclass{amsart}

\DeclareMathOperator{\Hom}{Hom}

\DeclareMathOperator{\id}{id}

\DeclareMathOperator{\ad}{ad}

\DeclareMathOperator{\im}{im}
\DeclareMathOperator{\ssc}{sc}

\DeclareMathOperator{\Int}{Int}

\DeclareMathOperator{\Gal}{Gal}

\DeclareMathOperator{\cor}{cor} 
\DeclareMathOperator{\Br}{Br} 
 
\DeclareMathOperator{\alg}{alg} 
 
\DeclareMathOperator{\Inf}{Inf} 
\DeclareMathOperator{\Res}{Res} 
\DeclareMathOperator{\res}{res}

\DeclareMathOperator{\tran}{tran} 
 
\DeclareMathOperator{\source}{source}
\DeclareMathOperator{\target}{target} 
\DeclareMathOperator{\eval}{eval} 
\DeclareMathOperator{\Cor}{Cor} 
 
\DeclareMathOperator{\bsc}{bsc}
\DeclareMathOperator{\der}{der} 
 
\DeclareMathOperator{\Loc}{Loc}

\usepackage{amscd}
\usepackage{amssymb}

\numberwithin{equation}{section}

\newtheorem{theorem}{Theorem}[section]
\newtheorem{corollary}[theorem]{Corollary}
\newtheorem{lemma}[theorem]{Lemma}
\newtheorem{proposition}[theorem]{Proposition}

\theoremstyle{definition}
\newtheorem{definition}[theorem]{Definition} 
\newtheorem{remark}[theorem]{Remark} 
\newtheorem{example}[theorem]{Example}

\begin{document}

\title[$\mathbf B(G)$ for all local and global fields]
{$\mathbf B(G)$ for all local and global fields}

\author{Robert\ E.\ Kottwitz}
\address{Robert E. Kottwitz\\Department of Mathematics\\ University of Chicago\\ 5734 University
Avenue\\ Chicago, Illinois 60637}

\email{kottwitz@math.uchicago.edu}



\maketitle 
\setcounter{tocdepth}{1}

\tableofcontents

\section{Introduction}

\subsection{Overview}
Let $F$ be a local or global field. In this paper we are going to define 
and study a functor 
$
G \mapsto B(F,G)
$
from the category of linear algebraic $F$-groups to the category of pointed
sets. When $G$ is commutative (e.g.~a torus) $B(F,G)$ will actually be an
abelian group. For $p$-adic fields $F$ the functor $G \mapsto B(F,G)$ is
naturally isomorphic to the functor $G \mapsto \mathbf B(G)$ studied in
\cite{IsoII} (which agrees with $B(G)$ from \cite{IsoI} when $G$ is
connected). There is a natural inclusion 
\begin{equation}\label{eq.IncHBIntro}
H^1(F,G) \hookrightarrow B(F,G),
\end{equation} 
so $B(F,G)$ can be thought of as an enlargement of $H^1(F,G)$. 

The study of $B(F,G)$ breaks into two parts. To get the theory off the
ground, one must first treat the case in which $G$ is an $F$-torus $T$. 
Let $K$ be a finite Galois extension of $F$ that splits $T$, and let $G(K/F)$ 
denote the Galois group of $K/F$. 
Then 
$H^1(F,T)=H^1(G(K/F),T(K))$, and Tate-Nakayama theory gives us a 
$G(K/F)$-module $X(K)$ (discussed in more detail  
later in this introduction) such that 
\begin{equation}\label{eq.int1}
H^1(G(K/F),T(K)) \xleftarrow{\simeq} H^{-1}(G(K/F),X_*(T) \otimes X(K)),
\end{equation}
the isomorphism being given by cup product with a canonical class 
\[
\alpha(K/F) \in H^2(G(K/F),\mathbb D_{K/F}(K)),
\]
where $\mathbb D_{K/F}$ is the protorus over $F$ whose character group is
$X(K)$. 
Now the Tate cohomology group 
$H^{-1}(G(K/F),X_*(T) \otimes X(K))$ is by definition the subgroup of 
$\bigl(X_*(T) \otimes X(K)\bigr)_{G(K/F)}$ obtained as the kernel of the
norm map  
\[
N_{K/F}:\bigl(X_*(T) \otimes X(K)\bigr)_{G(K/F)} \to
\bigl(X_*(T) \otimes X(K)\bigr)^{G(K/F)},  
\] 
and in this paper we will extend 
 \eqref{eq.int1}  to an isomorphism 
\begin{equation}\label{eq.int12}
B(F,T) \xleftarrow{\simeq} \bigl(X_*(T) \otimes X(K)\bigr)_{G(K/F)}.
\end{equation} 

The part of this paper that treats tori is very much inspired by two
sources. The first is Tate's article \cite{T}, which is used heavily
throughout the early sections of the paper. The second is Satz 2.3 in 
\cite{LR}. In fact Satz 2.3 of Langlands and Rapoport can be viewed as a
way of constructing  certain special elements in $B(F,T)$ (though they do
not phrase things in this way), at least when $F$ is a number field. 
In this paper we pursue such ideas systematically and end up with the
isomorphism \eqref{eq.int12}. 

The second part of the study of $B(F,G)$ consists in going from tori to
general connected reductive groups. As usual \cite{IsoI,IsoII,CTT,EST} 
this is done in two steps. First one goes from tori to connected reductive 
groups with simply connected derived group, and then one uses
$z$-extensions to go from these to general connected reductive groups. In
several respects our treatment has  been influenced by Borovoi's work
\cite{B2}. 

The rest of this introduction will summarize the main
results in the paper, but before doing so I want to express my
gratitude to T.~Kaletha and M.~Rapoport for encouraging me to flesh out and
write up the rough ideas I had on this topic, and for sharing 
with me their ideas about the relation between $\kappa_G(b)$ and the 
Newton point of $b$    (see subsection \ref{sub.NewtonAndKappa}
as well as \ref{subsub.KalRap}). 
I would also like to thank them, as well as T.~Haines, for some very 
helpful comments on a preliminary version of this paper. 

\subsection{Definition of $B(F,G)$} 
For any finite Galois extension $K/F$ we  consider the $G(K/F)$-module 
\[
X(K):=\begin{cases}
\mathbb Z &\text{ if $F$ is local,} \\
\mathbb Z[V_K]_0 &\text{ if $F$ is global,} 
\end{cases}
\]
where $\mathbb Z[V_K]$ is the free abelian group on the set $V_K$ 
of places of $K$,  and $\mathbb
Z[V_K]_0$ is the kernel of the homomorphism $\mathbb Z[V_K] \to\mathbb Z$
defined by $\sum_{v \in V_K}n_v v \mapsto \sum_{v \in V_K} n_v$. We 
define $\mathbb D_{K/F}$ to be the $F$-group of multiplicative type whose
character group is $X(K)$. When $F$ is local, $\mathbb D_{K/F}=\mathbb
G_m$, and when $F$ is global, $\mathbb D_{K/F}$ is an interesting protorus
over $F$. 

The Tate-Nakayama isomorphisms are given by cup product with a canonical
element 
\[
\alpha(K/F) \in H^2(G(K/F),\mathbb D_{K/F}(K)).
\] 
We choose an extension 
\[
1 \to \mathbb D_{K/F}(K) \to \mathcal E(K/F) \to G(K/F) \to 1
\] 
whose associated cohomology class is $\alpha(K/F)$. This extension is an
example of a Galois gerb for $K/F$, as in \cite{LR}. 

Using this extension (and the protorus $\mathbb D_{K/F}$), we define (see
subsection \ref{sub.DefH1Alg}), for each linear algebraic group $G$ over
$F$, a pointed set $H^1_{\alg}(\mathcal E(K/F),G(K))$. Up to canonical
isomorphism, this pointed set is independent of the choice of a specific
extension $\mathcal E(K/F)$ having $\alpha(K/F)$ as its associated
cohomology class. This is due to the vanishing of $H^1(G(K/F),\mathbb
D_{K/F}(K))$. We have to prove many such vanishing theorems; this is the
main purpose of Appendix \ref{App.A}. 

Given a larger finite Galois extension $L \supset K$ there are natural
$G(L/F)$-maps $p:X(K) \hookrightarrow X(L)$ and $j:X(L) \twoheadrightarrow
X(K)$. Moreover $j$ induces an isomorphism $\gamma:X(L)_{G(L/K)} \to X(K)$. 
Using $p$, one forms an inflation map 
\[
H^1_{\alg}(\mathcal E(K/F),G(K)) \to H^1_{\alg}(\mathcal E(L/F),G(L)). 
\]
Using these inflation maps as transition morphisms, we form a pointed set 
$B(F,G)$ as the  colimit of $H^1_{\alg}(\mathcal E(K/F),G(K))$, with $K$
varying over the directed set of finite Galois extensions of $F$ in some
fixed separable closure $\bar F$ of $F$. 

Readers familiar with \cite{LR} (or \cite{Saa}) will understand that, for
$F=\mathbb Q_p$, the category of representations of the Galois gerb
$\mathcal E(K/F)$ is equivalent to the category of isocrystals having all
slopes in $[K:\mathbb Q_p]^{-1}\mathbb Z \subset \mathbb Q$. 

In the same vein, for any $p$-adic field $F$, the pointed set $B(F,G)$ is
naturally isomorphic to the pointed set $\mathbf B(G)$ in \cite{IsoII}. 
The Tannakian reasoning required to justify this last statement is standard
enough to be left to the reader. In any case 
we start from scratch in this paper, proving everything we need about
$B(F,G)$ in the
$p$-adic case directly, without appealing to \cite{IsoI,IsoII}. 

\subsection{General discussion of $B(F,G)$ for linear algebraic groups $G$} 
\subsubsection{} For any finite separable extension $E/F$ there is a
restriction map 
\[
B(F,G) \to B(E,G) 
 \] 
and a Shapiro isomorphism (see section \ref{sec.Shapiro}) 
\[
B(F,R_{E/F} G_0)=B(E,G_0). 
\]
\subsubsection{} For any place $u$ of a global field $F$ there is a
localization map (see section~\ref{sec.Loc}) 
\[
B(F,G) \to B(F_u,G).
\] 
\subsubsection{} There is a Newton map (see subsection
\ref{sub.NewtonMapping}) 
\begin{equation}\label{eq.SlopeMapIntro} 
 B(F,G) \to [\Hom_{\bar F}(\mathbb D_F,G)/G(\bar F)]^{\Gamma},  
\end{equation} 
where $\Gamma:=\Gal(\bar F/F)$ and $\mathbb D_F:=\projlim_{K} \mathbb
D_{K/F}$, the limit being taken over the directed set of finite Galois
extensions $K$ of $F$ in $\bar F$. The kernel of the Newton map is the
image of $H^1(F,G)$ under the inclusion \eqref{eq.IncHBIntro}. 

\subsubsection{} 
Inside the target of the Newton map is the subset $\Hom_F(\mathbb
D_F,Z(G))$, where $Z(G)$ denotes the center of $G$. The preimage of 
$\Hom_F(\mathbb
D_F,Z(G))$ under the Newton map is by definition the set $B(F,G)_{\bsc}$ of
\emph{basic} elements in $B(F,G)$. Obviously $B(F,G)_{\bsc}$ contains the
image of $H^1(F,G)$ under the inclusion \eqref{eq.IncHBIntro}. 

\subsubsection{} 
When $G$ is connected, the total localization map 
\[
B(F,G) \to \prod_{u \in V_F} B(F_u,G)
\] 
takes values in  $\bigoplus_{u \in V_F}
B(F_u,G)$, by which we mean the subset of the direct product consisting of
families of elements that are trivial at all but finitely many places $u
\in V_F$. See Corollary  \ref{cor.ResDirProd} for this. 

\subsubsection{}
Let 
\[
1 \to Z \xrightarrow{i} G' \xrightarrow{p} G \to 1
\] 
be a short exact sequence of linear algebraic $F$-groups in which $Z$ is a 
central torus in $G'$. Then the natural map   
\begin{equation}\label{eq.GG'ppBI}
p:B(F,G') \to B(F,G) 
\end{equation}
is surjective (see Proposition
\ref{prop.ZactB}). Moreover the map \eqref{eq.GG'ppBI} induces a bijection 
between $B(F,G)$ and the quotient of
$B(F,G')$ by the action of $B(F,Z)$. Similarly \eqref{eq.GG'ppBI} induces a
bijection  between $B(F,G)_{\bsc}$ and the quotient of
$B(F,G')_{\bsc}$ by the action of $B(F,Z)$. These facts  are needed whenever
we use $z$-extensions to reduce results about $B(F,G)$ for general
connected reductive groups to the special case of ones with simply
connected derived group. 

\subsection{Discussion of $B(F,G)$ for connected reductive groups}  
\subsubsection{} 
Now let $G$ be a connected reductive $F$-group. Then $Z(G)$ is a group of
multiplicative type, and we denote by $C(G)$ the biggest torus in $Z(G)$. 
We write $\Lambda_G$ for Borovoi's algebraic fundamental
group of $G$.
Restricted to basic elements,  the Newton map yields 
\[
B(F,G)_{\bsc} \to \Hom_F(\mathbb D_F,C(G))=(\Lambda_{C(G)} \otimes
X^*(\mathbb D_F))^{\Gamma}. 
\] 

\subsubsection{}  Choose a finite Galois extension $K$ of $F$ in $\bar F$ such  
 that $\Gal(\bar F/K)$ acts trivially on $\Lambda_G$, and put 
\[
A(F,G):=(\Lambda_G \otimes X(K))_{G(K/F)},
\]
this group being independent of the choice of $K$, up to canonical
isomorphism. Then (see section \ref{sec.KappaGforB}) there is a natural
map 
\[
\kappa_G:B(F,G) \to A(F,G). 
\] 
\subsubsection{}\label{subsub.KalRap}

Let $b \in B(F,G)_{\bsc}$. Then the image of $b$ under the Newton map is
determined by $\kappa_G(b)$. More precisely, the square (see Proposition 
\ref{prop.NewtKappa}) 
\begin{equation}\label{CD.NewtIntroKappa}
\begin{CD}
B(F,G)_{\bsc} @>{\kappa_G}>> A(F,G)\\
@V{Newton}VV @VNVV \\
(\Lambda_{C(G)} \otimes X^*(\mathbb D_F))^{\Gamma} @>i>> (\Lambda_G \otimes
X^*(\mathbb D_F))^{\Gamma}
\end{CD}
\end{equation}
commutes, and the map $i$ is injective.  
Here the vertical map $N$ is the composite 
\[
(\Lambda_G \otimes X(K))_{G(K/F)} \xrightarrow{N_{K/F}}(\Lambda_G \otimes
X(K))^{G(K/F)}
\hookrightarrow (\Lambda_G \otimes X^*(\mathbb D_F))^{\Gamma}
\]  
and the bottom arrow $i$ is induced by the inclusion $C(G)  \hookrightarrow
G$. (We are taking $K$ large enough that $\Gal(\bar F/K)$ acts
trivially on $\Lambda_G$, and so $A(F,G)=(\Lambda_G \otimes
X(K))_{G(K/F)}$.) 

For arbitrary $b \in B(F,G)$ there is a compatibility between $\kappa_G(b)$ 
and the Newton point of $b$. For this see Proposition \ref{prop.NewtKappaBG}, 
which generalizes part of Theorem 1.15 of Rapoport-Richartz \cite{RapRich}.  

\subsubsection{} 

Since the map $i$ in  \eqref{CD.NewtIntroKappa} 
is injective,  we may view 
$(\Lambda_{C(G)} \otimes X^*(\mathbb D_F))^{\Gamma}$ as a subset
of $(\Lambda_G
\otimes X^*(\mathbb D_F))^{\Gamma}$ and then form its preimage $A_0(F,G)$ under
$N$.  Propositions  \ref{prop.LocImKap} and \ref{prop.ImKapGlob} assert 
that  the image of
$B(F,G)_{\bsc}$ under $\kappa_G$ is equal to $A_0(F,G)$. (It is clear from
the  commutative square \eqref{CD.NewtIntroKappa} that the image is
contained in $A_0(F,G)$,  so the real point of the propositions is to
prove the reverse inclusion.)

\subsubsection{} When $F$ is nonarchimedean local, the map $\kappa_G$
restricts to a bijection (see Proposition \ref{prop.BscLocalMain}) 
\begin{equation}\label{eq.LocalBijectionIntro}
B(F,G)_{\bsc}  \simeq A(F,G).
\end{equation}

\subsubsection{} 

When $F=\mathbb R$, the set $B(F,G)_{\bsc}$ can be understood in terms of
$B(F,T)$ for any fundamental maximal $\mathbb R$-torus $T$ in $G$ (see 
Lemma \ref{lem.BscReal}). 

\subsubsection{}

When $F=\mathbb C$, we have (see subsection \ref{sub.Cmplx}) 
\begin{equation}
B(\mathbb C,G)_{\bsc}=\Lambda_{C(G)}.
\end{equation}

\subsubsection{}

When $F$ is global, the square 
\begin{equation}\label{CD.BAMainIntro}
\begin{CD}
B(F,G)_{\bsc} @>>> \prod_{u \in S_{\infty}} B(F_u,G)_{\bsc}\\
@V{\kappa_G}VV @VVV \\
A(F,G) @>>> \prod_{u \in S_{\infty}} A(F_u,G)
\end{CD}
\end{equation}
is cartesian (see Proposition \ref{prop.BAMain}). So $B(F,G)_{\bsc}$ is a
fiber product involving three sets that are easy to understand, and
therefore
$B(F,G)_{\bsc}$ may itself be regarded as well understood. In the function
field case
$S_\infty$ is empty, and so $\kappa_G$ induces an isomorphism
$B(F,G)_{\bsc} \simeq A(F,G)$, just as in the nonarchimedean local case. 

The picture of $B(F,G)_{\bsc}$ given by the cartesian square 
\eqref{CD.BAMainIntro} is further enhanced by Proposition
\ref{prop.CokGlobalG}, which tells us that the image of the total
localization map 
\[
B(F,G)_{\bsc} \to \bigoplus_{u \in V_F} B(F_u,G)_{\bsc}
\]
is the kernel of a certain natural map 
\[
\bigoplus_{u \in V_F} B(F_u,G)_{\bsc} \to (\Lambda_G)_{\Gamma}. 
\]

\subsection{Tori} 

For a torus $T$ over a global field there is more to be said. Let $K$ be a
finite Galois extension of $F$ that splits $T$. We write $M$ for the 
$G(K/F)$-module $X_*(T)$. Then there is a commutative diagram 
\begin{equation*}
\begin{CD}
H^1(K/F,T(K)) @>>> H^1(K/F,T(\mathbb A_K)) @>>> H^1(K/F,T(\mathbb
A_K)/T(K)) \\ @A{\simeq}AA @A{\simeq}AA @A{\simeq}AA \\
H^{-1}(K/F,M \otimes X_3) @>>> H^{-1}(K/F,M \otimes X_2) @>>>
H^{-1}(K/F,M \otimes X_1)
\end{CD}
\end{equation*}
with exact rows, in which the vertical arrows are Tate-Nakayama
isomorphisms. Here $X_1=\mathbb Z$, $X_2=\mathbb Z[V_K]$ and 
$X_3=\mathbb Z[V_K]_0$, and so there is a natural short exact sequence 
\[
0 \to X_3 \to X_2 \to X_1 \to 0
\] 
of Galois modules. 
Earlier in this introduction we wrote
$X(K)$ rather than $X_3$. We are now removing $K$ from the notation, in
order to save space, and adding the subscript $3$ in order to have uniform
notation in the  diagram. 

In section \ref{sec.GTN} we enlarge all
the groups in the diagram, obtaining 
\begin{equation*}
\begin{CD}
B_3(F,T) @>>> B_2(F,T) @>>> B_1(F,T) @>>> 0\\ 
@A{\simeq}AA @A{\simeq}AA @A{\simeq}AA \\ 
(M \otimes X_3)_{G(K/F)} @>>> (M
\otimes X_2)_{G(K/F)} @>>> (M \otimes X_1)_{G(K/F)} @>>> 0
\end{CD}
\end{equation*}
in which again the rows are  exact and the vertical maps are 
isomorphisms. 
Here 
\begin{itemize}
\item $B_3(F,T)=H^1_{\alg}(\mathcal E_3(K/F),T(K))$, \\
\item $B_2(F,T)=H^1_{\alg}(\mathcal E_2(K/F),T(\mathbb A_K))$, \\
\item $B_1(F,T)=H^1_{\alg}(\mathcal E_1(K/F),T(\mathbb A_K)/T(K))$,  
\end{itemize}
 all three being independent of the choice of $K$, up to canonical
isomorphism. 

The group $B_3(F,T)$ was denoted simply by $B(F,T)$ previously
in this introduction. The group $B_2(F,T)$ is canonically isomorphic to the
direct sum of all the local groups $B(F_u,T)$. The group $B_1(F,T)$ is more
interesting. It is defined using algebraic $1$-cocycles of the group
$\mathcal E_1(K/F)$ in $T(\mathbb A_K)/T(K)$, and $\mathcal E_1(K/F)$ is
the usual Weil group associated to $K/F$. Since $X_1=\mathbb Z$, the
rightmost vertical isomorphism is telling us that 
$B_1(F,T)$ is canonically isomorphic to the Galois coinvariants on
$X_*(T)$, in perfect analogy to what happens in the local case. So these
Galois coinvariants measure the failure of the total localization map for
$T$ to be surjective. This is just a special case of  Proposition
\ref{prop.CokGlobalG}, but with the new feature  that
these Galois coinvariants can also be interpreted as  a group $B_1(F,T)$
built using a suitable notion of algebraic $1$-cocycles for global Weil
groups. For all this see  section \ref{sec.GTN} and the sections preceding
it. That all three groups $B_i(F,T)$ are independent of the choice of $K$
splitting $T$ is established in section \ref{sec.InflationLocGlob}. 

In subsection \ref{sub.ResAndCor} we define corestriction maps for tori.
Let $E/F$ be a finite separable extension of global
fields, and let $T$ be an $F$-torus. Then for $i=1,2,3$ we define a
corestriction map 
$\Cor:B_i(E,T) \to B_i(F,T)$. 
Now let $K/E$ be a finite
extension such that $K/F$ is Galois and $T$ is split by $K$.  Put
$Y_i(K):=X_*(T) \otimes X_i(K)$.  
Then, for $i=1,2,3$  there is a commutative diagram  
\begin{equation*}
\begin{CD}
Y_i(K)_{G(K/E)} @>{\simeq}>> B_i(E,T) @>>> Y_i(K)^{G(K/E)}
\\ @VVV @V{\Cor}VV @VVV \\
Y_i(K)_{G(K/F)} @>{\simeq}>> B_i(F,T) @>>> Y_i(K)^{G(K/F)}
\end{CD}
\end{equation*}
The left vertical arrow is induced by the identity map on $Y_i(K)$. The
middle vertical arrow is corestriction for $E/F$. The right vertical arrow
is given by $y \mapsto \sum_{\sigma \in G(K/F)/G(K/E)}
\sigma(y)$. 
Similarly, 
for $i=1,2,3$  there is a commutative diagram 
\begin{equation*}
\begin{CD}
Y_i(K)_{G(K/E)} @>{\simeq}>> B_i(E,T) @>>> Y_i(K)^{G(K/E)}
\\ @AAA @A{\Res}AA @AAA \\
Y_i(K)_{G(K/F)} @>{\simeq}>> B_i(F,T) @>>> Y_i(K)^{G(K/F)}
\end{CD}
\end{equation*} 
The left vertical arrow is  given
by $y \mapsto \sum_{\sigma \in G(K/E)\backslash G(K/F)}
\sigma(y)$. 
 The middle vertical arrow is restriction for $E/F$. The right
vertical arrow is induced by the identity map on $Y_i(K)$. There are two
analogous commutative diagrams in the local case (see Lemma 
\ref{lem.LocalCorRes}).

 \subsection{Comments on notation} 
In this introduction we have consistently used $G(K/F)$ to denote the Galois 
group of $K/F$, and $G$ to denote a linear algebraic $F$-group. In the body of 
the text we often, but not always,  follow the same conventions. In parts of the text 
in which an abstract finite group is being considered, it is usually denoted by $G$. 
In parts where a single Galois extension $K/F$ is in play, and a general linear algebraic 
group is not, we sometimes abbreviate $G(K/F)$ to $G$. Such conventions are 
spelled out at the beginning of sections or subsections, as appropriate. 

For any global field $F$ we denote by $V_F$ the set of places of $F$. 
When $K/F$ is an extension of global fields, we  typically denote places of $K$ by $v$, and 
places of $F$ by $u$. When $v \in V_K$ lies over $u \in V_F$, the local Galois group 
$G(K_v/F_u)$ can be identified with the stabilizer $G_v$ of $v$ in $G=G(K/F)$. 
When using the abbreviation $G=G(K/F)$, we often employ the notation $G_v$ 
rather than $G(K_v/F_u)$. 

We consistently write $\bar F$ for a separable closure of a given field $F$, and $\Gamma$ 
for the absolute Galois group $\Gal(\bar F/F)$. 

\section{The set $H^1_{\alg}(\mathcal E, G(K))$}\label{sec.GalGerb}

 \subsection{Goal of this section} Langlands-Rapoport \cite{LR} found a
convenient way to make concrete the notion of a gerb over a field $F$. In
this section we begin by reviewing their definition of Galois gerb, but
using slightly different conventions: 
\begin{itemize}
\item We do not require that our base field $F$ have characteristic $0$. 
\item We  only work with  Galois gerbs $\mathcal E$ for finite Galois
extensions
$K/F$. 
\item We require that $\mathcal E$ be bound by a 
group $D$ of multiplicative type  over
$F$. We do not assume that  $D$ is of finite type over $F$. 
\end{itemize} 
Once the definition of Galois gerb $\mathcal E$ has been reviewed, we
introduce, for any linear algebraic group $G$ over $F$,   the set
$H^1_{\alg}(\mathcal E,G(K))$ of equivalence classes of algebraic
$1$-cocycles. We then go on to develop some
 basic constructions involving $H^1_{\alg}(\mathcal E,G(K))$.

\subsection{Review of Galois gerbs $\mathcal E$ for $K/F$} \label{sub.RvGaGe}
Let $K$ be a finite Galois extension of $F$. We write $G(K/F)$ for the
Galois group of $K/F$.  Let $X$ be a
$G(K/F)$-module, and let $D$ denote the group of multiplicative type having
$X$ as its module of characters. (Over $K$ the ring of regular functions on
$D$ is the group algebra $K[X]$, and over $F$ it is $K[X]^{G(K/F)}$.) We
are going to consider gerbs bound by $D$.

For the  purposes of this note, a Galois gerb for $K/F$, bound by $D$, 
is an extension of groups 
\begin{equation}\label{eq.GalGerb}
1 \to D(K) \to \mathcal E \to G(K/F) \to 1. 
\end{equation} 
(The corresponding Tannakian category over $F$ is then equipped with a fiber
functor over $K$, but we are not going to pursue the Tannakian point of
view.) Associated to such an extension of groups is a class  
$\alpha \in H^2(G(K/F),D(K))$. We will typically denote an element in
$\mathcal E$ by the letter $w$. (Using $e$ for this purpose might be
confusing, because it is often used to denote the identity element in a
group. Besides, when
$F$ is a local field, the main Galois gerb of interest is the Weil group
$W_{K/F}$, and it is natural to denote its elements  by $w$.) 

\subsection{Algebraic $1$-cocycles}  
 Let $G$ be a linear algebraic group over $F$ (i.e.~a smooth 
affine group scheme of finite type over $F$).  The Galois group
$G(K/F)$ acts on
$G(K)$. The canonical surjection $\mathcal E \twoheadrightarrow
G(K/F)$ then yields an action of $\mathcal E$ on $G(K)$, with the
subgroup $D(K)$ acting trivially. So we can consider the set 
$Z^1(\mathcal E,G(K))$ of abstract $1$-cocycles of $\mathcal E$ in $G(K)$.
Such a
$1$-cocycle $x$ is a map $w \mapsto x_w$ from $\mathcal E$ to $G(K)$
satisfying the cocycle condition 
\[
x_{w_1w_2}=x_{w_1} w_1(x_{w_2}).
\]
We need to take notice of two simple consequences of the cocycle condition. 
The first is that 
\begin{itemize}
\item $d \mapsto x_d$ is a homomorphism $\nu_0$ from $D(K)$ to $G(K)$. 
\end{itemize} 
There is a natural action of $G(K/F)$ on the set of homomorphisms
$\nu_1:D(K) \to G(K)$, defined by
$\sigma(\nu_1)(d):=\sigma(\nu_1(\sigma^{-1}(d)))$. 
The second simple consequence of the cocycle condition is that 
\begin{itemize}
\item $\Int(x_w) \circ \sigma(\nu_0)=\nu_0$ 
whenever $w \in \mathcal E$ maps to $\sigma \in G(K/F)$. 
\end{itemize}
 (For an element $x$ in a
group, we denote by $\Int(x)$ the inner automorphism of that group defined
by $g \mapsto xgx^{-1}$.)

An \emph{algebraic $1$-cocycle} of $\mathcal E$ in $G(K)$ is a pair
$(\nu,x)$ consisting of 
\begin{itemize}
\item a homomorphism $\nu:D \to G$ over  $K$, and 
\item an abstract $1$-cocycle $x$ of $\mathcal E$ in $G(K)$, 
\end{itemize}
satisfying the following two compatibilities: 
\begin{itemize}
\item $x_d=\nu(d)$ for all $d \in D(K)$, 
\item $\Int(x_w) \circ \sigma(\nu)=\nu$ 
whenever $w \in \mathcal E$ maps to $\sigma \in G(K/F)$. 
\end{itemize} 

\subsection{The pointed set $H^1_{\alg}(\mathcal
E,G(K))$}\label{sub.DefH1Alg} We write $Z^1_{\alg}(\mathcal E,G(K))$ for the
set of algebraic
$1$-cocycles of $\mathcal E$ in $G(K)$.  There is an obvious action of $G(K)$ on the set of
algebraic $1$-cocycles: $g \in G(K)$ transforms an algebraic $1$-cocycle
$(\nu,x)$ into $(\Int(g) \circ \nu, w \mapsto gx_w w(g)^{-1})$.  
We write $H^1_{\alg}(\mathcal E,G(K))$ for the pointed set
obtained as the quotient of
$Z^1_{\alg}(\mathcal E,G(K))$ by the action of $G(K)$.  The basepoint is of
course  the class of   the pair consisting of the trivial homomorphism and
the trivial abstract $1$-cocycle.

The map $(\nu,x) \mapsto \nu$ induces a well-defined map 
\begin{equation}\label{eq.NuCpt}
H^1_{\alg}(\mathcal E,G(K)) \to \big(\Hom_K(D,G)/\Int(G(K))\bigr)^{G(K/F)}, 
\end{equation}
which we call a \emph{Newton map}. 

\subsection{The situation when $G$ is a torus}\label{sub.Torus}
 Suppose that $G$ is a torus 
$T$. Then the second compatibility in the definition of algebraic
$1$-cocycle just says that $\nu$ is defined over $F$. Moreover it is easy to
check that the  commutative square  
\begin{equation}
\begin{CD}
H^1_{\alg}(\mathcal E,T(K)) @>>> \Hom_F(D,T) \\
@VVV @VVV \\ 
H^1(\mathcal E,T(K)) @>>> \Hom(D(K),T(K))
\end{CD}
\end{equation}
is  cartesian. So, for a torus $T$, we could equally well have defined 
$H^1_{\alg}(\mathcal E,T(K))$ as a fiber product. This observation
will become relevant later, when we are working with  $T(\mathbb
A_K)$ and
$T(\mathbb A_K)/T(K)$ instead of $T(K)$.

\subsection{The $F$-group $J_b$} \label{sub.J_b}
For any $K$-homomorphism $\nu:D \to G$  
we denote by $G_\nu$ the $K$-group obtained as the centralizer in $G$ of 
$\nu$. Let $b=(\nu,x)$ be an algebraic $1$-cocycle of $\mathcal E$ in
$G(K)$. Let $\sigma \in G(K/F)$ and choose $w \in \mathcal E$ such that $w
\mapsto \sigma$. Then the restriction of $\Int(x_w)$ is a $K$-isomorphism,
call it $f_\sigma$, from $\sigma(G_\nu)=G_{\sigma(\nu)}$ to $G_\nu$.
Moreover $f_\sigma$ is independent of the choice of lifting $w$, and the
family of isomorphisms $(f_\sigma)_{\sigma \in G(K/F)}$ is descent data for
$K/F$. This descent data produces an $F$-form, call it $J_b$, of
the $K$-group $G_\nu$. The action $\sigma_{J_b}$ of $\sigma \in G(K/F)$ on
$h \in J_b(K)=G_\nu(K)$ is given by $\sigma_{J_b}(h)=x_w \sigma(h)
x_w^{-1}$ (for any lift $w$ of $\sigma$). It follows that the group
$J_b(F)$ coincides with the stabilizer of $b$ in $G(K)$. 

It is a tautology that the $K$-homomorphism $\nu$ factors through the center
of $G_{\nu}$. Moreover, $\nu$ is defined over $F$ for the $F$-form $J_b$ of
$G_\nu$. Therefore we may view $\nu$ as a central $F$-homomorphism 
$
\nu:D \to J_b. 
$

\subsection{Algebraic $1$-cocycles with a given first component $\nu$}
\label{sub.1rstCpt}
 Let
us fix a $K$-homomorphism  $\nu:D \to G$. There may or may not be an
algebraic $1$-cocycle having $\nu$ as its first component, but let us
suppose that there does exist such a $1$-cocycle $b=(\nu,x)$. It is then
easy to see that $j \mapsto (\nu,jx)$ is a bijection from the set of
$1$-cocycles $j$ of $G(K/F)$ in $J_b(K)$ to the set of
algebraic $1$-cocycles having $\nu$ as their first component. In this way
we obtain a bijection from $H^1(G(K/F),J_b(K))$ to  the fiber of 
\eqref{eq.NuCpt} through the class of $b$. 

In the special case that $\nu$ is trivial, we may take $b$ to be trivial
as well. Then $J_b=G$ and we obtain a canonical injection 
\[
 H^1(G(K/F),G(K)) \hookrightarrow H^1_{\alg}(\mathcal E,G(K)) 
\]
whose image consists of classes $b$ whose image under the Newton map is
trivial. So $H^1_{\alg}(\mathcal E,G(K))$ can be viewed as an enlargement of 
$H^1(G(K/F),G(K))$.

\subsection{Changing the band $D$} \label{sub.ChBnd} 
 Now let us consider two Galois gerbs
\[
1 \to D_i(K) \to \mathcal E_i \to G(K/F) \to 1
\] 
(for $i=1,2$) bound by two groups $D_1$, $D_2$ of multiplicative type over
$F$ (and coming from $G(K/F)$-modules $X_1$, $X_2$ respectively). Suppose
further that we are given a homomorphism $\phi:D_1 \to D_2$ and a
homomorphism $\eta:\mathcal E_1 \to \mathcal E_2$ making 
\[
\begin{CD}
1 @>>> D_1(K) @>>> \mathcal E_1 @>>> G(K/F) @>>> 1 \\ 
@. @V{\phi}VV @V{\eta}VV @| @. \\
1 @>>> D_2(K) @>>> \mathcal E_2 @>>> G(K/F) @>>> 1
\end{CD}
\] 
commute. Then there is a natural map 
\begin{equation}\label{eq.NewBnd}
\eta^*:H^1_{\alg}(\mathcal E_2,G(K)) \to H^1_{\alg}(\mathcal E_1,G(K))
\end{equation} 
for any $G$, induced by the cocycle-level map sending $(\nu,x)$ to
$(\nu\circ\phi, x \circ \eta)$.  

\subsection{Isomorphisms of Galois gerbs for $K/F$} 
Let us consider two Galois gerbs for $K/F$, both bound by $D$: 
\[
1 \to D(K) \to \mathcal E' \to G(K/F) \to 1
\] 
\[
1 \to D(K) \to \mathcal E \to G(K/F) \to 1
\] 
An isomorphism from the first to the second is an isomorphism
$\eta:\mathcal E' \to \mathcal E$ making the diagram 
\begin{equation}\label{CD.eta}
\begin{CD}
1 @>>> D(K) @>>> \mathcal E' @>>> G(K/F) @>>> 1 \\ 
@. @| @V{\eta}VV @| @. \\
1 @>>> D(K) @>>> \mathcal E @>>> G(K/F) @>>> 1 
\end{CD}
\end{equation} 
commute. Such an isomorphism   exists if and only if the 
associated classes
$\alpha,\alpha' \in  H^2(G(K/F),D(K))$ for $\mathcal E$,
$\mathcal E'$ are equal.

The map defined in the previous subsection is then an isomorphism 
\[
\eta^*:H^1_{\alg}(\mathcal E,G(K)) \to H^1_{\alg}(\mathcal E',G(K)). 
\] 
If we assume that the group $H^1(G(K/F),D(K))$ vanishes, then the isomorphism
$\eta^*$ is independent of the choice of
$\eta$. It is for this reason that we will need to prove quite a number of
such vanishing theorems.  

\subsection{Changing the Galois extension $K/F$} \label{sub.ChGalExt} 
Consider a Galois gerb 
\begin{equation}\label{eq.GlGb} 
1 \to D(K) \to \mathcal E \to G(K/F) \to 1
\end{equation}
for $K/F$,  
and suppose that we are given another finite Galois extension $K'/F'$, as
well as embeddings $F \hookrightarrow F'$ and $K\hookrightarrow K'$. We do
not assume that the extensions $F'/F$ and $K'/K$ are algebraic, but we do
assume that the square 
\begin{equation}
\begin{CD}
K @>>> K' \\
@AAA @AAA \\
F @>>> F'
\end{CD}
\end{equation}
commutes. There is then a canonical homomorphism $\rho:G(K'/F') \to G(K/F)$.
First pulling back \eqref{eq.GlGb} along $\rho$, and then pushing forward
along the inclusion  $D(K) \hookrightarrow D(K')$,  we obtain 
the following commutative
diagram (with exact rows): 
\[
\begin{CD}
1 @>>> D(K) @>>> \mathcal E @>>> G(K/F) @>>> 1\\
@. @| @AAA @A{\rho}AA @.\\
1 @>>> D(K) @>>> \mathcal E'' @>>> G(K'/F') @>>> 1\\ 
@. @VVV @VVV @| @. \\
1 @>>> D(K') @>>> \mathcal E' @>>> G(K'/F') @>>> 1.
\end{CD}
\] 
The
homomorphism $\mathcal E'' \to \mathcal E'$ is injective. Using it to
identify $\mathcal E''$ with a subgroup of $\mathcal E'$, we then have
$\mathcal E'=D(K')\mathcal E''$ and $D(K') \cap \mathcal E''=D(K)$. 

For any linear algebraic group $G$ over $F$ there is a natural 
map 
\begin{equation} \label{eq.AlgInf}
H^1_{\alg}(\mathcal E,G(K)) \to H^1_{\alg}(\mathcal E',G(K')),
\end{equation} 
induced by the cocycle-level  map sending $(\nu,x)$ to $(\nu,x')$,
where $x'$ is defined as follows. Write $w' \in \mathcal E'$ as a
product $dw''$, with $d \in D(K')$ and $w'' \in \mathcal E''$, and then
define the value of $x'$ on $w'$ to be $\nu(d) x_{w}$, where $w$ denotes
the image of
$w''$ under $\mathcal E'' \twoheadrightarrow \mathcal E$. (It is easy to see
that the product $\nu(d)x_w$ is independent of the choice of decomposition
$w'=dw''$.)

We will use \eqref{eq.AlgInf} in the following three situations. 
\begin{example}\label{Ex.ResMpp}
Let $F'$ be a field between $K$ and $F$, and take $K'=K$. Then
\eqref{eq.AlgInf} yields a restriction map 
\begin{equation}\label{Ex.Res}
\Res:H^1_{\alg}(\mathcal E,G(K)) \to H^1_{\alg}(\mathcal E',G(K)).
\end{equation}
In this case $\mathcal E'$ is simply the preimage of $G(K/F')$ under
$\mathcal E \twoheadrightarrow G(K/F)$.
\end{example}

\begin{example}\label{Ex.AbsInf}
Let $K'$ be a finite Galois extension of $F$ such that $K' \supset K$, and
take
$F'=F$. Then
\eqref{eq.AlgInf} yields an inflation  map 
\begin{equation}\label{Ex.Inf}
\Inf:H^1_{\alg}(\mathcal E,G(K)) \to H^1_{\alg}(\mathcal E',G(K')).
\end{equation}
In this situation we often write $\mathcal E^{\inf}$ instead of $\mathcal
E'$. 
\end{example} 

\begin{example}\label{Ex.AbsLoc} 
Suppose that $K/F$ is a finite Galois extension of global fields. Choose
a place $u$ of $F$ and a place $v$ of $K$ lying over $u$. Take $F',K'$ to be
$F_u$, $K_v$ respectively. The natural map $G(K_v/F_u) \to G(K/F)$
identifies $G(K_v/F_u)$ with the decomposition group of $v$, and 
\eqref{eq.AlgInf} yields a localization  map 
\begin{equation}\label{Ex.Loc}
\Loc:H^1_{\alg}(\mathcal E,G(K)) \to H^1_{\alg}(\mathcal E',G(K_v)).
\end{equation}
\end{example} 

\subsection{Short exact sequences of linear algebraic groups}
\label{sub.SESLAG}

In the remaining subsections of this section   we study the
behavior of $H^1_{\alg}(\mathcal E,G(K))$ with respect to short exact
sequences of linear algebraic groups. 

\begin{lemma}\label{lem.NG'G} 
Let 
$
1 \to N \xrightarrow{i} G' \xrightarrow{p} G \to 1
$
be a short exact sequence of linear algebraic $F$-groups, and assume that
$p:G'(K) \to G(K)$ is surjective. Then 
\[
H^1_{\alg}(\mathcal E,N(K)) \xrightarrow{i} H^1_{\alg}(\mathcal E,G'(K)) 
\xrightarrow{p} H^1_{\alg}(\mathcal E,G(K))
\] 
is an exact sequence of pointed sets, i.e.~the image of $i$ is equal to 
the kernel of $p$. 
\end{lemma}
\begin{proof}
This follows easily from the definitions. 
\end{proof} 

In the next lemma we are going to consider a short exact sequence 
\[
1 \to Z \xrightarrow{i} G' \xrightarrow{p} G \to 1
\] 
of linear algebraic $F$-groups with $Z$ central in $G'$. Of course $Z$ is
necessarily commutative. In this situation there is a natural action of the
abelian group $H^1_{\alg}(\mathcal E,Z(K))$ on the set $H^1_{\alg}(\mathcal
E,G'(K))$: the class of $(\mu,x) \in Z^1_{\alg}(\mathcal E,Z(K))$
transforms the class of $(\nu,y) \in Z^1_{\alg}(\mathcal E,G'(K))$ into the
class of $(\mu\nu,xy) \in Z^1_{\alg}(\mathcal E,G'(K))$. It is clear that
this action preserves the fibers of the map 
\begin{equation}\label{eq.GG'p}
H^1_{\alg}(\mathcal E,G'(K)) 
\xrightarrow{p} H^1_{\alg}(\mathcal E,G(K)).
\end{equation}

\begin{lemma}\label{lem.ZG'G}
If $p:G'(K) \to G(K)$ is surjective, then $H^1_{\alg}(\mathcal E,Z(K))$
acts transitively on the fibers of the map \eqref{eq.GG'p}. 
\end{lemma}
\begin{proof}
Suppose that $b',b'' \in H^1_{\alg}(\mathcal E,G'(K))$ 
have the same image in $H^1_{\alg}(\mathcal E,G(K))$. Because $G'(K) \to
G(K)$ is surjective, we can choose  algebraic $1$-cocycles $(\nu',x')$,
$(\nu'',x'')$ in the classes $b',b''$ in such a way that they have the
same image in
$Z^1_{\alg}(\mathcal E,G(K))$. It is then easy to check that there exists
a unique element $(\mu,z) \in Z^1_{\alg}(\mathcal E,Z(K))$ such that
$\mu\nu'=\nu''$ and $zx'=x''$.  
\end{proof} 

\begin{remark}\label{rem.T90}
The surjectivity of $p:G'(K) \to G(K)$ is automatic when $Z$ is an
$F$-torus split by $K$, because then $H^1(K,Z)$ vanishes by Hilbert's
Theorem 90. 
\end{remark}

\subsection{Lemma on extensions of tori} 
The following result is the most basic special case of 
Prop.~7.1.1 in  SGA 3, Tome II,
Exp. XVII.  

\begin{lemma}\label{lem.TorExt} 
Any extension of a torus by a torus is again a torus. 
\end{lemma} 

This result will be needed in the next subsection. 

\subsection{Stronger results under two additional
hypotheses}\label{sub.StrR}
 
Lemma \ref{lem.ZG'G} is especially useful when the  map
\eqref{eq.GG'p} is surjective, for then we may identify
$H^1_{\alg}(\mathcal E,G(K))$ with the quotient of 
$H^1_{\alg}(\mathcal E,G'(K))$ by the natural action of $H^1_{\alg}(\mathcal
E,Z(K))$. We are now going to prove a  result of
this kind, but only for Galois gerbs
$\mathcal E$ satisfying the following two assumptions (which will hold for  
all the specific local and global Galois gerbs studied later in this
paper). \newline

\noindent {\bf Assumption 1.} The group $D$ is a protorus split by $K$.
Equivalently, the group $X^*(D)$ is a $G(K/F)$-module that is torsion-free
as abelian group. \newline 

\noindent {\bf Assumption 2.} For every short exact sequence $1 \to T_1 \to
T_2 \to T_3 \to 1$ of $F$-tori split by $K$ the natural map 
$
H^1_{\alg}(\mathcal E,T_2(K)) \to H^1_{\alg}(\mathcal E,T_3(K)) 
$
is surjective. 

\begin{proposition}\label{prop.Zact}
Suppose that $\mathcal E$ satisfies the two assumptions above. Let 
\begin{equation}\label{SES.ZG'G}
1 \to Z \xrightarrow{i} G' \xrightarrow{p} G \to 1
\end{equation}  
be a short exact sequence of linear algebraic $F$-groups in which $Z$ is a
torus that splits over $K$ and is central in $G'$. Then the natural map 
\begin{equation}\label{eq.GG'pp}
H^1_{\alg}(\mathcal E,G'(K)) 
\xrightarrow{p} H^1_{\alg}(\mathcal E,G(K)) 
\end{equation}
is surjective. Moreover the map \eqref{eq.GG'pp} induces a bijection 
between $H^1_{\alg}(\mathcal E,G(K))$ and the quotient of
$H^1_{\alg}(\mathcal E,G'(K))$ by the action of $H^1_{\alg}(\mathcal
E,Z(K))$.
\end{proposition}

\begin{proof} 

Let $b \in H^1_{\alg}(\mathcal E,G(K))$ and choose $(\nu,x) \in
Z^1_{\alg}(\mathcal E,G(K))$ representing $b$. So $\nu \in \Hom_K(D,G)$ and
$x \in Z^1(\mathcal E,G(K))$ satisfy 
\begin{enumerate}
\item $x_d=d^\nu$ for all $d \in D(K)$, and 
\item $\Int(x_w) \circ \sigma(\nu) =\nu$ for any $\sigma \in G(K/F)$ and
any $w \in \mathcal E$ such that $w \mapsto \sigma$. 
\end{enumerate} 
Here we are writing $d^\nu$ (instead of $\nu(d)$) for the value of $\nu$ on
$d$.

By virtue of Assumption 1 the image of $D$ in $G$ is a split $K$-torus $T$
in $G$. Pulling back the extension \eqref{SES.ZG'G} along the inclusion $T
\hookrightarrow G$, we obtain an extension 
\begin{equation}\label{SES.ZT'T}
1 \to Z \to T' \to T \to 1
\end{equation}  
and an inclusion $T' \hookrightarrow G'$. By Lemma \ref{lem.TorExt} $T'$ is
a torus, and since $Z$ and $T$ are both $K$-split, so too is $T'$.
Therefore  the exact sequence \eqref{SES.ZT'T} splits (non-canonically),
from which it follows that 
 the sequence 
\[
0 \to \Hom_K(D,Z) \to \Hom_K(D,T') \to \Hom_K(D,T) \to 0
\] 
is exact.  

We conclude that there exists $\nu' \in \Hom_K(D,T')$ such that $\nu'
\mapsto \nu$. We would like to lift our $1$-cocycle $x$ to a $1$-cocycle
$x'$ in $G'(K)$, but there is an obvious obstruction to doing so. The
 map $p:G'(K) \to
G(K)$ is surjective (see Remark \ref{rem.T90}), so we may choose a
$1$-cochain $x'$ of $\mathcal E$ in $G'(K)$ such that $x' \mapsto x$. The
coboundary of $x'$ is then a $2$-cocycle of $\mathcal E$ in $Z(K)$. We can
choose $x'$ to be a $1$-cocycle if and only if the cohomology class of this
$2$-cocycle vanishes. Unfortunately 
the group $H^2(\mathcal E,Z(K))$ is hard to work with, so it would be
helpful if we could  choose 
$x'$ in such a way that its coboundary is inflated from a
$2$-cocycle of $G(K/F)$. Implementing this idea  will be our next task,
but in doing  so we will be led  to enlarge the group $G'$.  

For each
$\sigma
\in G(K/F)$ we choose
$\dot\sigma
\in
\mathcal E$ such that $\dot\sigma \mapsto \sigma$, and then we choose 
$x'_{\dot\sigma} \in G'(K)$ such that $x'_{\dot\sigma} \mapsto
x_{\dot\sigma}$. We claim that $\Int(x'_{\dot\sigma}) \circ \sigma(\nu')$
is independent of the choice of $\dot\sigma$ and $x'_{\dot\sigma}$. Indeed,
suppose we replace $\dot\sigma$ by $\sigma^\flat$. Then
$\sigma^\flat=\dot\sigma d$ for some $d \in D(K)$, and so
$x_{\sigma^\flat}=x_{\dot\sigma}\sigma(d^\nu)
=x_{\dot\sigma}\sigma(d)^{\sigma(\nu)}$. Therefore one particular lifting
of
$x_{\sigma^\flat}$ is 
$x'_{\dot\sigma}\sigma(d)^{\sigma(\nu')}$, and any
such lifting is of the form
$zx'_{\dot\sigma}\sigma(d)^{\sigma(\nu')}$ for some
$z \in Z(K)$. From this it is clear that $\Int(x'_{\dot\sigma}) \circ \sigma(\nu')$
is independent of the choice of $\dot\sigma$ and $x'_{\dot\sigma}$. 

Now $\nu'$ is one lifting of $\nu$ to $G'$, and by item (2) above
$\Int(x'_{\dot\sigma}) \circ \sigma(\nu')$ is another such lifting. So
there exists a unique $\lambda_\sigma \in \Hom_K(D,Z)$ such that 
\[
\Int(x'_{\dot\sigma}) \circ \sigma(\nu')=\nu' +\lambda_\sigma,
\] 
and an easy calculation shows that $\lambda$ is a $1$-cocycle of $G(K/F)$
in $\Hom_K(D,Z)$. 

In order to kill the cohomology class of $\lambda$ we are going to enlarge
$G'$. There is a canonical $F$-embedding 
\begin{equation}\label{eq.ZRZ} 
Z \hookrightarrow Z'',
\end{equation}
where $Z'':=R_{K/F}(Z)$. (Here we are applying Weil restriction of scalars
$R_{K/F}$ to the $K$-torus obtained from $Z$ by extension of scalars from
$F$ to $K$.) Pushing out the extension \eqref{SES.ZG'G} along the
inclusion \eqref{eq.ZRZ}, we obtain a commutative diagram 
\begin{equation}\label{CD.ZG'G''}
\begin{CD}
1 @>>> Z @>i>> G' @>p>> G @>>> 1 \\ 
@. @VVV @VVV @| @.\\
1 @>>> Z'' @>j>> G'' @>q>> G @>>> 1. 
\end{CD}
\end{equation}
 Of course $G''$ is a central extension of $G$ by the torus $Z''$. 

Now the $G(K/F)$-module $\Hom_K(D,Z'')=X^*(D) \otimes X_*(Z'')$ is
coinduced from the $\mathbb Z$-module $X^*(D) \otimes X_*(Z)$, so
$H^1(G(K/F),\Hom_K(D,Z''))$ vanishes, which guarantees that there exists
$\mu \in \Hom_K(D,Z'')$ such that $\lambda_\sigma=\sigma(\mu)-\mu$. It is
then clear that $\nu'':=\nu'-\mu$ is a $K$-homomorphism $D \to G''$
lifting $\nu$ such that 
\begin{equation}\label{eq.nu''}
\Int(x'_{\dot\sigma}) \circ \sigma(\nu'') =\nu'' 
\end{equation}
for all $\sigma \in G(K/F)$. 

Next we define a $1$-cochain $x''$ of $\mathcal E$ in $G''(K)$ by putting 
$x''_{d\dot\sigma}:=d^{\nu''} x'_{\dot\sigma}$. Obviously $x''$ is a lifting  
of  $x$  to $G''(K)$, so there exists a unique
$2$-cocycle $z$ of $\mathcal E$ in $Z''(K)$ such that 
\[
x''_{w_1}x''_{w_2} =z_{w_1,w_2}x''_{w_1w_2}
\]
for all $w_1,w_2 \in \mathcal E$. From \eqref{eq.nu''} it follows easily
that $z$ is inflated from a (unique) $2$-cocycle (still call it $z$) of
$G(K/F)$ in $Z''(K)$. But $H^2(G(K/F),Z''(K))$ vanishes by Shapiro's
lemma, and so there exists a $1$-cochain $y$ of $G(K/F)$ in $Z''(K)$ such
that $x'''_w:=x''_wy_\sigma$ (where $\sigma$ is the image of $w$ in
$G(K/F)$) is a $1$-cocycle of $\mathcal E$ in $G''(K)$ such that
$x'''\mapsto x$. By construction we have 
\begin{itemize}
\item $x'''_d=d^{\nu''}$,  
\item $\Int(x'''_w) \circ \sigma(\nu'')=\nu''$ when $w \mapsto \sigma$, 
\end{itemize}
and so $(\nu'',x''')$ is an algebraic $1$-cocycle of $\mathcal E$ in
$G''(K)$ lifting $(\nu,x)$. 

At this point we have constructed an element $b'' \in H^1_{\alg}(\mathcal
E,G''(K))$ such that $b'' \mapsto b$. What we really want is an element $b'
\in H^1_{\alg}(\mathcal E,G'(K))$ such that $b' \mapsto b$. To show that
$b'$ exists we are going to make use of Assumption 2. 

It is easy to see that the inclusion $Z'' \to G''$ induces a canonical 
isomorphism 
$Z''/Z = G''/G'$. To simplify notation we put $C:=Z''/Z = G''/G'$.
Obviously $C$ is an $F$-torus split by $K$, and there is a short exact
sequence 
\begin{equation}\label{SES.ZZ''C}
1 \to Z \to Z'' \to C \to 1. 
\end{equation}
Moreover \eqref{CD.ZG'G''} can be enlarged to a commutative diagram 
\begin{equation}\label{CD.ZG'G''C}
\begin{CD}
1 @>>> Z @>i>> G' @>p>> G @>>> 1 \\ 
@. @VVV @VVV @| @.\\
1 @>>> Z'' @>j>> G'' @>q>> G @>>> 1 \\
@. @VVV @VVV @. @. \\
@. C @= C @. @. 
\end{CD}
\end{equation}
and this gives rise to another commutative diagram 
\begin{equation}\label{CD.HZG'G''C}
\begin{CD}
 H^1_{\alg}(\mathcal E,Z(K)) @>i>> H^1_{\alg}(\mathcal E,G'(K)) @>p>> 
H^1_{\alg}(\mathcal E,G(K))  \\ 
 @VVV @VVV @| @.\\
 H^1_{\alg}(\mathcal E,Z''(K)) @>j>> H^1_{\alg}(\mathcal E,G''(K)) @>q>> 
H^1_{\alg}(\mathcal E,G(K)) \\
 @VVV @VVV @.  \\
H^1_{\alg}(\mathcal E, C(K)) @= H^1_{\alg}(\mathcal E,C(K)) @.  
\end{CD}
\end{equation}

We have constructed $b''$ such that $q(b'')=b$ and we seek $b'$ such that
$p(b')=b$. We want to apply Lemma \ref{lem.NG'G} to the short exact
sequence 
\[
1 \to G' \to G'' \to C \to 1,
\] 
so we need to check that $G''(K) \to C(K)$ is surjective. For this it
suffices to prove that $Z''(K) \to C(K)$ is surjective, and this follows
from Hilbert's Theorem 90 and the exactness of \eqref{SES.ZZ''C}. 
From Lemma \ref{lem.NG'G} we see that it is enough to produce $b''_1 \in 
H^1_{\alg}(\mathcal E,G''(K))$ such that 
\begin{itemize}
\item $q(b''_1)=q(b'')$, and 
\item $b''_1$ has trivial image in $H^1_{\alg}(\mathcal E,C(K))$. 
\end{itemize}

We now apply Assumption 2 to the short exact sequence \eqref{SES.ZZ''C} of
$F$-tori split by $K$. We conclude that there exists $b_2 \in 
H^1_{\alg}(\mathcal E,Z''(K))$ whose image under 
\[
H^1_{\alg}(\mathcal E,Z''(K)) \to H^1_{\alg}(\mathcal E,C(K))
\] 
is equal to the image of $b''$ under 
\[
H^1_{\alg}(\mathcal E,G''(K)) \to H^1_{\alg}(\mathcal E,C(K)).
\] 
It is then clear that $b_1'':=b_2^{-1}b''$ does the job (see subsection
\ref{sub.SESLAG} for a discussion of the natural action of $
H^1_{\alg}(\mathcal E,Z''(K))$ on $H^1_{\alg}(\mathcal E,G''(K))$). 

This finishes the proof that \eqref{eq.GG'pp} is surjective. To prove the
last statement of the proposition, we need only invoke Lemma
\ref{lem.ZG'G}, 
which applies by Remark \ref{rem.T90}. 
\end{proof}

 \section{Key result in an abstract setting}\label{sec.KRAS}
Throughout this section $G$ is a finite group. In our applications it will
be a Galois group. In this section, however, we work in an abstract
setting, which will allow us to prove the key result Lemma \ref{lem.AbsBT}  
in a way that brings out its simple, general nature.   

\subsection{Notation} 
Let $M$ be a $G$-module. There are then 
Tate cohomology groups $\hat H^r(G,M)$ for all $r \in \mathbb Z$. We are
going to simplify notation by  writing $H^r(G,M)$
instead of $\hat H^r(G,M)$. We write $M^G$ for the $G$-invariants in $M$,
and $M_G$ for the $G$-coinvariants of $M$. We write $N:M \to M$ for the map
$m \mapsto \sum_{g\in G} gm$. The map $N$ factors as 
\[
M \twoheadrightarrow M_G \to M^G \hookrightarrow M.
\]
Thus $N$ gives rise to maps $M \to M^G$, $M_G \to M^G$, and $M_G \to M$,
all of which will be denoted simply by $N$. Recall that $H^0(G,M)$ (resp.
$H^{-1}(G,M)$) is the cokernel (resp. kernel) of $N:M_G \to M^G$. 

Let $A$ and $B$ be abelian groups. We  write $\Hom(A,B)$ for the group 
$\Hom_{\mathbb Z}(A,B)$ of
homomorphisms from $A$ to $B$. When $A$, $B$ are $G$-modules, there is a
natural action of $G$ on $\Hom(A,B)$, given by
$(gf)(a)=g\bigl(f(g^{-1}a)\bigr)$, and 
$\Hom(A,B)^G$ coincides with the group  $\Hom_G(A,B)$ of $G$-maps from
$A$ to $B$. We write $A
\otimes B$ for the group $A \otimes_{\mathbb Z} B$. When $A$ and $B$ are
$G$-modules, there is a natural action of $G$ on $A \otimes B$, given by
$g(a\otimes b)=ga \otimes gb$.

\subsection{The extension $E$} 

We consider an extension 

\[
1 \to  A \to E \to G \to 1
\]
of $G$ by an abelian group $A$. As usual, this extension gives rise to a
cohomology class $\alpha$, which we now review. 

There is a unique
$G$-module structure on $A$ for which $ga=waw^{-1}$ for any $w \in E$ such
that $w
\mapsto g$. Choose a set-theoretic section $s:G \to E$ of the surjection
$E \to G$. Define a $2$-cochain $a_{\sigma,\tau}$ of $G$ in $A$ by the
rule 
\[
s(\sigma)s(\tau)=a_{\sigma,\tau}s(\sigma\tau). 
\] 
Then $a_{\sigma,\tau}$ is a $2$-cocycle whose cohomology class is
independent of the choice of section $s$. Let us denote this cohomology
class by $\alpha$. 

The inflation-restriction sequence for a $G$-module $M$ is simpler than
usual, because our normal subgroup $A$ is abelian, and because $A$ is acting
trivially on $M$. The  sequence boils down to 
\[
0 \to H^1(G,M) \xrightarrow{\inf} H^1(E,M) \xrightarrow{\res} \Hom_G(A,M) 
\xrightarrow{\tran} H^2(G,M).
\]
The homomorphism at the right end of this sequence is the transgression
homomorphism.   In this simple situation it coincides with the composed map 
\[
\Hom_G(A,M)=\Hom(A,M)^G \twoheadrightarrow H^0(G,\Hom(A,M))
\xrightarrow{\alpha\smile}H^2(G,M),
\] 
the cup product being formed using the tautological pairing 
$A \otimes \Hom(A,M)
\to M$. 

\subsection{Definition of $H^1_Y(E,M)$} \label{sub.H1Y}
We now
consider a triple
$(M,Y,\xi)$ consisting of a
$G$-module
$M$, a
$G$-module $Y$, and a $G$-map $\xi:Y \to \Hom(A,M)$.    Taking
$G$-invariants, we obtain $\xi^G:Y^G \to \Hom(A,M)^G=
\Hom_G(A,M)$, and we define 
 $H^1_Y(E,M)$ to be the fiber product of $H^1(E,M)$ and $Y^G$ over
$\Hom_G(A,M)$. 
In other words, we are forming the fiber product square 
\begin{equation}\label{CD.FibSqHY}
\begin{CD}
 H_Y^1(E,M) @>{r}>> Y^G 
\\
 @V{\pi}VV @V{\xi^G}VV  \\
 H^1(E,M) @>{\res}>> \Hom_G(A,M),
\end{CD}
\end{equation} 
in which $r$ and $\pi$ are the two
canonical projections. 
\subsection{Alternative description of $H^1_Y(E,M)$ using cocycles} 
It is sometimes convenient to think in terms of $1$-cocycles when working
with $H^1_Y(E,M)$. For this we now introduce groups $Z^1_Y(E,M)$ and 
$B^1_Y(E,M)$ of $1$-cocycles and $1$-coboundaries respectively. 
By definition, an element in $Z^1(E,M)$ is a pair $(y,m)$ consisting of 
an element $y \in Y^G$ and a $1$-cocycle $m$ of $E$ in $M$ such that the
restriction of $m$ to $A$ (which is a $G$-map from $A$ to $M$) 
coincides with the image of $y$ under $\xi^G:Y^G \to \Hom_G(A,M)$. By
definition,
$B^1_Y(E,M)$ is the subgroup of elements in $Z^1_Y(E,M)$ consisting of
pairs $(0,m)$, where $m$ is a $1$-coboundary for the $E$-module~$M$; since
$A$ acts trivially  on $M$, the restriction of $m$ to $A$ is trivial. 
Writing $[m]$ for the cohomology class of $m$, we see that 
\[
(y,m) \mapsto
(y,[m]) \in Y^G \times_{\Hom_G(A,M)}H^1(E,M)
\]  
induces an isomorphism from
$Z^1_Y(E,M)/B^1_Y(E,M)$ to
$H^1_Y(E,M)$.

\begin{remark}
The group  
denoted   by
$H^1_{\alg}(\mathcal E,T(K))$ in subsection \ref{sub.Torus}   can be
identified with
$H^1_Y(E,M)$, where
$E=\mathcal E$, $Y=\Hom_K(D,T)$ and $M=T(K)$. So the results in this section
apply to $H^1_{\alg}(\mathcal E,T(K))$. The advantage of the
 more general theory being developed now is that it can be used for other
groups
$M$ such as $T(\mathbb A_K)$ and $T(\mathbb A_K)/T(K)$. 
\end{remark}

\subsection{Inflation-restriction sequence for  $H^1_Y(E,M)$} 
\label{sub.InfRes} 

 The fiber square \eqref{CD.FibSqHY} occurs in the middle of   a
bigger diagram
\begin{equation*}
\begin{CD}
0 @>>> H^1(G,M) @>{i}>> H_Y^1(E,M) @>{r}>> Y^G 
@>{t}>> H^2(G,M)\\
@. @| @V{\pi}VV @V{\xi^G}VV @| \\
0 @>>> H^1(G,M) @>{\inf}>> H^1(E,M) @>{\res}>> \Hom_G(A,M) 
@>{\tran}>> H^2(G,M)
\end{CD}
\end{equation*}
that is obtained as follows.   The map
$i$ is the unique homomorphism such that  
\begin{itemize}
\item the left square commutes, and
\item $ri=0$.
\end{itemize}
 The map $t$ is the unique one making the right square commute. 

We already
know that the bottom row is exact. The top row will be referred to as the
\emph{inflation-restriction sequence for $H^1_Y(E,M)$}. It is easily seen to
be exact. 

\subsection{Definition of $ c:Y_G \to H^1_Y(E,M)$}\label{sub.Defnc}
Since $A$ has finite index in $E$, there is a corestriction map 
\[
\cor:\Hom(A,M)=H^1(A,M) \to H^1(E,M). 
\]
\begin{lemma} \label{lem.c0}
 There exists a unique map $c_0:Y \to H^1_Y(E,M)$ such that 
\begin{itemize}
\item $rc_0$ is equal to $N:Y \to Y^G$, and 
\item $\pi c_0$ is equal to the composed map $Y \xrightarrow{\xi} \Hom(A,M)
\xrightarrow{\cor} H^1(E,M)$. 
\end{itemize}
\end{lemma}

\begin{proof} 
Because $H^1_Y(E,M)$ is a fiber product, we just need to check that 
$\res \circ \cor \circ \xi$ coincides with  
$Y \xrightarrow{N} Y^G \xrightarrow{\xi^G}\Hom_G(A,M)$. This is true,
because   the composition 
\[
\Hom(A,M) \xrightarrow{\cor} H^1(E,M) \xrightarrow{\res} \Hom_G(A,M)
\] 
coincides with the norm map 
$
N:\Hom(A,M) \to \Hom_G(A,M)
$ 
 (see  Lemma \ref{lem.CorRes}(1)).
\end{proof}

\begin{lemma}\label{lem.cNew}
The homomorphism $c_0$ in the previous lemma factors through the quotient
$Y_G$ of $Y$.
\end{lemma}

\begin{proof} To show that $c_0$ factors through $Y_G$, it is enough to
show that both $rc_0$ and $\pi c_0$ do so. In the case of $\pi c_0$, this
is because the map $\cor$ factors through the quotient $\Hom(A,M)_G$
of
$\Hom(A,M)$ (see Lemma \ref{lem.CorRes}(2)). In the case of $rc_0$, it is
obvious, because
$rc_0=N$.  
\end{proof}

\begin{definition}\label{def.c} 
We define $c:Y_G \to H^1_Y(E,M)$ to be the unique homomorphism such that 
$c_0$ is equal to the composed map $Y \twoheadrightarrow Y_G
\xrightarrow{c} H^1_Y(E,M)$. 
\end{definition} 
In our next lemma we will see when $c$ is an isomorphism. 

\subsection{A key lemma}
There is a tautological pairing 
$
A  \otimes \Hom(A,M) \to M,
$
which,  combined with our given map $\xi:Y \to \Hom(A,M)$, yields a
pairing 
$
A\otimes Y \to M,
$ 
and this in turn induces cup product pairings $H^i(G,A) \otimes H^j(G,Y) \to
H^{i+j}(G,M)$. 
In particular cup product with $\alpha$ gives maps 
\begin{equation}\label{eq.abTNiso}
H^i(G,Y) \xrightarrow{\alpha\smile} H^{i+2}(G,M).
\end{equation}

\begin{lemma}\label{lem.AbsBT} 
\hfill 
\begin{enumerate} 
\item The diagram 
\begin{equation*}
\begin{CD}
 H^1(G,M) @>{i}>> H_Y^1(E,M)\\
 @A{\alpha\smile}AA  @A{c}AA \\
 H^{-1}(G,Y) @>>> Y_G 
\end{CD}
\end{equation*} 
commutes. 

\item The homomorphism $c:Y_G \to H^1_Y(E,M)$ is an isomorphism if and only
if the map \eqref{eq.abTNiso} 
 is bijective for  $i=-1$ and injective for $i=0$.  
\end{enumerate}
\end{lemma}
 
\begin{proof} 
 (1) is proved by 
reducing to the case in which
$Y$ is
$\Hom(A,M)$ and
$\xi$ is the identity map. Then one needs to show the equality of two
homomorphisms 
\[
H^{-1}(G,\Hom(A,M)) \to H^1(G,M),
\] 
one being cup product with $\alpha$, the other being induced by 
the restriction of $\cor$ to the kernel of $N$ on $\Hom(A,M)$. For this,
see  Lemma \ref{lem.CorCup}. 

 (2) is proved by 
applying  the $5$-lemma to the diagram   
\begin{equation*}
\begin{CD}
0 @>>> H^1(G,M) @>{i}>> H_Y^1(E,M) @>{r}>> Y^G 
@>{t}>> H^2(G,M)\\
@. @A{\alpha\smile}AA  @A{c}AA @| @A{\alpha\smile}AA \\
0 @>>> H^{-1}(G,Y) @>>> Y_G @>N>> Y^G 
@>>> H^0(G,Y) @>>> 0
\end{CD}
\end{equation*} 
The commutativity of the 
right and middle squares  is clear from the definitions of the maps
involved, and the left square was handled in the first part of this lemma.  
\end{proof} 

\subsection{Naturality with respect to $(M,Y,\xi)$} 
In order to form $H^1_Y(E,M)$ we need 
two $G$-modules $M$,
$Y$ and a $G$-map
$\xi:Y \to \Hom(A,M)$.  
There is an obvious naturality with respect to $(M,Y,\xi)$. 
Suppose we are given two such triples $(M_i,Y_i,\xi_i)$ ($i=1,2$). 

A morphism from the first triple to the second is a pair $(f,g)$ of
homomorphisms $f:M_1 \to M_2$, $g:Y_1 \to Y_2$ such that the square 
\begin{equation}
\begin{CD}
Y_1 @>{\xi_1}>> \Hom(A,M_1) \\
@VgVV @V{f}VV \\
Y_2 @>{\xi_2}>> \Hom(A,M_2)
\end{CD}
\end{equation} 
commutes. The right vertical map is $h \mapsto f\circ h$. It can also be
viewed as the map $H^1(A,M_1) \to H^1(A,M_2)$ induced by $f$. 

When we have such a morphism $(f,g)$,  there is an obvious map  
\[
\psi:H^1_{Y_1}(E,M_1) \to H^1_{Y_2}(E,M_2) 
\]
obtained from the vertical arrows in the commutative diagram 
\begin{equation}
\begin{CD}
Y_1^G @>{\xi_1^G}>> \Hom_G(A,M_1) @<{\res}<< H^1(E,M_1)\\
@V{g^G}VV @V{f^G}VV @V{f}VV \\
Y_2^G @>{\xi_2^G}>> \Hom_G(A,M_2) @<{\res}<< H^1(E,M_2)
\end{CD}
\end{equation}

\begin{lemma}\label{lem.Nat1}
The diagram  
 \begin{equation*}
\begin{CD}
 (Y_1)_G @>{c_1}>>     H^1_{Y_1}(E,M_1)\\
 @V{g_G}VV  @V{\psi}VV  \\
(Y_2)_G @>{ c_2}>> H^1_{Y_2}(E,M_2)
\end{CD}
\end{equation*} 
commutes. 
\end{lemma}
\begin{proof}
Unwind the definitions, and then use the naturality of
corestriction with respect to
$M_1 \to M_2$.
\end{proof} 

\subsection{Naturality with respect to $E$}
There is another kind of naturality, this time with respect to the
extension $E$.  Suppose we are given a commutative diagram 
 \begin{equation*}
\begin{CD}
1 @>>> A' @>>> E'     @>>> G @>>> 1 \\
 @. @VhVV  @V{\tilde h}VV @| @. \\
1 @>>> A @>>> E     @>>> G @>>> 1 
\end{CD}
\end{equation*} 
with exact rows. 
In other words we are considering a morphism from the extension in the top 
row to the one in the bottom row.  Let $M$ be a $G$-module. We consider a triple $(M,Y,\xi)$
 of the kind relevant for $E$. So $\xi$ is a $G$-map from $Y$ to
$\Hom(A,M)$. We then obtain a triple $(M,Y,\xi')$ relevant for $E'$ by
setting $\xi'$ equal to the composed map 
\[
Y \xrightarrow{\xi} \Hom(A,M) \xrightarrow{h} \Hom(A',M). 
\]
Moreover, there is an obvious homomorphism 
\[
\psi':H^1_{Y}(E,M) \to H^1_{Y}(E',M)
\]
obtained from the pullback map ${\tilde h}^*:H^1(E,M) \to H^1(E',M)$ together with the
identity map on $Y^G$. 
\begin{lemma}\label{lem.Nat2}
$\psi'$ is an isomorphism and the  
diagram 
 \begin{equation*}
\begin{CD}
 Y_G @>{c}>>     H^1_{Y}(E,M)\\
 @|  @V{\psi'}VV  \\
Y_G @>{ c'}>> H^1_{Y}(E',M)
\end{CD}
\end{equation*} 
commutes. 
\end{lemma}
\begin{proof}
To see that the square commutes, one uses the functoriality of corestriction
with respect to
$(E,A)$ (see Lemma \ref{lem.CorInf}). To see that $\psi'$ is an
isomorphism, one  applies the $5$-lemma to 
\begin{equation*}
\begin{CD}
0 @>>> H^1(G,M) @>{i}>> H_Y^1(E,M) @>{r}>> Y^G 
@>{t}>> H^2(G,M)\\
@. @| @V{\psi'}VV @| @| \\
0 @>>> H^1(G,M) @>{i'}>> H_{Y}^1(E',M) @>{r'}>> Y^G 
@>{t'}>> H^2(G,M)
\end{CD}
\end{equation*}
\end{proof}

\subsection{Naturality with respect to $G$}  \label{sub.NatRespG} 
Now we want to consider a form of naturality in which $G$ is allowed to
vary.  As usual we start with an extension 
\begin{equation}\label{eq.ExtNatG}
1 \to A \to E \to G \to 1
\end{equation}
and a triple $(M,Y,\xi)$.  We may form $H^1_Y(E,M)$. 

Now suppose  we are given a finite group $G'$, a group homomorphism
$\rho:G'
\to G$, a
$G'$-module $A'$, and a $G'$-module map $h:A \to A'$. First pulling back
\eqref{eq.ExtNatG} along $\rho$ and then pushing forward along $h$, we
obtain a commutative diagram 
\begin{equation}
\begin{CD}
1 @>>> A @>>> E @>>> G @>>> 1 \\
@. @| @AAA @A{\rho}AA @. \\
1 @>>> A  @>>> E'' @>>> G' @>>> 1 \\
@. @VhVV @VVV @| @. \\
1 @>>> A' @>>> E' @>>> G' @>>> 1
\end{CD}
\end{equation} 
with exact rows. 

Consider a triple $(Y',M',\xi')$ of the kind relevant to the extension $E'$
in the bottom row of our commutative diagram. We are going to define a
natural homomorphism  
\begin{equation}
\Phi(f,g,h):H^1_Y(E,M) \to H^1_{Y'}(E',M')
\end{equation}
that depends on the map $h:A \to A'$ we already chose as well as on
$G'$-module maps $f:M \to M'$, $g:Y \to Y'$ that we choose now. The map
$\Phi(f,g,h)$ is defined only when the diagram 
\begin{equation}
\begin{CD}
Y @>{\xi}>> \Hom(A,M) \\
@| @VfVV \\
Y @. \Hom(A,M') \\
@VgVV @AhAA \\
Y' @>{\xi'}>> \Hom(A',M') \\
\end{CD}
\end{equation}
commutes. 

The map $\Phi(f,g,h)$ is induced by the cocycle level map $(\nu,x) \mapsto
(g(\nu),x')$, where $x'$ is the unique $1$-cocycle of $E'$ in $M'$ such
that 
\begin{itemize}
\item the restriction of $x'$ to $A'$ is equal to  the map $\xi'(g(\nu)):A'
\to M'$,  and  
\item 
the pullback of $x'$ to $E''$ (via $E'' \to E'$) is equal to   
the $1$-cocycle obtained by applying $f$ 
 to the pullback of
$x$ to $E''$ (via $E'' \to E$). 
\end{itemize}

\begin{example}\label{Ex.AbsRes}
When $G'$ is a subgroup, $\rho$ is the inclusion, and
$(M',Y',\xi')=(M,Y,\xi)$, we obtain a restriction map 
\begin{equation}
\Res:H^1_Y(E,M) \to H^1_Y(E',M), 
\end{equation}
where $E'$ is the preimage of $G'$ in $E$. 
\end{example}

\begin{example}\label{Ex.Inf2}
When $\rho$ is surjective, 
 $\Phi(f,g,h)$ is a very  
general kind of inflation map. Particular examples will arise in a later
section on inflation in the context of local and global Tate-Nakayama
triples. 
\end{example}

\section{Abstract Tate-Nakayama triples for a finite group $G$} \label{sec.absTN}
We again consider a finite group $G$. 
As we will now see,  groups $H^1_Y(E,M)$ of the type studied 
in the previous section arise naturally in  any setting 
 in which one has Tate-Nakayama
isomorphisms. In fact there are four such settings, one local 
and three global. So, in order to avoid much  tiresome
repetition, we need an axiomatic version of Tate-Nakayama theory.

\subsection{Definition of Tate-Nakayama triples}
Let $X$, $A$ be $G$-modules, and let $\alpha \in H^2(G,\Hom(X,A))$.    We
say that $(X,A,\alpha)$ is a
\emph{weak Tate-Nakayama  triple for
$G$} 
if the following condition holds for every subgroup $G'$ of $G$:   
\begin{itemize}
\item For all $r \in \mathbb Z$ cup product with   $\Res_{G/G'}(\alpha)$
induces isomorphisms 
\[ 
H^r(G', X) \to
H^{r+2}(G', A). 
\] 
\end{itemize}
We say that $(X,A,\alpha)$ is \emph{rigid} if 
\begin{itemize}
\item $H^1(G',\Hom(X,A))$ is trivial for every subgroup $G'$ of $G$. 
\end{itemize} 
 Finally, a \emph{Tate-Nakayama triple} is a weak
Tate-Nakayama triple that is also rigid. 

For any weak Tate-Nakayama triple it is  result of Nakayama \cite{N}
(see
\cite[p.~156]{Se} for a textbook reference) that  cup product with $\alpha$
induces isomorphisms 
\begin{equation}\label{eq.NakIso}
H^r(G,M\otimes X) \to
H^{r+2}(G,M\otimes A) 
\end{equation}
for all $r \in \mathbb Z$ and every $G$-module $M$
that is torsion-free as abelian group.

In sections \ref{sec.LTN} and \ref{sec.GTN} we will review the  standard
examples of Tate-Nakayama triples. For the moment our  goal is merely to explain how 
the theory  in the last section applies to Tate-Nakayama
triples.  
Weak Tate-Nakayama triples will make an appearance  only in 
 Appendix \ref{App.A}, where we develop tools to show  that certain 
weak Tate-Nakayama triples are rigid.

\subsection{The extension $\mathcal E$} 
Given a Tate-Nakayama triple $(X,A,\alpha)$  for $G$, we choose  an  extension 
\begin{equation} \label{eq.calE}
1 \to \Hom(X,A) \to \mathcal E \to G \to 1
\end{equation} 
 whose associated class in
$H^2(G,\Hom (X,A))$ is equal to $\alpha$.  Because our triple is 
assumed to be  rigid, 
the group 
$H^1(G,\Hom(X,A))$ vanishes, and so every automorphism of our
extension (by which we mean an automorphism $\theta$ of $\mathcal E$ making 
\[
\begin{CD}
1 @>>> \Hom(X,A) @>>> \mathcal E @>>> G @>>> 1 \\
@. @| @V{\theta}VV @|  @. \\
1 @>>> \Hom(X,A) @>>> \mathcal E @>>> G @>>> 1 
\end{CD}
\]
commute) is the inner automorphism $\Int(x)$ coming from some element $x$
in the subgroup 
$\Hom(X,A)$ of $\mathcal E$. Such automorphisms are harmless for
our purposes, and so $\mathcal E$ is canonical enough. (The
situation is just like that for the Weil group.)

\subsection{Extending $H^{-1}(G,M \otimes X) \simeq H^1(G,M \otimes A)$ to
$(M \otimes X)_G \simeq H^1_{\alg}(\mathcal E,M \otimes A)$} 

Let $(X,A,\alpha)$ be a Tate-Nakayama triple for $G$, and 
choose $\mathcal E$ as above.  
Let $M$ be a $G$-module.  
There is a tautological pairing $X \otimes \Hom(X,A) \to A$. Tensoring this
with $M$, we obtain 
\[
M \otimes X \otimes \Hom(X,A) \to M\otimes A,
\]
adjoint to which is a homomorphism 
\[
\xi:M\otimes X \to \Hom\bigl(\Hom(X,A),M \otimes A\bigr). 
\]
Applying the discussion in  subsection
\ref{sub.H1Y}  to $\mathcal E$ and 
 the triple $(M\otimes A,M\otimes X,\xi)$,  we may form the group 
$H^1_Y(\mathcal E,M\otimes A)$ with $Y:=M\otimes X$. 
The rigidity of $(X,A,\alpha)$ ensures that this group is 
 independent of the choice of $\mathcal E$, up to canonical 
isomorphism.

We no longer need to choose $Y$ and $\xi$ (as was
the case in subsection 
\ref{sub.H1Y}); they are determined by $A$, $M$ and $X$. For
this reason it is now less useful to retain $Y$ in the notation, and we will
often write
$H^1_{\alg}(\mathcal E,M \otimes A)$ in place of $H^1_Y(\mathcal E,M \otimes
A)$.

Our next result makes use of  the canonical homomorphism 
\begin{equation}\label{eq.TNc}
c:Y_G \to H^1_Y(\mathcal E,M\otimes A)
\end{equation} 
of Definition \ref{def.c}. 

\begin{lemma}\label{lem.AbTN2}
Let $M$ be any $G$-module that is torsion-free as abelian
group. The canonical homomorphism 
\eqref{eq.TNc} 
is then an isomorphism. Moreover, the diagram 
\[
\begin{CD}
(M \otimes X)_G @>c>> H^1_{\alg}(\mathcal E,M\otimes A)\\
@AAA @AAA \\
H^{-1}(G,M \otimes X) @>{\alpha\smile}>> H^1(G,M\otimes A)
\end{CD}
\] 
commutes,  the two vertical maps being the canonical injections. 
\end{lemma}
\begin{proof}
 In view of the Nakayama isomorphism \eqref{eq.NakIso}, this follows from
Lemma \ref{lem.AbsBT}. 
\end{proof} 

\subsection{Restriction for a subgroup $G'$ of $G$}\label{sub.ResTN} 
Let $(X,A,\alpha)$ be a Tate-Nakayama triple for $G$, and let $\mathcal E$
be an extension of $G$ by $\Hom(X,A)$ with corresponding cohomology class
$\alpha$.  Let 
$G'$ be a subgroup of $G$, and put $\alpha':=\Res_{G/G'}(\alpha) \in
H^2(G',\Hom(X,A))$. It is evident that
$(X,A,\alpha')$ is a Tate-Nakayama triple for $G'$.    For every $G$-module
$M$ there is a restriction map 
\[
\Res:H^1_{\alg}(\mathcal E,M \otimes A) \to H^1_{\alg}(\mathcal E',M \otimes
A), 
\]
where $\mathcal E'$ denotes the preimage of $G'$ under $\mathcal E
\twoheadrightarrow G$ (see Example \ref{Ex.AbsRes}). 

Define a homomorphism $M \otimes X \to (M \otimes X)_{G'}$ by sending $\mu
\in M \otimes X$ to $\sum_{g \in G'\backslash G} g\mu$. This map factors
through the coinvariants of $G$ on $M \otimes X$, yielding a natural
homomorphism 
\begin{equation}\label{eq.CoAv}
(M \otimes X)_G \to (M \otimes X)_{G'}
\end{equation}

\begin{lemma}\label{lem.Trick}
The square 
\begin{equation}
\begin{CD}
(M \otimes X)_G @>{\eqref{eq.CoAv}}>> (M \otimes X)_{G'} \\
@VcVV @VcVV\\
H^1_{\alg}(\mathcal E,M \otimes A) @>{\Res}>> H^1_{\alg}(\mathcal E',M
\otimes A)
\end{CD}
\end{equation} 
commutes. 
\end{lemma}
\begin{proof}
We enlarge the square to a diagram 
\begin{equation}
\begin{CD}
(M \otimes X)_G @>{\eqref{eq.CoAv}}>> (M \otimes X)_{G'} \\
@VcVV @VcVV\\
H^1_{\alg}(\mathcal E,M \otimes A) @>{\Res}>> H^1_{\alg}(\mathcal E',M
\otimes A)\\
@VrVV @V{r'}VV \\
(M \otimes X)^G @>>> (M \otimes X)^{G'},
\end{CD}
\end{equation} 
the bottom horizontal arrow being the obvious inclusion. The bottom square
and outer rectangle are easily seen to commute. Consequently the top square
commutes whenever $r'$ is injective. This is the case when $M$ is free of
finite rank as $\mathbb Z[G]$-module, because then the kernel $H^1(G',M
\otimes A)$ of $r'$ obviously vanishes. 

Since the square we seek to prove commutative is functorial in $M$, to prove
its commutativity for a given $M$, it is sufficient to prove its
commutativity for any $M'$ that dominates $M$ in the sense that there exists
a $G$-map $M' \to M$ for which $(M' \otimes X)_G \to (M \otimes X)_G$ is
surjective. This can obviously be achieved by using a suitable $M'$ that is 
free of
finite rank as $\mathbb Z[G]$-module. 
\end{proof}

\subsection{Naturality} \label{sub.BetterNat}
We are now going to discuss the naturality of the
construction
$(X,A,\alpha) \mapsto H^1_{\alg}(\mathcal E,M\otimes A)$, and in order to do
so we need a suitable notion of morphism. 

\begin{definition}\label{def.TNmor}
A \emph{morphism $(X_2,A_2,\alpha_2) \to (X_1,A_1,\alpha_1)$ of
Tate-Nakayama triples} is a pair $(b,a)$ of $G$-maps $b:X_2  \to X_1$,
$a:A_2 \to A_1$ such that $a(\alpha_2)=b(\alpha_1) \in
H^2(G,\Hom(X_2,A_1))$. 
\end{definition}

Given such a morphism $(b,a)$, we are now going to define a natural map 
\begin{equation}\label{eq.Nat2}
\rho:H^1_{\alg}(\mathcal E_2,M \otimes A_2) \to H^1_{\alg}(\mathcal E_1,M
\otimes A_1) 
\end{equation}
for any $G$-module $M$ that is torsion-free as abelian group. Here of course
we have chosen extensions 
\[
1 \to \Hom(X_i,A_i) \to \mathcal E_i \to G \to 1 \qquad (i=1,2)
\] 
with associated cohomology classes $\alpha_i \in H^2(G,\Hom(X_i,A_i))$. 
In order to define the map \eqref{eq.Nat2} we begin by choosing an
extension 
\[
1 \to \Hom(X_2,A_1) \to \mathcal F \to G \to 1
\] 
with associated cohomology class $a(\alpha_2)=b(\alpha_1)$, and then
choosing  homomorphisms $\tilde a$, $\tilde b$ making the diagram 
\begin{equation}
\begin{CD}
1 @>>> \Hom(X_2,A_2) @>>> \mathcal E_2 @>>> G @>>> 1\\
@. @VaVV @V{\tilde a}VV @| @.\\
1 @>>> \Hom(X_2,A_1) @>>> \mathcal F @>>> G @>>> 1\\
@. @AbAA @A{\tilde b}AA @| @.\\
1 @>>> \Hom(X_1,A_1) @>>> \mathcal E_1 @>>> G @>>> 1\\
\end{CD}
\end{equation}
commute. We have not assumed that $H^1(G,\Hom(X_2,A_1))$ vanishes, so
$\tilde a$, $\tilde b$ are far from unique. Fortunately, however, the
map \eqref{eq.Nat2} we are going to define will turn out to be independent
of the choice of
$\mathcal F$, $\tilde a$, $\tilde b$. 

There is commutative diagram 
\begin{equation*}
\begin{CD}
X_2 \otimes \Hom(X_2,A_2) @>>> A_2 \\
@| @VaVV \\
X_2 \otimes \Hom(X_2,A_2) @>>> A_1 \\
@V{\id \otimes a}VV @| \\
X_2 \otimes \Hom(X_2,A_1) @>>> A_1 \\
@A{\id \otimes b}AA @| \\
X_2 \otimes \Hom(X_1,A_1) @>>> A_1 \\
@V{b \otimes \id}VV @| \\
X_1 \otimes \Hom(X_1,A_1) @>>> A_1  
\end{CD}
\end{equation*}
in which the top, middle and bottom pairings are the tautological ones, and
the other two are the unique ones making the diagram commute. 
Put $Y_i:=M \otimes X_i$. 
Tensoring the
entire diagram with $M$, and then applying adjointness of $\otimes$ and
$\Hom$, we obtain another commutative diagram 
\begin{equation*}
\begin{CD}
Y_2 @>>> \Hom(\Hom(X_2,A_2),M\otimes A_2)  
\\ @| @VVV \\
Y_2 @>>> \Hom(\Hom(X_2,A_2),M\otimes A_1) \\
@| @AAA \\
Y_2 @>>> \Hom(\Hom(X_2,A_1),M\otimes A_1)  \\
@| @VVV \\
Y_2 @>>> \Hom(\Hom(X_1,A_1),M\otimes A_1)  \\
@V{\id_M \otimes b}VV @|  \\
Y_1 @>>> \Hom(\Hom(X_1,A_1),M\otimes A_1).  
\end{CD}
\end{equation*}
From Lemmas \ref{lem.Nat1} and \ref{lem.Nat2} we obtain a commutative
diagram 
\begin{equation*}
\begin{CD}
(Y_2)_G @>c>>  H^1_{Y_2}(\mathcal E_2,M \otimes A_2) 
\\ @|  @VVV\\
(Y_2)_G @>c>>  H^1_{Y_2}(\mathcal E_2,M \otimes A_1)\\
@|  @A{\simeq}AA\\
(Y_2)_G @>c>>  H^1_{Y_2}(\mathcal F,M \otimes A_1) \\
@|  @V{\simeq}VV\\
(Y_2)_G @>c>>  H^1_{Y_2}(\mathcal E_1,M \otimes A_1) \\
@V{\id_M \otimes b}VV  @VVV \\
(Y_1)_G @>c>>  H^1_{Y_1}(\mathcal E_1,M \otimes A_1), 
\end{CD}
\end{equation*}
where the right vertical arrows are as follows (starting from the top): 
\begin{enumerate}
\item $a:H^1_{Y_2}(\mathcal E_2,M \otimes A_2) \to H^1_{Y_2}(\mathcal E_2,M
\otimes A_1)$, induced by the $G$-map $\id_M \otimes a:M \otimes A_2 \to
M\otimes A_1$,  
\item $ {\tilde a}^*: 
H^1_{Y_2}(\mathcal F,M \otimes A_1) 
\to 
H^1_{Y_2}(\mathcal E_2,M \otimes A_1) 
$, 
\item $ {\tilde b}^*:  H^1_{Y_2}(\mathcal F,M \otimes A_1)  \to 
H^1_{Y_2}(\mathcal E_1,M \otimes A_1)$, 
\item $b:H^1_{Y_2}(\mathcal E_1,M \otimes A_1) \to H^1_{Y_1}(\mathcal E_1,M
\otimes A_1)$, induced by  $\id_M \otimes b:Y_2 \to Y_1$. 
\end{enumerate}
Remembering that $H^1_{\alg}(\mathcal E_i,M\otimes A_i)$ stands for 
$H^1_{Y_i}(\mathcal E_i,M\otimes A_i)$, we see that the composition of the
right vertical arrows yields a homomorphism  
\begin{equation}
 \rho:H^1_{\alg}(\mathcal E_2,M \otimes A_2) \to
H^1_{\alg}(\mathcal E_1,M \otimes A_1)
\end{equation}
making the square 
\begin{equation}
\begin{CD}
(Y_2)_G @>c>>  H^1_{\alg}(\mathcal E_2,M \otimes A_2)  \\ 
@V{\id_M \otimes b}VV  @V{\rho}VV\\
(Y_1)_G @>c>>  H^1_{\alg}(\mathcal E_1,M \otimes A_1)
\end{CD}
\end{equation} 
commute. 
Since the horizontal maps $c$ are isomorphisms, we conclude that $\rho$ is
independent of the choice of $\mathcal F$, $\tilde a$, $\tilde b$, and it
is $\rho$ that we take as the  map \eqref{eq.Nat2} we wished to define.

\subsection{Preview} 
The standard situations in which there are Tate-Nakayama isomorphisms are
all associated with Tate-Nakayama triples.   
There is a canonical Tate-Nakayama triple associated with every finite
Galois extension of local fields. It will be discussed in section
\ref{sec.LTN}.  
Associated to every finite Galois  extension $K/F$ of global fields (and a
suitable set $S$ of places of $F$) are three 
Tate-Nakayama triples, and there are canonical morphisms from the third to
second, and from the second to the first. All this, together with a
localization map (global to adelic), will be discussed in section
\ref{sec.GTN}.

\section{Local Tate-Nakayama triples} \label{sec.LTN}
\subsection{Notation} In this section we consider a finite Galois extension
$K/F$ of local fields, whose Galois group we denote by $G$. It is part of
local class field theory that we get a Tate-Nakayama triple $(X,A,\alpha)$
for $G$ by taking 
\begin{itemize}
\item $X$ to be $\mathbb
Z$, with $G$ acting trivially, 
\item $A$ to be $K^\times$, with the natural $G$-action. 
\item $\alpha$ to be the fundamental class in $H^2(G,K^\times)$. 
\end{itemize} 
Observe that the group $\mathcal E$ occurring in the
extension \eqref{eq.calE} is a Weil group for
$K/F$. 

\subsection{The group $H^1_{\alg}(\mathcal E,T(K))$}\label{sub.TNlc}

Let $T$ be a torus over $F$ that splits over $K$. Its group $X_*(T)$
of cocharacters is a $G$-module that is finite free as $\mathbb Z$-module.
Take $M=X_*(T)$ in our abstract theory. We then obtain from Lemma
\ref{lem.AbTN2} the following result,  the nonarchimedean case of which 
gives a slightly different perspective on  the results in 
\cite[\S 8]{IsoII}. 
In the lemma we write $H^1_{\alg}(\mathcal E,T(K))$
in place  of $H^1_Y(\mathcal E,T(K))$, where  $Y$ is    
$M\otimes X=X_*(T)$. 

\begin{lemma}\label{lem.LocTN}
 The diagram 
\[
\begin{CD}
X_*(T)_G @>c>> H^1_{\alg}(\mathcal E,T(K))\\
@AAA @AAA \\
H^{-1}(G,X_*(T)) @>{\alpha\smile}>> H^1(G,T(K))
\end{CD}
\] 
commutes,  the two vertical maps being the canonical injections. Moreover,
the two horizontal maps are isomorphisms, the bottom one being one of the 
Tate-Nakayama isomorphisms.
\end{lemma}
\begin{proof}
We just need to notice that $M\otimes A$ works out to $T(K)$. 
\end{proof}

 \section{Global Tate-Nakayama triples} \label{sec.GTN}

\subsection{Notation} \label{sub.TNnot}
For any global field $F$ we write $V_F$ for the set of places of $F$. For
any finite extension $E/F$ of global fields there is a natural map
$f_{E/F}:V_E \to V_F$ sending a place of $E$ to the unique place of $F$
below it.

In this section
$K/F$ is a finite Galois extension of global fields, with Galois group~$G$. 
For any set $S$ of places of $F$ and any finite extension $E/F$, we write
$S_E$ for the preimage of
$S$ under $f_{E/F}$.

In this section we work with a set $S$  of places of $F$ satisfying
the following conditions:
\begin{itemize}
\item $S$ contains all archimedean places. 
\item $S$ contains all finite places that ramify in $K$.
\item For every intermediate field $E$ of $K/F$,  every ideal class of $E$
contains an ideal with support in $S_E$.  
\end{itemize}

\subsection{Three global Tate-Nakayama triples} \label{sub.GlobTNT}
We need to recall the  constructions  Tate \cite{T} used to prove global 
Tate-Nakayama isomorphisms for tori over $F$. We will see that they yield
Tate-Nakayama triples $(X_i,A_i,\alpha_i)$ (for $i=1,2,3$) as well as
morphisms 
\[
(X_3,A_3,\alpha_3) \to (X_2,A_2,\alpha_2) \to (X_1,A_1,\alpha_1).
\]

 Tate considers two  short exact
sequences of
$G$-modules. The first is 
\begin{equation}\label{eq.TateA}
1 \to A_3 \xrightarrow{a'} A_2 \xrightarrow{a}
A_1\to 1, \tag{A}
\end{equation} 
where 
\begin{itemize}
\item 
$A_3$ is the group of $S_K$-units in $K^\times$, that is, elements of
$K^\times$ that are units at all places not in $S_K$. 
 \item $A_2$ is the group  of $S_K$-ideles of $K$,
that is, ideles whose $v$-component is a unit for each place $v$ not in
$S_K$.  
\item $A_1=A_2/A_3$ is the group of $S_K$-idele classes
of $K$.
\end{itemize} 
Our third assumption on $S$ tells us 
that the inclusion of $A_2$ in $\mathbb A_K^\times$ induces an isomorphism 
$A_2/A_3 \simeq \mathbb A_K^\times/K^\times$, so $A_1$ is in fact the
group of idele classes of $K$.  

\begin{lemma}\label{lem.H1Van}
Let $G'$ be any subgroup of $G$.  Then
$H^1(G',A_1)$, $H^1(G',A_2)$ and $H^1(G',A_3)$  vanish, and the sequence 
\begin{equation}
1 \to A_3^{G'} \to A_2^{G'} \to A_1^{G'}  \to 1
\end{equation}
is short exact. 
\end{lemma}
\begin{proof} Let $E$ be the fixed field of $G'$ on
$K$. 
The vanishing of $H^1(G',A_1)$ is part of global class field theory for
$K/E$. The vanishing of $H^1(G',A_2)$ follows from Hilbert's theorem 90 for
local fields, together with the vanishing of the first Galois cohomology of 
$\mathbb G_m$ and $\mathbb G_a$ over  finite fields. See Tate's article for
more details. 
 Since $H^1(G',A_2)$ vanishes, there is an exact sequence 
\begin{equation}
1 \to A_3^{G'} \to A_2^{G'} \to A_1^{G'}  \to H^1(G',A_3) \to 1.
\end{equation}
Now $A_1^{G'}$ is the idele class group for $E$, and $A_2^{G'}$ is the
group of $S_E$-ideles of $E$. So  our third assumption on $S$
tells us that $A_2^{G'} \to A_1^{G'}$ is surjective, and hence that
$H^1(G',A_3)$ vanishes. 
\end{proof}

 The second short exact sequence considered by Tate is 
\begin{equation}
0 \to X_3 \xrightarrow{b'} X_2 \xrightarrow{b} X_1 \to 0, \tag{X}
\end{equation}\label{eq.TateX} 
where \begin{itemize} 
\item $X_1$ is the group of integers, with $G$ acting trivially. 
\item $X_2$ is the free abelian group on the set $S_K$, the
$G$-action on $X_2$ being induced by the natural $G$-action on $S_K$. (Thus
$X_2=\mathbb Z[S_K]$). 
\item The homomorphism $b$ maps $\sum_{v \in S_K} n_v v \in X_2$ to
$\sum_{v
\in S_K} n_v$. 
\item $X_3$ is the kernel of $b$, and $b'$ the canonical inclusion. (Thus
$X_3=\mathbb Z[S_K]_0$).
\end{itemize}

Next Tate constructs a commutative diagram 
\begin{equation}\label{CD.Tate}
\begin{CD}
\dots @>>>  H^r(G,X_3) @>>>  H^r(G,X_2) @>>>  H^r(G, X_1) @>>>
\dots 
\\ @. @V{\alpha_3^r}VV @V{\alpha_2^r}VV @V{\alpha_1^r}VV @. \\
\dots @>>>  H^{r+2}(G,A_3) @>>>  H^{r+2}(G,A_2) @>>> 
H^{r+2}(G,A_1) @>>>
\dots 
\end{CD}
\end{equation}
 in which the vertical arrows are isomorphisms given by cup product with 
certain cohomology classes  $\alpha_i \in  H^2(G,\Hom(X_i,A_i))$.  
The two rows in the commutative diagram are the long exact Tate-cohomology
sequences for the short exact sequences (A) and (X). 

Global class field theory is encoded in the arrows $\alpha_1^r$, and in fact
 $\alpha_1$ is nothing but the global fundamental class. Similarly, local
class field theory is encoded in the arrows
$\alpha_2^r$, and in fact $\alpha_2$ is built up from  the various local
fundamental classes. We need to be more precise about this.

Let $v \in S_K$. Then $K_v^\times$ is, in an obvious way, a
direct factor of $A_2$.  The
canonical injection
$i_v:K_v^\times \hookrightarrow A_2$ and canonical projection $\pi_v:A_2
\to K_v^\times$ are both $G_v$-equivariant, where $G_v$ denotes the
stabilizer in $G$ of $v$ (in other words, the decomposition group of $v$).

In order to specify $\alpha_2$, Tate uses the following lemma
(see page 714 of his article). 

\begin{lemma}\label{lem.TateX2}
For any $G$-module $M$ there is a canonical isomorphism 
\[
H^r(G,\Hom(X_2,M)\bigr) \to \prod_{u \in S} H^r(G_v,M),
\]
where,   for each  place
$u \in S$, we choose a place $v$ of $K$ above $u$, and write
$G_v$ for its decomposition group. The $u$-th component of this isomorphism
is the composed map 
\begin{equation}\label{eq.TateLemM}
H^r(G,\Hom(X_2,M)) \xrightarrow{\Res_{G/G_v}} H^r(G_v,\Hom(X_2,M))
\xrightarrow{\eval_v} H^2(G_v,M),
\end{equation} 
where $\eval_v$ is the $G_v$-map sending $f \in \Hom(X_2,M)$ to its value
at the basis element $v \in X_2$. 
\end{lemma} 

To define $\alpha_2$ Tate applies the lemma with $M=A_2$ and $r=2$.
According to that lemma, giving $\alpha_2$ is the same as giving a family
of elements $\alpha_2(u) \in H^2(G_v,A_2)$, one for each $u \in S$. 
Tate takes $\alpha_2(u)$ to be the image under $i_v:K^\times_v
\hookrightarrow A_2$ of the local fundamental class $\alpha(K_v/F_u) \in
H^2(G_v,K^\times_v)$.

  Tate constructs $\alpha_3$ as follows. 
First he remarks that in
order to produce $\alpha_1$, $\alpha_2$, $\alpha_3$ making  
diagram \eqref{CD.Tate} commute, it would be enough to produce an element
$\alpha 
\in H^2(G,\Hom(X,A))$ whose image under the $i$-th projection
$\pi_i:\Hom(X,A) \to \Hom(X_i,A_i)$ is equal to $\alpha_i$. Here
$\Hom(X,A)$ denotes the subgroup of $\Hom(X_3,A_3) \times \Hom(X_2,A_2)
\times \Hom(X_1,A_1)$ consisting of all triples $(h_3,h_2,h_1)$ such that 
\[
\begin{CD}
X_3 @>>> X_2 @>>> X_1 \\
@Vh_3VV @Vh_2VV @Vh_1VV \\
A_3 @>>> A_2 @>>> A_1 
\end{CD}
\] 
commutes. 

The following lemma is proved in the course of the 
discussion on page 716 of Tate's article. 
\begin{lemma}\label{lem.cart} 
The diagrams
\[
\begin{CD}
\Hom(X,A) @>{\pi_1}>> \Hom(X_1,A_1) \\
@V{\pi_2}VV @VbVV  \\
\Hom(X_2, A_2) @>a>> \Hom(X_2,A_1) 
\end{CD}
\] 
and 
\[
\begin{CD}
H^2(G,\Hom(X,A)) @>{\pi_1}>> H^2(G,\Hom(X_1,A_1)) \\
@V{\pi_2}VV @VbVV  \\
H^2(G,\Hom(X_2, A_2)) @>a>> H^2(G,\Hom(X_2,A_1)) 
\end{CD}
\] 
are cartesian. 
\end{lemma}

Tate observes that $a(\alpha_2)=b(\alpha_1)$; this boils down to the
statement that, for $u \in S$ and a place $v$ of $K$ above $u$, the
restriction of the global fundamental class $\alpha_1$ to the subgroup
$G_v$ is the image under 
\[
K^\times_v \xrightarrow{i_v} A_2 \xrightarrow{a} A_1 
\]
of the local fundamental class $\alpha(K_v/F_u)$. 
 From  Lemma \ref{lem.cart} he
concludes that there exists unique $\alpha \in H^2(G,\Hom(X,A))$ such that 
$\pi_i(\alpha)=\alpha_i$ for $i=1,2$, and  he then defines $\alpha_3 \in
H^2(G,\Hom(X_3,A_3))$ to be $\pi_3(\alpha)$.

\subsection{Proof that the three triples
$(X_i,A_i,\alpha_i)$ are Tate-Nakayama triples} 

The maps 
\begin{equation}\label{eq.GlobTNisom}
H^r(G', X_i) \xrightarrow{\Res_{G/G'}(\alpha_i)\smile} H^{r+2}(G',  A_i)
\end{equation}
are isomorphisms for every subgroup $G'$ of $G$ (see Tate's article). 
Therefore
the triples
$(X_i,A_i,\alpha_i)$ are weak  Tate-Nakayama
triples.  
Here we used the following result of Tate (see the lemma on page 717 of
his article).  

\begin{lemma}[Tate] \label{lem.TateRes}
Let $G'$ be a subgroup of $G$, and let $E$ denote the fixed field of $G'$ on
$K$. Then the
canonical class $\alpha \in H^2(G,\Hom(X,A))$  restricts to the canonical
class
$\alpha' \in H^2(G',\Hom(X,A))$ for $K/E$. 
Therefore, for
$i=1,2,3$ the class $\alpha_i$ for $K/F$ restricts to the analogous class
for $K/E$.  
\end{lemma}  

To verify rigidity of our weak Tate-Nakayama triples we must  prove the
following lemma. 

\begin{lemma}\label{lem.TNcond2}
$H^1\bigl(G',\Hom(X_i,A_i)\bigr)=0$ for $i=1,2,3$ and every subgroup $G'$ of $G$.
\end{lemma}

\begin{proof}
We apply Lemma \ref{lem.rigid}. Our
goal is to establish the rigidity of $(X_i,A_i,\alpha_i)$ for $i=1,2,3$. In
the notation of that lemma,
$X_2=\mathbb Z[S_K]$, 
$X_3=\mathbb Z[S_K]_0$, and of course $X_1$ is of the form
$\mathbb Z[T]$ with
$T$ any set having exactly one element. When $i=1,2$ Lemma \ref{lem.rigid}
shows directly that $(X_i,A_i,\alpha_i)$ is rigid. When $i=3$, we need to
check that $H^{-1}(G',X_3)$ vanishes for every subgroup $G'$ of $G$. In
view of the isomorphism \eqref{eq.GlobTNisom} it is the same to check the
vanishing of $H^1(G',A_3)$, and this was done in Lemma \ref{lem.H1Van}. 
\end{proof}

\begin{remark}
It follows from Lemma \ref{lem.TateRes} that
$\Cor_{E/F}(\alpha_i(K/E))=[E:F]\alpha_i(K/F)$ for $i=1,2,3$, but we will
have no occasion to use this.  
\end{remark}

\subsection{The morphisms $(X_3,A_3,\alpha_3) \to (X_2,A_2,\alpha_2) \to
(X_1,A_1,\alpha_1)$}
It is evident  that 
\begin{itemize}
\item $(b',a')$ is a morphism from $(X_3,A_3,\alpha_3)$ to
$(X_2,A_2,\alpha_2)$, and 
\item  $(b,a)$ is a morphism from $(X_2,A_2,\alpha_2)$ to
$(X_1,A_1,\alpha_1)$.
\end{itemize} 
Here the notion of morphism is the one in Definition 
\ref{def.TNmor}. 

\subsection{Main global result}\label{sub.MainGlobal}
 We now apply the general results in   
section \ref{sec.absTN} to the three Tate-Nakayama triples
$(X_i,A_i,\alpha_i)$ and the
$G$-module $M$ obtained as the cocharacter group of a torus $T$ over $F$
that is split by $K$. We choose
an extension $\mathcal E_i$ corresponding to the class $\alpha_i \in
H^2(G,\Hom(X_i,A_i))$.  Observe that $\mathcal E_1$ is a global Weil
group for $K/F$. As in section \ref{sec.absTN} we form the group 
\[
H^1_{\alg}(\mathcal E_i,M \otimes A_i):=H^1_{Y_i}(\mathcal E_i,M
\otimes A_i) 
\]
with $Y_i=X_*(T) \otimes X_i$.

Our general results on Tate-Nakayama triples then yield a commutative
diagram 
\begin{equation} \label{CD.main}
\begin{CD}
(X_*(T) \otimes X_3)_G @>>> (X_*(T) \otimes X_2)_G @>>> X_*(T) _G
\\
 @VVV @VVV @VVV \\
H^1_{\alg}(\mathcal E_3,M \otimes A_3) @>>> H^1_{\alg}(\mathcal
E_2,M \otimes A_2) @>>> H^1_{\alg}(\mathcal E_1,M \otimes A_1) 
\end{CD}
\end{equation} 
in which the vertical arrows are the isomorphisms $c$ in Definition
\ref{def.c}, and the bottom horizontal maps are obtained by naturality 
(see subsection \ref{sub.BetterNat}) from
the morphisms of Tate-Nakayama triples that we  discussed in the previous
subsection. 

\subsection{Discussion of the top row in \eqref{CD.main}}
The first two groups in the top row of \eqref{CD.main} can be better
understood by remembering that the exact sequence (X) was defined to be 
\[
0 \to X_3 \to \bigoplus_{v \in S_K}\mathbb Z \to \mathbb Z \to 0. 
\]
Tensoring with  $X_*(T)$ preserves exactness, and
yields 
\[
0 \to X_*(T)\otimes X_3 \to \bigoplus_{v \in S_K} X_*(T) \to X_*(T) \to 0. 
\] 
Taking coinvariants for $G$ is right exact, so the top row of
\eqref{CD.main} is part of the exact sequence 
\begin{equation}\label{eq.Homol}
 \bigl(X_*(T)\otimes X_3\bigr)_G \to \bigoplus_{u \in S} X_*(T)_{G_v} \to
X_*(T)_G
\to 0, 
\end{equation} 
where   $G_v$ is again the decomposition
group of a place $v$ of $K$ that lies over $u$. 

\begin{lemma}
The kernel of the map at the left end of \eqref{eq.Homol} is canonically
isomorphic to $\ker[H^1(G,M \otimes A_3) \to H^1(G,M \otimes A_2)]$. 
\end{lemma} 

\begin{proof}
The vertical isomorphisms in the commutative diagram \eqref{CD.main} yield a
canonical isomorphism from $\ker[\bigl(X_*(T)\otimes X_3\bigr)_G \to
\bigoplus_{u} X_*(T)_{G_v}]$ to $\ker[H^1_{\alg}(\mathcal E_3,M \otimes A_3)
\to H^1_{\alg}(\mathcal E_2,M \otimes A_2)]$. One sees that this last
kernel coincides with $\ker[H^1(G,M \otimes A_3) \to H^1(G,M \otimes A_2)]$
by using  the inflation-restriction sequences (see subsection
\ref{sub.InfRes}) for both $H^1_{\alg}(\mathcal E_3,M \otimes A_3)$ and
$H^1_{\alg}(\mathcal E_2,M \otimes A_2)$, bearing in mind that $Y_3^G \to
Y_2^G$ is an injective map. 
\end{proof} 

\begin{corollary}\label{cor.CoinX}
The result of applying the functor of $G$-coinvariants to the short exact
sequence $(X)$ is a short exact sequence 
\begin{equation}
0 \to (X_3)_G \to (X_2)_G \to (X_1)_G \to 0.
\end{equation}
\end{corollary}
\begin{proof}
This can be viewed as the special case of the previous lemma in which
$T=\mathbb G_m$ and $M=\mathbb Z$. The vanishing of $H^1(G,M\otimes
A_3)=H^1(G,A_3)$ is part of Lemma \ref{lem.H1Van}.  
\end{proof}

\subsection{The special case in which $S=V_F$} 
The special case in which $S=V_F$ is especially important. The case in
which $S$ is finite (but sufficiently large) is useful as well, but plays a
more technical role. When $S=V_F$ it seems more natural to express things
differently.  For
any torus $T$ we then have 
$M\otimes A_3=T(K)$, $M \otimes A_2=T(\mathbb A_K)$ and $M \otimes
A_1=T(\mathbb A_K)/T(K)$ (with $M=X_*(T)$, as usual). 

\subsection{Restriction} \label{sub.ResTN2}
It follows from Lemma \ref{lem.TateRes} and subsection \ref{sub.ResTN}
that there are natural restriction maps 
\begin{equation}
\Res_{G/G'}: H^1_{\alg}(\mathcal E_i,M\otimes A_i) \to H^1_{\alg}(\mathcal
E'_i,M \otimes A_i)
\end{equation}
for $i=1,2,3$ and any subgroup $G'$ of $G$. Here $\mathcal E'_i$ is the
analog for $K/E$ of $\mathcal E_i$ for $K/F$, where $E$ is the fixed field
of
$G'$ on
$K$. 
  
\subsection{Discussion of the localization map (global to adelic)}
\label{sub.LocMapTori} 

The discussion of naturality   in subsection \ref{sub.BetterNat} yielded 
 the maps in the bottom row of diagram \eqref{CD.main}. In
particular, when $S=V_F$ we have constructed a localization map (global to
adelic)
\[
H^1_{\alg}(\mathcal E_3,T(K)) \to H^1_{\alg}(\mathcal E_2,T(\mathbb A_K)).
\] 
 We will see later that there is a more direct way to define localization
maps (from global all the way to local), and this can even be done in a
more general situation in which the torus $T$ is replaced by a linear
algebraic group  over $F$. 

\subsection{Preliminary discussion of inflation} The canonical isomorphism
$c$ from
$(X_*(T) \otimes X_3)_G$ to $H^1_{\alg}(\mathcal E_3,T(K))$ is a
satisfying generalization of the Tate-Nakayama isomorphism. The reader may
be troubled, however, that both the source and target of this isomorphism 
seem to depend on the choice  of $K$, which can be any finite Galois
extension of
$F$ that splits $T$. Fortunately, the choice of $K$ is unimportant, in the
following  sense. Let $K'$ be a finite Galois extension of $F$ such
that
$K'
\supset K$. Eventually we will see that  there is a commutative diagram 
\begin{equation}
\begin{CD}
(X_*(T) \otimes X_3)_G @>c>> H^1_{\alg}(\mathcal E_3,T(K)) \\
@A{\simeq}AA @V{\simeq}VV \\
(X_*(T) \otimes X'_3)_{G'} @>c'>> H^1_{\alg}(\mathcal E'_3,T(K'))
\end{CD}
\end{equation}
in which the bottom row is the analog for $K'$ of the top row, and the
vertical arrows are natural isomorphisms that we will define later.

\section{Localization $H^1_{\alg}(\mathcal E_3(K/F),G(K)) \to
H^1_{\alg}(\mathcal E(K_v/F_u),G(K_v))$}\label{sec.Loc} 

\subsection{Notation}\label{sub.LocNotation}
For the most part we retain the notation of section
\ref{sec.GTN}. There are a few differences however. In section \ref{sec.GTN}
we worked with an arbitrary set $S$ of places of $F$ satisfying the
conditions imposed in subsection \ref{sub.TNnot}, but in this section we
will keep things simple by considering only the case in which $S$ is the
set $V_F$ of all places of $F$. Another difference is that we now denote
the groups $\mathcal E_i$ ($i=1,2,3$) of subsection \ref{sub.MainGlobal} by
$\mathcal E_i(K/F)$. Moreover, we fix a place $u$ of $F$ and a place $v$ of
$K$ above $u$, and we now denote the local group
$\mathcal E$ (see section
\ref{sub.TNlc}) attached to $K_v/F_u$   by $\mathcal
E(K_v/F_u)$. Finally, the Galois group of $K/F$ (resp., $K_v/F_u$) will be
denoted by $G(K/F)$ (resp., $G(K_v/F_u)$).  

We need some additional notation. We write $E$ for the fixed field on $K$
of the decomposition group of $v$, and we write $\tilde u$ for the unique
place of $E$ under $v$. Then $v$ is the unique place of $K$ over $\tilde
u$. Moreover, $E_{\tilde u}=F_u$ and $G(K/E)=G(K_v/F_u)$. We write $\mathbb
A=\mathbb A_F$ for the adele ring of $F$. 

In the terminology of section \ref{sec.GalGerb} the extension $\mathcal
E(K_v/F_u)$ is a Galois gerb for $K_v/F_u$, bound by $\mathbb G_m$.
Similarly, the extension $\mathcal E_3(K/F)$ is a Galois gerb for $K/F$,
bound by the protorus $\mathbb T_{K/F}$ over $F$ whose character group is
$X_3$.  We write $\tilde{\mathbb T}_{K/F}$ for the protorus whose character
group is $X_2$. 
From the short exact sequence 
\[
0 \to X_3 \xrightarrow{b'} X_2 \xrightarrow{b}  X_1 \to 0 
\] 
of Galois modules, we obtain the short exact sequence  
\begin{equation}
1 \to \mathbb G_m \xrightarrow{b} \tilde{\mathbb T}_{K/F} \xrightarrow{b'} 
\mathbb T_{K/F} \to 1
\end{equation}
 of protori.

For every place $w$ of $K$ there are homomorphisms $\lambda_w:\mathbb Z \to
X_2$ and $\mu_w:X_2 \to \mathbb Z$ defined by $\lambda_w(n)=nw$ and
$\mu_w(\sum_{v \in V_K} n_v v)=n_w$. Dually, we have homomorphisms 
\[
\mathbb G_m \xrightarrow{\mu_w} \tilde{\mathbb T}_{K/F}
\xrightarrow{\lambda_w} \mathbb G_m,
\]
defined over the fixed field of the decomposition group of $w$. In
particular, $\mu_v$ and $\lambda_v$ are defined over $E$. 
We denote by $\mu_v'$ the composition  
\begin{equation}\label{eq.DefMu'}
\mu_v':\mathbb G_m \xrightarrow{\mu_v} \tilde{\mathbb T}_{K/F}
\xrightarrow{b'} \mathbb T_{K/F}. 
\end{equation}

\subsection{Goal of this section} 
In this section $G$ denotes a linear algebraic group over $F$. 
As in section \ref{sec.GalGerb} we may then consider the pointed set 
$
H^1_{\alg}(\mathcal E_3(K/F),G(K))
$
as well as its local analog
$H^1_{\alg}(\mathcal
E(K_v/F_u),G(K_v))
$. 
The goal of this section is to define a localization map 
\begin{equation}\label{eq.Loc} 
\ell^F_u:H^1_{\alg}(\mathcal E_3(K/F),G(K)) \to
H^1_{\alg}(\mathcal E(K_v/F_u),G(K_v)).
\end{equation}  
This map will be defined as the composition of two other maps. The first is
very easy to define, the second a little less so.

\subsection{Construction of the first map} As in Example 
\ref{Ex.AbsLoc}, there is a localization map 
\begin{equation}\label{eq.FirstMap}
\Loc:H^1_{\alg}(\mathcal E_3(K/F),G(K)) \to H^1_{\alg}(\mathcal
E_3^v(K/F),G(K_v)),
\end{equation}
where $\mathcal E_3^v(K/F)$ is obtained from $\mathcal E_3(K/F)$ by first
pulling back along $G(K_v/F_u) \hookrightarrow G(K/F)$ and then pushing
forward using $\mathbb T_{K/F}(K) \to \mathbb T_{K/F}(K_v)$. 
The arrow \eqref{eq.FirstMap} is the first of the two maps we need to
define. 

\subsection{Construction of the second map} The second map we need to define
is of the type considered in subsection \ref{sub.ChBnd}; it involves a change
in band. The Galois gerb $\mathcal E_3^v(K/F)$ for $K_v/F_u$ is bound by
the protorus $\mathbb T_{K/F}$, but now viewed  over $F_u$
rather than $F$. The local Galois gerb $\mathcal E(K_v/F_u)$ is bound by
$\mathbb G_m$; the associated cohomology class is the fundamental 
class $\alpha(K_v/F_u)$.  We have already defined (see \eqref{eq.DefMu'}) a canonical
$F_u$-homomorphism  
$\mu_v':\mathbb G_m \to \mathbb T_{K/F}$. In order to invoke subsection
\ref{sub.ChBnd} we need to extend $\mu_v'$ to $\tilde\mu_v'$, as in the next
lemma.

\begin{lemma}\label{lem.LocWorksWell}
\hfill 
\begin{enumerate}
\item  The groups $H^1(G(K_v/F_u),{\mathbb T}_{K/F}(K_v))$ and  
$H^1(G(K_v/F_u),\tilde{\mathbb T}_{K/F}(K_v))$ vanish.

\item There 
exists a homomorphism $\tilde\mu_v'$ making the diagram 
\begin{equation}
\begin{CD}
1 @>>> \mathbb G_m(K_v) @>>> \mathcal E(K_v/F_u)@>>> G(K_v/F_u) @>>> 1 \\ 
@. @V{\mu_v'}VV @V{\tilde\mu_v'}VV @| @.\\
1 @>>> \mathbb T_{K/F}(K_v) @>>> \mathcal E^v_3(K/F)@>>> G(K_v/F_u) @>>> 1
\end{CD}
\end{equation}
commute, and $\tilde\mu_v'$ is unique up to conjugation by $\mathbb
T_{K/F}(K_v)$.
\end{enumerate}
 
\end{lemma}

\begin{proof}
(1) 
When we view 
$
0 \to X_3 \xrightarrow{b'} X_2 \xrightarrow{b} X_1 \to 0
$ 
as a short exact sequence of $G(K_v/F_u)$-modules, it has a canonical
splitting, namely the  homomorphism $\lambda_v$ defined in subsection 
\ref{sub.LocNotation}.

From this we conclude that 
$H^1(G(K_v/F_u),\mathbb T_{K/F}(K_v))$ is a direct summand of  
$H^1(G(K_v/F_u),\tilde{\mathbb T}_{K/F}(K_v))$, a group that vanishes by
Lemma 
\ref{lem.TateX2} and Hilbert's Theorem~90. 

(2) The uniqueness assertion regarding  $\tilde \mu_v'$
follows from part (1) of this lemma. In proving  the existence
statement it is harmless to replace $F$ by $E$, and so we may assume that 
$F=E$ (and hence that $G(K/F)=G(K_v/F_u)$). This is convenient notationally,
because  we may then write
$\alpha_i$ without having to specify whether we are referring to $K/F$ or
$K/E$. We do this for the rest of the proof, and, because the
linear algebraic group is irrelevant at the moment, we temporarily 
 revert to denoting $G(K/F)$ by $G$,
and the decomposition group of a place $w \in V_K$ by $G_w$. 

The existence of $\tilde \mu_v'$ is equivalent to the equality 
\begin{equation}\label{eq.MuClass79} 
\mu'_v(\alpha(K_v/F_u))=\pi_va'\alpha_3, 
\end{equation} 
where $\pi_v$ is (induced by) the projection of $\mathbb A_K$ on its direct factor
$K_v$. 
We claim that \eqref{eq.MuClass79} is a consequence of 
\begin{equation}\label{eq.MuClass} 
\mu_v(\alpha(K_v/F_u))=\pi_v\alpha_2. 
\end{equation}  
Indeed, one obtains the first equation from the second by applying $b'$ to both sides, 
bearing in mind that $b'$ commutes with $\pi_v$ and that $b'\alpha_2=a'\alpha_3$.

It remains to prove \eqref{eq.MuClass}. 
 By Lemma
\ref{lem.TateX2}, in order to show that
$\pi_v\alpha_2$ and
$\mu_v(\alpha(K_v/F_u))$ are equal, it suffices to show that, for every
place $w$ of $K$, they have the same image under 
\begin{equation*} 
\rho_w:H^2(G,\tilde {\mathbb T}_{K/F}(K_v))  \xrightarrow{\Res_{G/G_w}}
H^2(G_w,\tilde {\mathbb T}_{K/F}(K_v)) \xrightarrow{\lambda_w}
H^2(G_w,K_v^\times). 
\end{equation*}
There is a similarly defined map $\rho_w$ with $\mathbb A_K$ replacing
$K_v$, as well as a  commutative diagram 
\begin{equation*}
\begin{CD}
H^2(G,\tilde {\mathbb T}_{K/F}(\mathbb A_K))  @>{\rho_w}>>
H^2(G_w,\mathbb A_K^{\times})\\
@V{\pi_v}VV @V{\pi_v}VV \\
H^2(G,\tilde {\mathbb T}_{K/F}(K_v)) 
@>{\rho_w}>> H^2(G_w,K_v^\times),
\end{CD}
\end{equation*}
so we conclude that $\rho_w\pi_v\alpha_2=\pi_v\rho_w\alpha_2$.

 It follows  (see the discussion at  the bottom of page 714 in \cite{T}) from the 
definition of $\alpha_2$ that $\rho_w\alpha_2=i_w(\alpha(K_w/F_w))$,
where
$F_w$ is the completion of $F$ at the unique place of $F$ lying under $w$,
and $i_w$ is the obvious inclusion of $K_w^\times$ as a direct factor of
$\mathbb A_K^\times$.  Therefore 
\[
\rho_w\pi_v\alpha_2=\pi_v i_w \alpha(K_w/F_w)=
\begin{cases}
\alpha(K_v/F_u) &\text{ if $w= v$,} \\ 
0 &\text{ if $w\ne v$.}
\end{cases}
\]
Furthermore 
\[
\rho_w \mu_v(\alpha(K_v/F_u))=
\begin{cases}
\alpha(K_v/F_u) &\text{ if $w= v$,} \\ 
0 &\text{ if $w\ne v$,}
\end{cases}
\]
because $\lambda_w\mu_v$ is the identity map when $w=v$ and is $0$
otherwise. This concludes the proof.  
\end{proof}

We use the homomorphism  $\tilde \mu_v'$ in the lemma to obtain 
\begin{equation}\label{eq.SecondMap}
(\tilde\mu_v')^*:H^1_{\alg}(\mathcal E_3^v(K/F),G(K_v)) \to
H^1_{\alg}(\mathcal E(K_v/F_u),G(K_v))
\end{equation}
as in subsection \ref{sub.ChBnd}.

\subsection{End of the definition of the localization map $\ell^F_u$} 
We now define the localization map \eqref{eq.Loc} to be the
composition of \eqref{eq.FirstMap} and \eqref{eq.SecondMap}. 
\begin{remark} 
It follows immediately from the equality \eqref{eq.MuClass} that there
exists a homomorphism $\tilde\mu_v$ making the diagram 
\begin{equation}
\begin{CD}
1 @>>> \mathbb G_m(K_v) @>>> \mathcal E(K_v/F_u)@>>> G(K_v/F_u) @>>> 1 \\ 
@. @V{\mu_v}VV @V{\tilde\mu_v}VV @| @.\\
1 @>>> \tilde{\mathbb T}_{K/F}(K_v) @>>> \mathcal E^v_2(K/E)@>>> G(K/E)
@>>> 1
\end{CD}
\end{equation}
commute, where the bottom row is obtained by  pushforward from 
the extension $\mathcal E_2(K/E)$ of $G(K/E)$ by $\tilde{\mathbb
T}_{K/F}(\mathbb A_K)$.  Moreover, by the first part of the lemma, 
$\tilde\mu_v$ is unique up to conjugation by
$\tilde{\mathbb T}_{K/F}(K_v)$. Using $\tilde\mu_v$ one can easily construct
a localization map 
\begin{equation}\label{AdelLoc}
H^1_{\alg}(\mathcal E_2(K/F),G(\mathbb A_K)) \to H^1_{\alg}(\mathcal
E(K_v/F_u),G(K_v)) 
\end{equation}
once one takes the trouble to define the set $H^1_{\alg}(\mathcal
E_2(K/F),G(\mathbb A_K))$. Then one can go on to construct a commutative
diagram 
\begin{equation}
\begin{CD}
H^1_{\alg}(\mathcal E_3(K/F),G(K)) @>>>H^1_{\alg}(\mathcal
E_2(K/F),G(\mathbb A_K)) \\
@| @V{\eqref{AdelLoc}}VV \\
H^1_{\alg}(\mathcal E_3(K/F),G(K)) @>{\eqref{eq.Loc}}>> H^1_{\alg}(\mathcal
E(K_v/F_u),G(K_v)).
\end{CD}
\end{equation}
We omit the details, as we will not make use of this diagram. 
\end{remark}

\subsection{Compatibility of localization with the Newton map} 
 The next result gives the compatibility
between  localization and the Newton map 
\eqref{eq.NuCpt}.  
\begin{lemma} \label{lem.NewtLocGlob}
Let $G$ be any linear algebraic group over $F$. Then 
the diagram 
\begin{equation}
\begin{CD}
H^1_{\alg}(\mathcal E_3(K/F),G(K)) @>{\eqref{eq.Loc}}>>
H^1_{\alg}(\mathcal E(K_v/F_u),G(K_v)) \\
@V{\eqref{eq.NuCpt}}VV @V{\eqref{eq.NuCpt}}VV \\
[\Hom_K(\mathbb T_{K/F},G)/G(K)]^{G(K/F)} @>{\mu_v'}>> [\Hom_{K_v}(\mathbb
G_m,G)/G(K_v)]^{G(K_v/F_u)}
\end{CD}
\end{equation}
commutes.
\end{lemma} 
\begin{proof}
Easy.  
\end{proof} 

\subsection{A commutative square involving the localization map for tori}\label{sub.LocForTorComm}
Now consider an $F$-torus $T$ split by $K$, and write $M$ for $X_*(T)$. 
Then form the square 
\begin{equation}\label{CD.LocKappaT}
\begin{CD}
(M \otimes X_3)_{G(K/F)} @>>> M_{G(K_v/F_u)}\\
@VcV{\simeq}V @VcV{\simeq}V \\
H^1_{\alg}(\mathcal E_3(K/F),T(K)) @>{\eqref{eq.Loc}}>>
H^1_{\alg}(\mathcal E(K_v/F_u),T(K_v)) 
\end{CD}
\end{equation} 
The top arrow is obtained as follows. We begin with  the composed map 
 \[
X_3 \hookrightarrow X_2  \twoheadrightarrow \mathbb Z[V_u],
\]
where $V_u$ denotes the set of places of  $K$ lying over $u$, and the second arrow 
is projection onto the direct summand $\mathbb Z[V_u]$ of $X_2=\mathbb Z[V_K]$.  
   Tensoring with $M$ and 
forming $G(K/F)$-coinvariants, we obtain the composed map 
\[
(M \otimes X_3)_{G(K/F)}  \to  (M \otimes X_2)_{G(K/F)}  \twoheadrightarrow 
(M \otimes \mathbb Z[V_u])_{G(K/F)} \simeq M_{G(K_v/F_u)},
\]
and it is this that we take as the top arrow in our square. More explicitly, the top arrow 
is induced by the map sending  
\[
\sum_{w \in V_K} m_w w \in M \otimes X_3 \hookrightarrow M \otimes X_2 = \bigoplus_{w \in V_K} M 
\]
 to  the element 
$\sum_{\sigma \in G(K_v/F_u)\backslash G(K/F)} \sigma(m_{\sigma^{-1}v}) \in M_{G(K_v/F_u)}$. 

\begin{lemma}\label{lem.KappaTLocCompat}
The square \eqref{CD.LocKappaT} commutes.
\end{lemma}
\begin{proof}
Use the same method as in the proof of Lemma \ref{lem.Trick}, the point being that the diagram 
 \begin{equation}
\begin{CD}
(M \otimes X_3)_{G(K/F)} @>>>   M_{G(K_v/F_u)} \\
@VVV @VVV\\
(M \otimes X_3)^{G(K/F)} @>>>   M^{G(K_v/F_u)}
\end{CD}
\end{equation}
commutes. Here the left (resp.~ right) vertical arrow is the global (resp.~local) norm map. 
The top arrow is  the same as the top arrow in \eqref{CD.LocKappaT}. The bottom arrow 
sends  $G(K/F)$-invariant  $\sum_{w \in V_K} m_w w \in M \otimes X_3$ to its $v$-component $m_v$. 
\end{proof}

\section{Inflation $H^1_{\alg}(\mathcal E(K/F),G(K)) \to H^1_{\alg}(\mathcal E(L/F),G(L))$} \label{sec.InflationLocGlob}
When we studied the four examples of Tate-Nakayama triples (one local, 
three global),  we always fixed the Galois extension $K$ of $F$. 
 Now we need to see what
happens to
$H^1_{\alg}$ when we enlarge $K$ to $L$. 
The essential
point is the relationship between the various fundamental classes for the
two layers $K/F$ and $L/F$.  

\subsection{Local theory} 
We consider local fields  $L\supset K \supset F$ with both $L$ and $K$ 
finite 
Galois over $F$. We keep track of which Galois group and fundamental
class we are talking about by labeling them  with $K/F$ or $L/F$, as
appropriate. We then have (see section \ref{sec.LTN}) the local
Tate-Nakayama triple $(\mathbb Z,K^\times,\alpha(K/F))$ for $K/F$, as well
as the Galois gerb for $K/F$ (and bound by $\mathbb G_m$) 
\begin{equation} \label{eq.EKF}
1 \to \mathbb G_m(K) \to \mathcal E(K/F) \to G(K/F) \to 1. 
\end{equation}
 When $F$ is
$\mathbb Q_p$, this is the Dieudonn\'e gerb of \cite{LR} attached to $K$.
When $K/F$ is $\mathbb C/\mathbb R$, it is the weight gerb of \cite{LR}. 

Now let $G$ be a linear algebraic group over $F$. We may then consider 
the set
$H^1_{\alg}(\mathcal E(K/F),G(K))$, as well as its analog for $L/F$.  
Our goal is to  
define  
a natural map  
\begin{equation}\label{eq.arr3} 
H^1_{\alg}(\mathcal E(K/F),G(K)) \to H^1_{\alg}(\mathcal E(L/F),G(L)), 
\end{equation}
and then to show it is an isomorphism when $G$ is a torus split by $K$.

We can inflate
the Galois gerb $\mathcal E(K/F)$ to $L/F$ (see
subsection \ref{sub.ChGalExt}), obtaining 
\begin{equation} \label{eq.infEKF}
1 \to \mathbb G_m(L) \to \mathcal E(K/F)^{\inf} \to G(L/F) \to 1.
\end{equation} 
 As in Example \ref{Ex.AbsInf}, we  
then have an inflation map 
\begin{equation}\label{eq.arr1}
H^1_{\alg}(\mathcal E(K/F),G(K)) \to H^1_{\alg}(\mathcal E(K/F)^{\inf},G(L)).
\end{equation}

Define a homomorphism $p_{L/K}:\mathbb G_m \to \mathbb G_m$ by $x \mapsto
x^{[L:K]}$. The fundamental classes $\alpha(K/F)
\in H^2(G(K/F),K^\times)$ and $\alpha(L/F) \in H^2(G(L/F),L^\times)$ are
related by the equation 
\begin{equation}\label{eq.LocFundInf}
\inf(\alpha(K/F))=[L:K]\alpha(L/F),  
\end{equation} 
and therefore there exists a homomorphism $\eta_{L/K}:\mathcal E(L/F)
\to \mathcal E(K/F)^{\inf}$ making the diagram 
\begin{equation}\label{eq.arr2}
\begin{CD}
1 @>>> \mathbb G_m(L) @>>> \mathcal E(L/F) @>>> G(L/F) @>>> 1 \\ 
@. @V{p_{L/K}}VV @V{\eta_{L/K}}VV @| @.\\
1 @>>> \mathbb G_m(L) @>>> \mathcal E(K/F)^{\inf} @>>> G(L/F) @>>> 1
\end{CD}
\end{equation} 
commute. 
We then obtain the induced map 
\begin{equation}\label{eq.arr4}
\eta_{L/K}^*: H^1_{\alg}(\mathcal E(K/F)^{\inf}, G(L)) \to
H^1_{\alg}(\mathcal E(L/F),G(L))
\end{equation}
defined in subsection \ref{sub.ChBnd}. 
The composition of 
\eqref{eq.arr1} and \eqref{eq.arr4} is the  map \eqref{eq.arr3} that 
we wanted to define.

Now take $G$ to be an $F$-torus $T$ split by $K$ and consider the
diagram  
\begin{equation}\label{CD.LcInf}
\begin{CD}
X_*(T)_{G(K/F)} @>{c}>> H^1_{\alg}(\mathcal E(K/F),T(K)) @>r>>
X_*(T)^{G(K/F)} \\ 
@AAA @V{\eqref{eq.arr3}}VV @V[L:K]VV\\
X_*(T)_{G(L/F)} @>{c}>> H^1_{\alg}(\mathcal E(L/F),T(L)) @>r>>
X_*(T)^{G(L/F)}
\end{CD}
\end{equation} 
with $c$ as in subsection
\ref{sub.TNlc}, 
$r$ as in subsection \ref{sub.InfRes}, and where  
the left vertical arrow is the obvious isomorphism
(the one induced by the identity map on $X_*(T)$).

\begin{lemma}\label{lem.LocInf}
The diagram \eqref{CD.LcInf} commutes, and all four arrows in the left square
are isomorphisms. 
\end{lemma} 
\begin{proof}
Since three of the four arrows in the left square are already known to be
isomorphisms, our only real task is to prove that the diagram commutes. 

The right square commutes, as one sees easily from the definition of
the arrow \eqref{eq.arr3}, the point being   that $p_{L/K}:\mathbb G_m \to
\mathbb G_m$ induces multiplication by $[L:K]$ on cocharacter groups. 
Moreover, we know (see Lemmas  
\ref{lem.c0} and \ref{lem.cNew}) that the composition $rc$ in the top row
(resp., bottom row) is the norm map for $G(K/F)$ (resp., $G(L/F)$). It is
therefore clear that the outer rectangle commutes.  We  conclude  that the
left square
 commutes whenever the restriction map $r$ in the bottom
row is injective, and, by the inflation-restriction sequence 
of subsection \ref{sub.InfRes}, this happens
if and only if
$H^1(G(L/F),T(L))$ vanishes. 

In particular, by Shapiro's lemma and Hilbert's Theorem 90, the left square
does commute when
$X_*(T)$ is free of finite rank as $\mathbb Z[G(K/F)]$-module. For general
$T$, we choose a free $\mathbb Z[G(K/F)]$-module $M$ and a surjective
$G(K/F)$-module map $f:M \twoheadrightarrow X_*(T)$. We then obtain $f:T_M
\to T$, where $T_M$ is the torus with cocharacter group $M$. The
 left square is functorial in $T$, so its commutativity
for $T$ follows from that for $T_M$,  because 
$f:M_{G(L/F)} \to X_*(T)_{G(L/F)}$ is surjective. 
\end{proof}

\subsection{New system of notation for our three global
Tate-Nakayama triples} Next we are going to study inflation for our three
global Tate-Nakayama triples $(X_i,A_i,\alpha_i)$, so it is no longer
feasible to omit the extension $K/F$ from the notation. We write $G(K/F)$
for the Galois group of $K/F$, and 
we now denote our three Tate-Nakayama triples by 
$(X_i(K),A_i(K),\alpha_i(K/F))$. In this more elaborate system of
notation the Tate class $\alpha \in H^2(G,\Hom(X,A))$ discussed near 
the end of
subsection~\ref{sub.GlobTNT} becomes 
\[
\alpha(K/F) \in H^2\bigl(G(K/F),\Hom(X(K),A(K)\bigr). 
\] 
We also need to remember that $A_2$ and $A_3$ depend on a choice of subset
$S \subset V_F$, and that the conditions imposed on $S$ become more and
more stringent the bigger the top field $K$ gets. 

The relevant extensions are now 
denoted by 
\begin{equation}\label{eq.extK/F}
1 \to \Hom(X_i(K),A_i(K)) \to \mathcal E_i(K/F) \to G(K/F) \to 1
\end{equation} 
 and the relevant cohomology groups by $H^1_{\alg}(\mathcal E_i(K/F),M
\otimes A_i(K))$, where $M$ is the cocharacter group of an $F$-torus $T$
that is split by $K$. When $S=V_F$ and $i=3$ the cohomology group is 
$H^1_{\alg}(\mathcal E_3(K/F),T(K))$ and is of the type considered in
section 
\ref{sec.GalGerb}, the relevant group of multiplicative type being the
protorus $\mathbb T_{K/F}$ whose group of characters is $X_3(K/F)$. 

The norm map for the finite group $G(K/F)$ will now 
be denoted by   $N_{K/F}$.  
We recall that there are functorial maps 
\begin{equation}\label{eq.GlobalPreInf}
(M \otimes X_i(K))_{G(K/F)} \xrightarrow{c} H^1_{\alg}(\mathcal E_i(K/F),M
\otimes A_i(K)) \xrightarrow{r} (M \otimes X_i(K))^{G(K/F)}.  
\end{equation}
Here $c$ is the  isomorphism of subsection \ref{sub.MainGlobal},  
and $r$ is the restriction map of subsection \ref{sub.InfRes}. Moreover 
$rc=N_{K/F}$ by  Lemmas  
\ref{lem.c0} and \ref{lem.cNew}.

\subsection{The setup in which to discuss inflation in the global
situation} \label{sub.InfSetUp}

We now consider a finite Galois extension $L$ of $F$ with $L
\supset K$. We then have a short exact sequence 
\[
1 \to G(L/K) \to G(L/F) \to G(K/F) \to 1.
\]
 We seek an analog of Lemma \ref{lem.LocInf} for our three global
Tate-Nakayama triples, in which $S \subset V_F$ is assumed to be big enough
to satisfy the conditions (see subsection \ref{sub.TNnot})  needed for the
extension
$L/F$; it is then automatic that $S$ also satisfies these conditions for 
$K/F$.  In order to get started, we need maps relating the groups
$X_i(K)$,
$A_i(K)$ to the parallel objects for
$L/F$. In the case of $A_i$, this is  straightforward: there is an obvious
isomorphism of $A_i(K)$ with $A_i(L)^{G(L/K)}$. 

In the case of $X_i$ there are maps in both directions. In fact we are
going to define two commutative diagrams 
\begin{equation}
\begin{CD}
0 @>>> X_3(L) @>>> X_2(L) @>>> X_1(L) @>>> 0 \\
@. @V{j_3}VV @V{j_2}VV @V{j_1}VV @. \\
0 @>>> X_3(K) @>>> X_2(K) @>>> X_1(K) @>>> 0
\end{CD}
\end{equation}
and 
\begin{equation}
\begin{CD}
0 @>>> X_3(L) @>>> X_2(L) @>>> X_1(L) @>>> 0 \\
@. @A{p_3}AA @A{p_2}AA @A{p_1}AA @. \\
0 @>>> X_3(K) @>>> X_2(K) @>>> X_1(K) @>>> 0.
\end{CD}
\end{equation} 
All we really need to do is to define maps $j_2$, $j_1$, $p_2$, $p_1$
making  the two right squares commute;  we are then forced to define $j_3$,
$p_3$ by restriction. Now $X_1(L) =\mathbb Z=X_1(K)$. We take 
$j_1$ to be the identity map on $\mathbb Z$, and $p_1$ to be multiplication
by $[L:K]$. 

Recall that $X_2(K)$ is the free abelian group on the set
$S_K$ of places of $K$ that lie over some place
in $S$.   The value of $j_2$ on the basis element $w \in S_L$
of $X_2(L)$ is defined to be $v$, where $v$ is the unique place of $K$
lying under $w$. The value of $p_2$ on the basis element $v \in S_K$ of
$X_2(K)$ is defined to be 
\[
p_2(v):=\sum_{w|v} [L_w:K_v] w. 
\]
The desired commutativity of the two right squares is clear. 

\begin{lemma}\label{lem.pandj}
The following statements hold for $i=1,2,3$.  
\begin{enumerate} 
\item $p_i \circ j_i=N_{L/K}$.
\item $j_i \circ  p_i=[L:K]$. Consequently $p_i$ is injective. 
\item The map $j_i:X_i(L) \twoheadrightarrow X_i(K)$ factors through the
coinvariants of
$G(L/K)$ on
$X_i(L)$, inducing an isomorphism 
\[
\gamma_i:X_i(L)_{G(L/K)} \to X_i(K). 
\]
\end{enumerate}
\end{lemma} 
\begin{proof}
The only thing that might not be obvious is that $\gamma_i$ is an
isomorphism. This follows readily from the definitions when $i=1,2$. To
handle $i=3$ one appeals to Corollary \ref{cor.CoinX}, applied to $L/K$.  
\end{proof}

\subsection{Definition of  global inflation maps for tori} 
As before we write $M$ for the cocharacter group of an $F$-torus $T$ split
by $K$. We need to define global inflation maps (for $i=1,2,3$) 
\begin{equation}\label{>1}
H^1_{\alg}(\mathcal E_i(K/F),M \otimes A_i(K)) \to H^1_{\alg}(\mathcal
E_i(L/F),M \otimes A_i(L)).
\end{equation} 
As in the local case, we will do so in three steps, first defining two auxiliary maps 
and then taking their composition as the definition of  \eqref{>1}. 

{\bf Step 1.} We  
 use $G(L/F) \twoheadrightarrow G(K/F)$ and
$A_i(K)=A_i(L)^{G(L/K)}$ to define 
\begin{equation}\label{>2}
H^1_{\alg}(\mathcal E_i(K/F),M \otimes A_i(K)) \to H^1_{Y_i(K)}(\mathcal
E_i(K/F)^{\inf},M \otimes A_i(L)),
\end{equation} 
where $Y_i(K):=M \otimes X_i(K)$, and $\mathcal
E_i(K/F)^{\inf}$ is obtained from $\mathcal E_i(K/F)$ by first pulling back
along $G(L/F) \twoheadrightarrow G(K/F)$ and then pushing out along 
the inclusion of 
$\Hom(X_i(K),A_i(K))$ as the set of $G(L/K)$-fixed points in 
$ \Hom(X_i(K),A_i(L))$. Thus our inflated extension sits in an exact
sequence  
\begin{equation}\label{eq.extK/Finf}
1 \to \Hom(X_i(K),A_i(L)) \to \mathcal E_i(K/F)^{\inf} \to G(L/F) \to 1.
\end{equation} 
The map $\xi':Y_i(K) \to \Hom\bigl(\Hom(X_i(K),A_i(L)),M\otimes
A_i(L)\bigr)$  used to form the group $H^1_{Y_i(K)}(\mathcal
E_i(K/F)^{\inf},M \otimes A_i(L))$ is the
obvious tautological one. 

 In the special case $S=V_F$ and $i=3$,
the  map
\eqref{>2} is an instance of the inflation map 
in Example \ref{Ex.AbsInf}. In all cases it is an instance of the very
general inflation map in Example \ref{Ex.Inf2}; more precisely it is of the
form $\Phi(f,g,h)$ where  
\begin{itemize}
\item $f:M \otimes A_i(K) \to M \otimes A_i(L)$ is induced by $A_i(K)
\hookrightarrow A_i(L)$, 
\item $g$ is the identity map on $Y_i(K)$, and 
\item $h$ is the inclusion $\Hom(X_i(K),A_i(K))\hookrightarrow
\Hom(X_i(K),A_i(L))$ that we used to form the pushout.  
\end{itemize}

{\bf Step 2.}
From $p_i:X_i(K) \hookrightarrow X_i(L)$ we obtain an induced map 
\[
p_i:\Hom(X_i(L),A_i(L)) \to 
\Hom(X_i(K),A_i(L))
\]
We want to choose a homomorphism $\tilde p_i$ making the diagram 
\begin{equation}
\begin{CD} 
1 @>>> \Hom(X_i(L),A_i(L)) @>>> \mathcal E_i(L/F) @>>> G(L/F)
@>>> 1\\ 
@. @V{p_i}VV @V{\tilde p_i}VV @| @.\\
1 @>>> \Hom(X_i(K),A_i(L)) @>>> \mathcal E_i(K/F)^{\inf} @>>> G(L/F)
@>>> 1
\end{CD}
\end{equation}
commute. 

\begin{lemma}\label{lem.tildePiExists}
For $i=1,2,3$ such a homomorphism $\tilde p_i$ exists and is unique 
up to conjugation by an element in the subgroup $\Hom(X_i(K),A_i(L))$. 
\end{lemma}
\begin{proof}
The existence of $\tilde p_i$ is equivalent to the statement that 
\begin{equation}\label{eq.weaker}
p_i(\alpha_i(L/F))=\inf(\alpha_i(K/F)) \in
H^2(G(L/F),\Hom(X_i(K),A_i(L)) ). 
\end{equation}
In fact, we will prove a slightly
stronger statement involving the Tate classes 
\begin{itemize}
\item $\alpha(K/F) \in H^2(G(K/F),\Hom(X(K),A(K)))$, 
\item $\alpha(L/F)\in H^2(G(L/F),\Hom(X(L),A(L)))$. 
\end{itemize}

The slightly stronger statement is that 
\begin{equation}\label{eq.stronger}
p(\alpha(L/F))=\inf(\alpha(K/F)) \in H^2(G(L/F),\Hom(X(K),A(L))).
\end{equation}
As the notation suggests, $\Hom(X(K),A(L))$ is defined as the group of
triples
$(h_3,h_2,h_1) $ making 
\[
\begin{CD}
X_3(K) @>>> X_2(K) @>>> X_1(K) \\
@Vh_3VV @Vh_2VV @Vh_1VV \\
A_3(L) @>>> A_2(L) @>>> A_1(L) 
\end{CD}
\] 
commute, and the map $p$ is induced by $(p_3,p_2,p_1)$.   Because the   
image of
$\alpha(K/F)$ under the
$i$-th projection is
$\alpha_i(K/F)$, the new statement does in fact imply the three old
statements.

Lemma \ref{lem.cart} exhibits 
$H^2(G(K/F),\Hom(X(K),A(K)))$ as a fiber product, and the same goes
with $K/F$ replaced by $L/F$. We are now going to prove an analog of 
Lemma \ref{lem.cart} 
for $H^2(G(L/F),\Hom(X(K),A(L)))$.

{\bf Claim.} The diagrams  
\begin{equation}\label{CD.car1}
\begin{CD}
\Hom(X(K),A(L)) @>{\pi_1}>> \Hom(X_1(K),A_1(L)) \\
@V{\pi_2}VV @VbVV  \\
\Hom(X_2(K), A_2(L)) @>a>> \Hom(X_2(K),A_1(L)) 
\end{CD}
\end{equation}
and 
\begin{equation}\label{CD.car2}
\begin{CD}
H^2(G(L/F),\Hom(X(K),A(L))) @>{\pi_1}>> H^2(G(L/F),\Hom(X_1(K),A_1(L))) \\
@V{\pi_2}VV @VbVV  \\
H^2(G(L/F),\Hom(X_2(K), A_2(L))) @>a>> H^2(G(L/F),\Hom(X_2(K),A_1(L))) 
\end{CD}
\end{equation}
are cartesian. 

We prove the claim by imitating Tate's argument. It is clear that
\eqref{CD.car1} is cartesian. In other words $\Hom(X(K),A(L))$ is the
equalizer of the two obvious maps 
\[
\Hom(X_1(K),A_1(L)) \oplus \Hom(X_2(K), A_2(L)) \to
\Hom(X_2(K),A_1(L)). 
\]
Equivalently, $\Hom(X(K),A(L))$ is the
kernel of the difference $\delta$ of these two obvious maps. Now $\delta$
is surjective, as follows from the surjectivity of the  bottom horizontal arrow
in \eqref{CD.car1} (itself a consequence of the fact that $X_2(K)$ is a free
abelian group). The fact that \eqref{CD.car2} is cartesian now follows from
the long exact $G(L/F)$-cohomology sequence for the short exact sequence 
\[
0 \to \ker(\delta) \to \source(\delta) \to \target(\delta) \to 0,
\]
together with the vanishing of $H^1(G(L/F),\Hom(X_2(K),A_1(L)))$, an easy
consequence of  Lemmas \ref{lem.TateX2} and \ref{lem.H1Van}. This finishes
the proof of the claim. 

The claim shows that in order to prove \eqref{eq.stronger} (and hence prove
\eqref{eq.weaker} for $i=3$), it is enough to prove \eqref{eq.weaker} for
$i=1,2$. Now for $i=1$ \eqref{eq.weaker} reduces to the (standard) fact
that the inflation to $L/F$ of the global fundamental class for $K/F$ is
$[L:K]$ times the global fundamental class for $L/F$.

To handle $i=2$ we are going to reduce 
to the local case. 
We need to prove that 
\begin{equation}\label{eq.weaker'}
p_2(\alpha_2(L/F))=\inf(\alpha_2(K/F)) \in
H^2(G(L/F),\Hom(X_2(K),\mathbb A_L^\times) ). 
\end{equation} 
The first step is to use Lemma \ref{lem.ShapVar} in order to analyze the 
maps 
\begin{equation}\label{Map.1}
H^2(K/F,\Hom(X_2(K), \mathbb A_K^\times)) \xrightarrow{\inf}
H^2(L/F,\Hom(X_2(K), \mathbb A_L^\times)) 
\end{equation} 
and 
\begin{equation}\label{Map.2} 
H^2(L/F,\Hom(X_2(L), \mathbb A_L^\times))
\xrightarrow{p_2}
H^2(L/F,\Hom(X_2(K), \mathbb A_L^\times)). 
\end{equation} 

From that lemma we have identifications 
\begin{itemize}
\item $H^2(K/F,\Hom(X_2(K), \mathbb A_K^\times))=\prod_{u \in V_F}
H^2(G(K/F)_v,
\mathbb A_K^\times)$,
\item $H^2(L/F,\Hom(X_2(K), \mathbb A_L^\times))= \prod_{u \in V_F}
H^2(G(L/F)_v,
\mathbb A_L^\times)$,
\item $H^2(L/F,\Hom(X_2(L), \mathbb A_L^\times))= \prod_{u \in V_F}
H^2(G(L/F)_w,
\mathbb A_L^\times)$,  
\end{itemize} 
where, for each $u \in V_F$, we have first chosen a
place
$v$ of $K$ lying over $u$ and then chosen a place $w$ of $L$ lying over
$v$. 
With these identifications, the map \eqref{Map.1} becomes the product (over
 $u\in V_F$) of the inflation maps
\begin{equation}\label{Map.3}
H^2(G(K/F)_v, \mathbb A_K^\times) \xrightarrow{\inf} 
H^2(G(L/F)_v, \mathbb A_L^\times)
\end{equation}
coming from  $G(L/F)_v \twoheadrightarrow G(K/F)_v$. 
From the definition of $p_2$ it follows fairly easily that the map 
\eqref{Map.2} becomes the product (over $u\in V_F$) of the maps 
\begin{equation}\label{Map.4}
H^2(G(L/F)_w, \mathbb A_L^\times) \xrightarrow{[L_w:K_v]\Cor} 
H^2(G(L/F)_v, \mathbb A_L^\times),
\end{equation} 
where $\Cor$ is the corestriction map for the subgroup $G(L/F)_w$ of
$G(L/F)_v$. 

At this point it is helpful to consider the diagrams (one for each $u \in
V_F$)  
\begin{equation*}
\begin{CD}
H^2(G(K/F)_v, \mathbb A_K^\times) @>{\inf}>> 
H^2(G(L/F)_v, \mathbb A_L^\times) @<{[L_w:K_v]\Cor}<< H^2(G(L/F)_w,
\mathbb A_L^\times)
\\
@A{i_K}AA@A{i_L}AA @A{i_L}AA \\
H^2(G(K/F)_v, K_v^\times) @>{\inf}>> 
H^2(G(L/F)_v, L_v^\times) @<{[L_w:K_v]\Cor}<< H^2(G(L/F)_w,
L_v^\times) \\
@| @V{Sh}VV @AjAA \\
H^2(G(K/F)_v, K_v^\times) @>{\inf}>> H^2(G(L/F)_w, L_w^\times)
@<{[L_w:K_v]}<< H^2(G(L/F)_w, L_w^\times)
\end{CD}
\end{equation*}
where $L_v:=L \otimes_K K_v=\prod_{w'|v}L_{w'}$, the map $Sh$ is the Shapiro
isomorphism, and the maps $i_K$, $i_L$, $j$ are (induced by) the obvious
inclusions $K_v^\times \hookrightarrow \mathbb A_K^\times$, 
$L_v^\times \hookrightarrow \mathbb A_L^\times$,  
$L_w^\times \hookrightarrow L_v^\times$, respectively.

Unwinding the definition of the adelic fundamental class $\alpha_2$, we see
that \eqref{eq.weaker'} is equivalent to the equality, for all $u \in V_F$,
of the elements 
\[
\beta'_u,\beta''_u \in 
H^2(G(L/F)_v, \mathbb A_L^\times)
\]
defined by  $\beta'_u:=\inf (i_K(\alpha(K_v/F_u)))$ and 
$\beta''_u:=[L_w:K_v]\Cor (i_Lj(\alpha(L_w/F_u)))$. Now it is part of local
classfield theory that $[L_w:K_v]\alpha(L_w/F_u)$ is equal to the inflation
of $\alpha(K_v/F_u)$. So \eqref{eq.weaker'} follows from the
commutativity of the big diagram above. The commutativity of the left top
square comes from the naturality of inflation. The commutativity of the
right top square comes from the naturality of corestriction. The
commutativity of the left bottom square is easy to check,  using that the
Shapiro isomorphism is given by restriction (for $G(L/F)_w \subset
G(L/F)_v$) followed by the natural projection $L_v^\times
\twoheadrightarrow L_w^\times$. The commutativity of the right lower square
follows from Lemma \ref{lem.IndCor}. So we are done proving the  equality 
\eqref{eq.weaker'} for $i=2$.

We have finished the proof that the maps $\tilde p_i$ exist. Now we need to
establish the uniqueness statement asserted in the statement of the lemma. 
For this we just need to prove the vanishing of
$H^1(G(L/F),\Hom(X_i(K),A_i(L)))$ for $i=1,2,3$. 

For $i=1$  we just need  the vanishing of $H^1(G(L/F),
\mathbb A_L^\times/L^\times)$, and this is one of the standard results of
global classfield theory (see Lemma  \ref{lem.H1Van}).  
For $i=2$ the desired vanishing follows from 
 Lemmas \ref{lem.HrZS} and \ref{lem.H1Van}. 

For $i=3$ we must prove the vanishing of $H^1(G(L/F),\Hom(X_3(K),A_3(L))$. 
We claim that cup product with $\alpha_3(L/F)$ yields  isomorphisms
\[
H^r(G(L/F),\Hom(X_3(K),X_3(L))) \simeq H^{r+2}(G(L/F),\Hom(X_3(K),A_3(L)))
\] 
for all $r \in \mathbb Z$. Indeed, $X_3(K)$ lies in the class 
$\mathcal C(X_3(L),A_3(L),\alpha_3(L/F))$ of Definition \ref{def.calC}, as 
one sees from parts (2) and  (5) of Lemma
\ref{lem.calC}.  

So we just need to check that $H^{-1}(G(L/F),\Hom(X_3(K),X_3(L)))$
vanishes,  
and this follows from 
 Lemma \ref{lem.H-10S}(3). In
that lemma we  take $\epsilon$ to be the obvious surjection $S_L
\twoheadrightarrow S_K$. To apply the lemma, we need the vanishing of 
$H^{-1}(G',X_3(L))$ for every subgroup $G'$ of $G(L/F)$, and this  follows
from   Tate's isomorphism (cup product with the restriction to $G'$ 
of $\alpha_3(L/F)$) 
\[
H^{-1}(G',X_3(L)) \simeq H^{1}(G',A_3(L))
\]
 together with Lemma \ref{lem.H1Van}. The proof of Lemma \ref{lem.tildePiExists}
 is finally complete. 
\end{proof}

As a consequence of Lemma \ref{lem.tildePiExists}, we obtain a well-defined map 
\begin{equation}  \label{>3}
\tilde p_i^*:H^1_{Y_i(K)}(\mathcal
E_i(K/F)^{\inf},M \otimes A_i(L)) \to  H^1_{\alg}(\mathcal E_i(L/F),M
\otimes A_i(L)). 
\end{equation}
When $S=V_F$ and  $i=3$, this is an instance of the map \eqref{eq.NewBnd}.
In general one uses the map $\id_M \otimes p_i:Y_i(K) \to Y_i(L)$ to define
the cocycle-level map 
\[
(\nu,x) \mapsto \bigl((\id_M \otimes p_i) (\nu), x \circ \tilde p_i\bigr).
\]

{\bf Step 3.}
Define the arrow \eqref{>1}  as the composition of
\eqref{>2} and \eqref{>3}.

\subsection{Global inflation isomorphisms} Now we are in a position to
prove our main result on inflation in the global situation. 
We consider the diagram  
\begin{equation}\label{CD.GlobInf}
\begin{CD}
(M \otimes X_i(K))_{G(K/F)} @>c>> H^1_{\alg}(\mathcal E_i(K/F),M \otimes
A_i(K)) @>r>> (M \otimes X_i(K))^{G(K/F)}
\\ @A{\id_M
\otimes j_i}AA @V{\eqref{>1}}VV @V{\id_M
\otimes p_i}VV\\
(M \otimes X_i(L))_{G(L/F)} @>c>> H^1_{\alg}(\mathcal E_i(L/F),M \otimes
A_i(L)) @>r>> (M \otimes X_i(L))^{G(L/F)}
\end{CD}
\end{equation}
in which the rows are instances of \eqref{eq.GlobalPreInf}.

\begin{lemma}\label{CD.GlobInflation}
For $i=1,2,3$ the diagram \eqref{CD.GlobInf} commutes, and all four arrows in
the left square are isomorphisms.  
\end{lemma}
\begin{proof}
The two horizontal arrows $c$ are isomorphisms. So too is the left vertical
arrow, because it is obtained by applying the functor of
$G(K/F)$-coinvariants to the isomorphism 
\[
(M\otimes
X_i(L))_{G(L/K)} = M \otimes X_i(L)_{G(L/K)} \xrightarrow{\id_M
\otimes\gamma_i} M \otimes X_i(K)
\] 
(recall that $\gamma_i$ was defined
in  Lemma \ref{lem.pandj}(3)). Since three of the four arrows in the
left square are isomorphisms, we will know that the fourth one is too, once
we have proved that the diagram commutes. 

We are going to prove that the diagram commutes by the method we used in 
the local case. We must show that \eqref{CD.GlobInf}
commutes for all
$G(K/F)$-modules $M$ that are free of finite rank as $\mathbb Z$-modules. 
One sees directly that the right square commutes, and it follows easily from 
Lemma \ref{lem.pandj}(1) that the outer rectangle also commutes. It remains
to show that the left square commutes.  We are going to use that all the maps
in the left square
 are functorial in
$M$.  First let us treat the special case  
 when $M$ is free of finite rank as $\mathbb Z[G(K/F)]$-module. 

In this special case it is easily checked that $H^1(G(K/F),M
\otimes A_i(K))$ vanishes, and the same is true with $K/F$ replaced by
$L/F$. Therefore   the  horizontal arrows $r$ (coming from the
inflation-restriction sequence discussed in subsection \ref{sub.InfRes})
are both injective. So the commutativity of the
right square and outer rectangle  implies that of the left square in this
special case.

 In the general case we choose a finite free 
$\mathbb Z[G(K/F)]$-module $M'$ and a surjective Galois-equivariant map $M'
\twoheadrightarrow M$.  The desired commutativity for $M$ follows
from the known commutativity for $M'$, simply because the natural
map 
\[
(M' \otimes X_i(L))_{G(L/F)} \to (M \otimes X_i(L))_{G(L/F)}
\] 
is surjective. 
\end{proof}

\subsection{Definition of the groups $B_i(F,T)$ } \label{sub.DefBiT}
Let $T$ be a torus defined over $F$. We choose a separable algebraic closure
$\bar F$ of $F$.  For
$i=1,2,3$ we  put 
\begin{equation}\label{def.BiT}
B_i(F,T):=\injlim_{\mathcal K} H^1_{\alg}(\mathcal E_i(K/F),M \otimes
A_i(K)),
\end{equation}
the direct limit being taken over the set $\mathcal K$ of finite Galois
extensions
$K$ of
$F$ in $\bar F$ such that $T$ splits over $K$.

As an immediate consequence of Lemma \ref{CD.GlobInflation}, we obtain for
$i=1,2,3$ a canonical isomorphism 
\begin{equation}\label{eq.BiT}
\projlim_{\mathcal K} (M \otimes X_i(K))_{G(K/F)} = B_i(F,T),
\end{equation} 
with $\mathcal K$ as before. 
 As we saw in that lemma, the transition morphisms in the
projective system are all isomorphisms. So the real content of
\eqref{eq.BiT} is that $B_i(F,T)$ can be identified with any of the groups 
$(M \otimes X_i(K))_{G(K/F)}$ (for $K \in \mathcal K$), these all being
canonically isomorphic to each other.  

\subsection{Definition of global inflation maps for linear algebraic groups} 

We have already defined three global inflation maps 
\begin{equation}\label{>1NewEq}
H^1_{\alg}(\mathcal E_i(K/F),M \otimes A_i(K)) \to H^1_{\alg}(\mathcal
E_i(L/F),M \otimes A_i(L)) 
\end{equation} 
for any $F$-torus $T$  split by $K$ (with $M=X_*(T)$). We did so for any sufficiently 
big subset $S$ of $V_F$.  
In this subsection we take $S=V_F$. For $i=3$ the inflation map \eqref{>1NewEq} 
then becomes 
\begin{equation}\label{>1NewEq2}
H^1_{\alg}(\mathcal E_3(K/F),T(K)) \to H^1_{\alg}(\mathcal
E_3(L/F),T(L)). 
\end{equation} 

We are now going to generalize \eqref{>1NewEq2} to an inflation map 
\begin{equation}\label{>1NewEq3}
H^1_{\alg}(\mathcal E_3(K/F),G(K)) \to H^1_{\alg}(\mathcal
E_3(L/F),G(L)) 
\end{equation} 
defined for any linear algebraic $F$-group $G$. This is easy to do; all 
the real work was done in proving in Lemma \ref{lem.tildePiExists}.  
Just as in the local case the global inflation map \eqref{>1NewEq3} is defined as 
the composed map 
\begin{equation}\label{>1NewEq4}
H^1_{\alg}(\mathcal E_3(K/F),G(K)) \to H^1_{\alg}(\mathcal E_3(K/F)^{\inf},G(L)) 
\xrightarrow{\tilde p_3^*} H^1_{\alg}(\mathcal E_3(L/F),G(L)) 
\end{equation} 
with the first arrow as in Example \ref{Ex.AbsInf} and the second one as in 
subsection \ref{sub.ChBnd}.

\section{The natural transformations $\kappa_G$ and $\bar\kappa_G$ for $H^1_{\alg}$} 
\subsection{Assumptions and notation for this section}\label{sub.NatKapAsmpt}
Let $K/F$ be a finite Galois extension of fields. In this section we
consider a Tate-Nakayama triple $(X,A,\alpha)$ for the Galois group
$G(K/F)$ such that 
\begin{itemize}
\item $X$ is torsion-free as abelian group, and 
\item $A$ is the $G(K/F)$-module $K^\times$. 
\end{itemize} 
When $F$ is local or global we have already seen that there exists
a canonical such Tate-Nakayama triple. (The $G(K/F)$-module $X$ is $\mathbb
Z$ in the local case and  $\mathbb Z[V_K]_0$ in the global case.) The point
of working with $(X,A,\alpha)$ as above is that it allows us to treat the
local and global cases simultaneously.

We write $D_X$ for the $F$-group of multiplicative type with $X^*(D_X)=X$.
Then $\Hom(X,A)=D_X(K)$, so $\alpha \in H^2(G(K/F),D_X(K))$ provides us with a
Galois gerb 
\[
1 \to D_X(K) \to \mathcal E \to G(K/F) \to 1,
\]
and  the pointed
set 
$H^1_{\alg}(\mathcal E, G(K))$ is defined for each linear algebraic
$F$-group
$G$. 

\subsection{Review of the algebraic fundamental group} 
Let $G$ be a
  connected reductive $F$-group  split by $K$.  
We write $\Lambda_G$ for
Borovoi's \cite{B2} algebraic fundamental group of $G$. For any maximal
$F$-torus $T$  in $G$ there is a canonical identification
$\Lambda_G=X_*(T)/X_*(T_{\ssc})$ of $G(K/F)$-modules. 

For any torus $T$ split by $K$, we have $\Lambda_T=X_*(T)$.
When the derived group of $G$ is simply connected, the natural map
$\Lambda_G
\to \Lambda_D$ is an isomorphism, where $D$ denotes the quotient of $G$ by
its derived group. When 
$
1 \to Z \xrightarrow{i} G' \xrightarrow{p} G \to 1
$ 
is a $z$-extension, the sequence 
$
0 \to \Lambda_Z \xrightarrow{i} \Lambda_{G'} \xrightarrow{p} \Lambda_{G}
\to 0
$
is easily seen to be exact. 

\subsection{Construction of $\kappa_G$} 
For any $F$-torus $T$ split by $K$, the inverse of the map $c$ appearing in
Lemma \ref{lem.AbTN2} provides us with a canonical isomorphism 
\begin{equation}\label{eq.NatT}
\kappa_T:H^1_{\alg}(\mathcal E, T(K)) \to (\Lambda_T \otimes X)_{G(K/F)}.
\end{equation}
In the next result we consider both $H^1_{\alg}(\mathcal E, G(K)) $ and 
$(\Lambda_G \otimes X)_{G(K/F)}$ as functors from the category of connected
reductive $F$-groups $G$ split by $K$ to the category of pointed sets. 

\begin{proposition}\label{prop.NatG}
There exists a unique natural transformation 
\begin{equation}\label{eq.NatG}
\kappa_G:H^1_{\alg}(\mathcal E, G(K)) \to (\Lambda_G \otimes X)_{G(K/F)}
\end{equation}
that agrees with \eqref{eq.NatT} for $F$-tori split by $K$. 
\end{proposition}

\begin{proof}
As usual we construct $\kappa_G$ in two stages. In the first stage we
consider only those $G$ whose derived group is simply connected. We are 
then forced to define $\kappa_G$ as the unique map making 
\begin{equation}\label{CD.UnwindKappa1}
\begin{CD}
H^1_{\alg}(\mathcal E, G(K)) @>{\kappa_G}>> (\Lambda_G \otimes X)_{G(K/F)}
\\ @VVV @| \\
H^1_{\alg}(\mathcal E, D(K)) @>{\kappa_D}>> (\Lambda_D \otimes X)_{G(K/F)}
\end{CD}
\end{equation}
commute, where $D$ is the quotient of $G$ by its derived group. It is easy
to see that $\kappa_G$ is functorial in $G$ (for homomorphisms between
groups with simply connected derived group). 

In the second stage we use $z$-extensions. As is well-known, for any
connected reductive $F$-group $G$ split by $K$, there exists an
extension 
\[
1 \to Z \xrightarrow{i} G' \xrightarrow{p} G \to 1 
\] 
such that 
\begin{itemize}
\item $Z$ is a central torus in $G'$, 
\item $Z$ is obtained by Weil restriction of scalars from a split
$K$-torus, and 
\item $G'_{\der}$ is simply connected. 
\end{itemize}
This is true even when $F$ is not perfect. Indeed it is very easy to
construct an extension as above if one only asks that $Z$ be a central
subgroup of multiplicative type (take $G'$ to be the product of $G_{\ssc}$
and the biggest central torus in $G$). Then push out along some embedding $Z
\hookrightarrow Z'$ such that  
\begin{itemize}
\item $Z'$ is  a torus whose character group is a finitely
generated free module over $\mathbb Z[G(K/F)]$, and 
\item  $X^*(Z') \to X^*(Z)$ is 
surjective 
\end{itemize}
in order to obtain the desired $z$-extension. 

We contend that 
\begin{itemize}
\item The  map  
\[
p:H^1_{\alg}(\mathcal E, G'(K)) \to H^1_{\alg}(\mathcal E, G(K))
\]
 identifies $H^1_{\alg}(\mathcal E, G(K))$ 
with the quotient of $H^1_{\alg}(\mathcal E, G'(K))$ by
the action of  $H^1_{\alg}(\mathcal E, Z(K))$. 

\item The map $(\Lambda_{G'} \otimes X)_{G(K/F)} 
\xrightarrow{p}
(\Lambda_{G} \otimes X)_{G(K/F)}$ identifies $(\Lambda_{G} \otimes
X)_{G(K/F)}$ with the quotient of 
$(\Lambda_{G'} \otimes X)_{G(K/F)}$ by the action of 
$(\Lambda_Z \otimes X)_{G(K/F)}$. 

\item 
The map $\kappa_{G'}$ constructed in the
first stage is  equivariant with respect to the group $H^1_{\alg}(\mathcal
E, Z(K))=(\Lambda_Z
\otimes X)_{G(K/F)}$.
\end{itemize}

The first item 
will  follow from Proposition \ref{prop.Zact} once we check that $\mathcal
E$ satisfies the two assumptions made in subsection \ref{sub.StrR}. It is
clear that Assumption 1 holds, because we are assuming in this section that
$X^*(D_X)=X$ is torsion-free.  Assumption 2 holds due to 
the isomorphism \eqref{eq.NatT} and the right-exactness of the functor $M
\mapsto (M
\otimes X)_{G(K/F)}$. 

The second item is a restatement of the exactness of 
\[
 (\Lambda_Z \otimes X)_{G(K/F)} 
\xrightarrow{i} (\Lambda_{G'} \otimes X)_{G(K/F)} 
\xrightarrow{p}
(\Lambda_{G} \otimes X)_{G(K/F)}
\to 0,
\] 
itself a consequence of the exactness of 
 $
0 \to \Lambda_Z \xrightarrow{i} \Lambda_{G'} \xrightarrow{p} \Lambda_{G}
\to 0. 
$ 
The third item is evident from the way $\kappa_{G'}$ was defined.

The three items we just verified imply that there exists a unique
bottom horizontal arrow making   the diagram 
\begin{equation}\label{CD.UnwindKappa2}
\begin{CD}
H^1_{\alg}(\mathcal E, G'(K)) @>{\kappa_{G'}}>> (\Lambda_{G'} \otimes
X)_{G(K/F)} \\ @VpVV @VpVV \\
H^1_{\alg}(\mathcal E, G(K)) @>>> (\Lambda_{G} \otimes
X)_{G(K/F)} 
\end{CD}
\end{equation}
commute, and clearly we are forced to take $\kappa_G$ to be this
arrow.  
 It follows easily from Lemma 2.4.4 in
\cite{CTT} that
$\kappa_G$ is well-defined (independent of the choice of $z$-extension $G'$)
and that it is functorial in $G$. 
\end{proof}

\subsection{Construction of $\bar\kappa_G$}
We continue with $(X,A,\alpha)$ as in subsection 
\ref{sub.NatKapAsmpt}.  
For any intermediate field $E$ (for $K/F$) we obtain 
a Tate-Nakayama triple $(X,A,\alpha_E)$ for $K/E$, where $\alpha_E$ denotes the 
restriction of $\alpha$ to the subgroup $G(K/E)$ of $G(K/F)$. We denote 
by $\mathcal E'$ the preimage of $G(K/E)$ under $\mathcal E \twoheadrightarrow G(K/F)$. 
We again  consider a connected reductive $F$-group $G$ split by $K$.

For the definition of the restriction map occurring as the left vertical map in the next lemma see 
Example \ref{Ex.ResMpp}. 

\begin{lemma}\label{lem.KapEnF}
The square 
\begin{equation}\label{CD.KappaEandF}
\begin{CD}
H^1_{\alg}(\mathcal E, G(K)) @>{\kappa_{G}}>> (\Lambda_{G} \otimes
X)_{G(K/F)} \\ @V{\Res}VV @VVV \\
H^1_{\alg}(\mathcal E', G(K)) @>{\kappa_{G}}>> (\Lambda_{G} \otimes
X)_{G(K/E)} 
\end{CD}
\end{equation} 
commutes. Here the right vertical arrow sends the class of 
$y \in \Lambda_G \otimes X$ to the class of $\sum_{\sigma \in G(K/E)\backslash 
G(K/F)} \sigma y$. 
\end{lemma} 

\begin{proof} Using 
\eqref{CD.UnwindKappa2} (and its analog for $\mathcal E'$), we reduce to the case in which the derived group of $G$ 
is simply connected. Next, using \eqref{CD.UnwindKappa1} (for $\mathcal E$ and $\mathcal E'$), 
 we reduce to the case in which $G$ is a torus. Finally, tori are handled by Lemma \ref{lem.Trick}.  
\end{proof} 

In the extreme case when the intermediate field $E$ is $K$, the bottom horizontal arrow 
in \eqref{CD.KappaEandF} is a map 
\begin{equation}\label{eq.KappaZeroPrelim}
\Hom_K(D_X,G)/G(K) \to \Lambda_G \otimes X.
\end{equation} 
We claim that the map \eqref{eq.KappaZeroPrelim} is $G(K/F)$-equivariant.  Indeed, the strategy of the 
proof of the previous lemma reduces us to the case of tori, for which the claim is obvious. We now 
define 
\begin{equation}\label{eq.KappaZero}
\bar\kappa_G:\bigl[\Hom_K(D_X,G)/G(K)\bigr]^{G(K/F)} \to (\Lambda_G \otimes X)^{G(K/F)}
\end{equation} 
to be the map obtained by applying the functor of $G(K/F)$-invariants to  
\eqref{eq.KappaZeroPrelim}. 

\subsection{Compatibility of $\kappa_G$ and $\bar\kappa_G$}

As an immediate consequence of Lemma \ref{lem.KapEnF} we obtain the following result.  

\begin{lemma}\label{lem.NewtonKappaPrelim}
The square 
\begin{equation}\label{CD.KappaEandFNew}
\begin{CD}
H^1_{\alg}(\mathcal E, G(K)) @>{\kappa_{G}}>> (\Lambda_{G} \otimes
X)_{G(K/F)} \\ @V{Newton}VV @VNVV \\
\bigl[\Hom_K(D_X,G)/G(K)\bigr]^{G(K/F)} @>{\bar\kappa_{G}}>> (\Lambda_{G} \otimes
X)^{G(K/F)} 
\end{CD}
\end{equation} 
commutes. The right vertical arrow $N$ is the norm map  for $K/F$. 
\end{lemma} 

\section{The set $B(F,G)$}\label{sec.DiscSetBFG}

\subsection{Notation} 
In this section $F$ is a local or global field. We fix a separable closure
$\bar F$ of $F$ and put $\Gamma=\Gal(\bar F/F)$.  For any finite Galois
extension
$K/F$ inside $\bar F$ there is a canonical Tate-Nakayama triple
$(X,A,\alpha)$ in which
$A$ is the
$G(K/F)$-module
$K^\times$. In this section we   denote this triple by 
$(X(K),K^\times,\alpha(K/F))$  in order to keep track of its
dependence on
$K/F$, and we write $\mathbb D_{K/F}$ for the $F$-group of multiplicative
type with character group $X(K)$.

  In the local case the triple is the one considered
in section
\ref{sec.LTN}, so
$X(K)=\mathbb Z$, $\mathbb D_{K/F}=\mathbb G_m$ and $\alpha(K/F)$ is the
fundamental class in
$H^2(K/F,\mathbb G_m(K))$. In the global case  $X(K)$, $\alpha(K/F)$
are  the objects
$X_3$,  $\alpha_3$ from subsection \ref{sub.GlobTNT}, with $S$
chosen to be the set of all places of $F$. So, in the global case, 
$X(K)=\mathbb Z[V_K]_0$, $\mathbb D_{K/F}=\mathbb T_{K/F}$ and
$\alpha(K/F)$ is Tate's canonical class  $\alpha_3 \in H^2(K/F,\mathbb
T_{K/F}(K))$. 

We choose an extension  
 \[
1 \to \mathbb D_{K/F}(K) \to \mathcal E(K/F) \to G(K/F) \to 1 
\] 
whose associated cohomology class is $\alpha(K/F)$. 
In this section we are using a unified system of notation for the local and
global cases, so that we can give uniform statements and proofs.  In the
global case   
$\mathcal E(K/F)$ was previously denoted by $\mathcal E_3(K/F)$.

\subsection{Definition of $B(F,G)$} 

For any linear algebraic $F$-group $G$ we  
 define a pointed set $B(F,G)$ 
by 
\[
B(F,G):=\injlim_K H^1_{\alg}(\mathcal E(K/F), G(K)).
\]
The colimit is taken over the set of finite Galois extensions $K$ of
$F$ in
$\bar F$. For $L \supset K$, the transition map is  the inflation
map 
\begin{equation}\label{eq.CurrentInf}
H^1_{\alg}(\mathcal E(K/F), G(K)) \to H^1_{\alg}(\mathcal E(L/F), G(L))
\end{equation}
defined in  section \ref{sec.InflationLocGlob}  
(see \eqref{eq.arr3} in the local case and
\eqref{>1NewEq3} in the global case). 
The transition maps are easily seen to be transitive.  

In order to define $B(F,G)$ we had to choose a separable closure $\bar F$. 
Just as for Galois cohomology, this choice is unimportant: an isomorphism 
$\phi$ from $\bar F$ to another separable closure $\bar F'$ induces an isomorphism 
$\phi_*$ 
from the set $B(F,G)$ formed using $\bar F$ to the one formed using $\bar F'$, 
and this induced isomorphism is independent of the choice of $\phi$. 
(To see that $\phi_*$ is well-defined one uses the vanishing of the groups 
$H^1(G(K/F),\mathbb D_{K/F}(K))$, and to see that $\phi_*$ is independent of the 
choice of $\phi$, one uses that inner automorphisms by elements in $\mathcal E(K/F)$  
induce trivial automorphisms of  $H^1_{\alg}(\mathcal E(K/F),G(K))$.) 

For a torus $T$ over a global field $F$, it is clear from the definitions
that 
\begin{equation}
B(F,T)=B_3(F,T).
\end{equation}
(The group $B_3(F,T)$ was defined in subsection \ref{sub.DefBiT}.) 

\subsection{The maps $p$ and $j$}\label{sub.defp&j} 

Let $K$ and $L$ be finite Galois extensions of $F$ in $\bar F$ with $K
\subset L$. We are now going to define a canonical injection $p:X(K) \to
X(L)$ and a canonical surjection $j:X(L) \to X(K)$. In the local case
$X(K)=X(L)=\mathbb Z$, and we take 
\begin{itemize}
\item $p$ to be multiplication by $[L:K]$, 
\item $j$ to be the identity map. 
\end{itemize}
In the global case we take $p$, $j$ to be the maps $p_3$, $j_3$ defined in
subsection \ref{sub.InfSetUp}. 

\begin{lemma}\label{lem.pandjNew}
The following statements hold.  
\begin{enumerate} 
\item $p \circ j=N_{L/K}$.
\item $j \circ  p=[L:K]$.  
\item The map $j:X(L) \twoheadrightarrow X(K)$ factors through the
coinvariants of
$G(L/K)$ on
$X(L)$, inducing an isomorphism 
\[
\gamma:X(L)_{G(L/K)} \to X(K). 
\]
\end{enumerate}
\end{lemma}

\begin{proof}
The local case is obvious and the global case is part of Lemma
\ref{lem.pandj}. 
\end{proof} 

\subsection{The protorus $\mathbb D_F$ over $F$} 
We define $\mathbb D_F$ to be the protorus over $F$ whose
character group is 
\[
X^*(\mathbb D_F):=\injlim_K \, X(K),
\]
with transition maps  $p:X(K) \hookrightarrow X(L)$. 
Thus $\mathbb D_F=\projlim_K \mathbb D_{K/F}$. 

\subsection{Concrete description  of $X^*(\mathbb
D_F)$ in the local case}

In this subsection $F$ is a local field. 
We then have $X^*(\mathbb D_F)= \Br^*(F)$, where 
\[
\Br^*(F):= 
\begin{cases}
\mathbb Q &\text{ if $F$ is nonarchimedean,} \\
\frac{1}{2}\mathbb Z &\text{ if $F$ is $\mathbb R$,} \\
\mathbb Z &\text{ if $F$ is $\mathbb C$.} 
\end{cases}
\]
In all cases $\Br^*(F)$ is an extension of $\Br(F)$ by $\mathbb Z$. In
other words, there are natural short exact sequences 
\[
0 \to \mathbb Z \to \Br^*(F) \to \Br(F) \to 0. 
\] 
 
\subsection{Concrete description  of $X^*(\mathbb
D_F)$ in the global case} 

In this subsection $F$ is a global field. 
For  each place $u$ of $F$ we have $X^*(\mathbb
D_{F_u})=\Br^*(F_u)$. 
We are going to introduce an analogous global group $\Br^*(F)$. 
First we recall that $\Br(F)$ can be identified with 
\[
\ker\bigl[\bigoplus_{u \in V_F} \Br(F_u) \to \mathbb
Q/\mathbb Z\bigr], 
\] 
the map to $\mathbb Q/\mathbb Z$ sending $(x_u)_{u \in V_F}$ to 
$\sum_{u \in V_F} x_u$. 
This suggests defining $\Br^*(F)$ as  
\[
\Br^*(F):=\ker\bigl[\bigoplus_{u \in V_F} \Br^*(F_u) \to \mathbb Q\bigr],
\]
the map to $\mathbb Q$ sending $(x_u)_{u \in V_F}$ to 
$\sum_{u \in V_F} x_u$. 
There is then an obvious short exact sequence 
\[
0 \to X(F) \to \Br^*(F) \to \Br(F) \to 0. 
\]
We usually consider $X(K)$ for some finite Galois extension $K/F$, but
in the exact sequence above we are working with $K=F$. However, we can now
apply these definitions to any finite Galois extension $K$ of $F$,
obtaining 
 \[
0 \to X(K) \to \Br^*(K) \to \Br(K) \to 0. 
\]

Now we need to see what happens when we vary $K$. Suppose we have $F
\subset K \subset L \subset \bar F$, with both $K$ and $L$ finite Galois
over $F$. There is then a commutative diagram 
\begin{equation}
\begin{CD}
0 @>>> X(K) @>>> \Br^*(K) @>>> \Br(K) @>>> 0 \\ 
@. @V{p}VV @V{q}VV @V{\Res}VV @. \\
 0 @>>> X(L) @>>> \Br^*(L) @>>>
\Br(L) @>>> 0 .
\end{CD}
\end{equation}
Here $p$ is  our usual map, and $\Res$ is the restriction map for the
extension $L/K$. The map $q$ sends $(x_v)_{v \in V_K} \in \Br^*(K)$ to
$(y_w)_{w \in V_L} \in \Br^*(L)$, with $y_w=[L_w:K_v]x_v$ when $w$ lies
over $v$.

Now pass to the colimit over $K$. Since any element in $\Br(K)$ dies in
$\Br(L)$ for sufficiently large $L$, the colimit of the groups $\Br(K)$
(with restriction maps as transition morphisms) is trivial. 
So we obtain a canonical isomorphism 
\[
X^*(\mathbb D_F)=\injlim_{K} X(K) \simeq \injlim_K \Br^*(K). 
\]
It is easy to check that the natural map $\Br^*(K) \to X^*(\mathbb D_F)$ is
injective, and that its image consists of the fixed points of $\Gal(\bar
F/K)$ on $X^*(\mathbb D_F)$. Similarly, for any finite extension $E$ of $F$ 
in $\bar F$, there is a natural identification of $\Br^*(E)$ with the fixed 
points of $\Gal(\bar F/E)$ on $X^*(\mathbb D_F)$. 

\subsection{Newton map}\label{sub.NewtonMapping} 
The Newton maps \eqref{eq.NuCpt} 
\[ 
H^1_{\alg}(\mathcal E(K/F),G(K)) \to [\Hom_K(\mathbb
D_{K/F},G)/G(K)]^{G(K/F)}
\]
fit together to give a Newton map 
\begin{equation}\label{eq.SlopeMap} 
 B(F,G) \to [\Hom_{\bar F}(\mathbb D_F,G)/G(\bar F)]^{\Gamma}. 
\end{equation} 
The image of $b \in B(F,G)$ under the Newton map is called the 
\emph{Newton point} of $b$. 

The maps $H^1(G(K/F),G(K)) \hookrightarrow H^1_{\alg}(\mathcal E(K/F),G(K))$
fit together to give an injective map 
\begin{equation}
H^1(F,G) \hookrightarrow B(F,G),
\end{equation}
whose image is the kernel of the Newton map \eqref{eq.SlopeMap}.

\subsection{Basic elements} We denote by $Z(G)$ the center of $G$. The
inclusion of $Z(G)$ in $G$ induces an injection  
\begin{equation}\label{eq.TarNewt}
\Hom_F(\mathbb D_F,Z(G)) \hookrightarrow [\Hom_{\bar F}(\mathbb D_F,G)/G(\bar
F)]^{\Gamma}.
\end{equation}

\begin{definition}
We say that $b \in B(F,G)$ is
\emph{basic} if its Newton point 
lies in the image of \eqref{eq.TarNewt}.  We write $B(F,G)_{\bsc}$ for the
set of basic elements in $B(F,G)$. 
\end{definition}

\subsection{Localization} 
Now suppose that $F$ is global, and consider a place $u$ of $F$. In order to 
define $B(F,G)$ and $B(F_u,G)$ we have to choose separable closures $\bar F$ 
and $\bar F_u$, although, as we have already mentioned, this choice is 
of no real importance.  Now we choose some embedding $\bar F \hookrightarrow \bar F_u$ 
over $F$. 
This embedding gives us, for each finite 
Galois extension $K$ of $F$ in $\bar F$, a place $v$ of $K$ lying over $u$.

The localization maps obtained using these places $v$ 
 are compatible with inflation and therefore yield a 
map  
\begin{equation}  \label{eq.BLCzero} 
B(F,G) \to  B(F_u,G). 
\end{equation} 
Just as for ordinary Galois cohomology, the map 
\eqref{eq.BLCzero} is actually 
independent of the choice  of $F$-embedding 
  $\bar F \hookrightarrow \bar F_u$.

When $G$ is connected,  Corollary \ref{cor.ResDirProd} 
tells us that the components of any element in the image of the
total localization map
\begin{equation}  \label{eq.BLC} 
B(F,G) \to \prod_{u \in V_F} B(F_u,G). 
\end{equation}
 are trivial for all but finitely many  $u \in V_F$. 

The choice of $F$-embedding $\bar F \hookrightarrow \bar F_u$ 
yields an $F_u$-homomorphism $\mu':\mathbb D_{F_u} \to \mathbb D_F$
(assembled out of the maps $\mu'_v$ in subsection \ref{sub.LocNotation}), 
and the induced map 
\[
[\Hom_{\bar F}(\mathbb D_F,G)/G(\bar
F)]^{\Gamma} \to 
[\Hom_{\bar F_u}(\mathbb D_{F_u},G)/G(\bar
F_u)]^{\Gamma(u)}
\] 
is easily seen to be independent of the choice of $F$-embedding 
$\bar F \hookrightarrow \bar F_u$. Here we are writing $\Gamma(u)$ for 
the Galois group of $\bar F_u/F_u$. Moreover the Newton map is compatible 
with localization: the diagram 
\begin{equation}
\begin{CD}
B(F,G) @>>>  B(F_u,G) \\
@VVV@VVV \\
[\Hom_{\bar F}(\mathbb D_F,G)/G(\bar
F)]^{\Gamma} @>>>
[\Hom_{\bar F_u}(\mathbb D_{F_u},G)/G(\bar
F_u)]^{\Gamma(u)}
\end{CD}
\end{equation} 
commutes.

\begin{lemma}\label{lem.BscLocGlobal}
An element  $b\in B(F,G)$ is basic if and only if its image $b_u$ in
$B(F_u,G)$ is basic for every place $u$ of $F$. 
\end{lemma}
\begin{proof} 
It is clear that a globally basic element is locally basic everywhere.  So our 
real task is to prove the converse, and we now suppose that $b_u$ is basic for every place 
$u$ of $F$.  
There exists a finite Galois extension $K/F$ such that $b$ is represented
by an algebraic $1$-cocycle 
 $(\nu,x)$ of $\mathcal E_3(K/F)$ in  $G(K)$. 
The Newton point 
 for $b$ is  the $G(\bar
F)$-conjugacy class of the $\bar F$-homomorphism 
\[
\mathbb D_F \twoheadrightarrow \mathbb T_{K/F} \xrightarrow{\nu} G.
\]
To show that $b$ is basic we must show that  $\nu$ is central. 

Now let $u$ be a place of $F$, and let $v$ be any place 
of $K$ lying over $u$.  Choose an $F_u$-embedding $K_v \hookrightarrow \bar F_u$. 
By Lemma \ref{lem.NewtLocGlob} the Newton point  for
$b_u$ is the $G(\bar F_u)$-conjugacy class of 
\[
\mathbb D_{F_u} \twoheadrightarrow \mathbb G_m \xrightarrow{\nu_v} G,
\]
where $\nu_v$ is the composed homomorphism 
\[
\mathbb G_m \xrightarrow{\mu_v} \tilde{\mathbb T}_{K/F} \to \mathbb T_{K/F}
\xrightarrow{\nu} G.
\] 
Because $b_u$ is basic, the homomorphism 
$\nu_v$ is central.  
To see that $b$ is basic, we now use the 
 following observations. 
\begin{itemize}
\item $\mathbb T_{K/F} \xrightarrow{\nu} G$ is
central if and only if $\tilde{\mathbb T}_{K/F} \to \mathbb T_{K/F}
\xrightarrow{\nu} G$ is central. 
\item A homomorphism $\tilde{\mathbb T}_{K/F} \to G$ is central if and
only if its composition with $\mu_v:\mathbb  G_m \to \tilde{\mathbb
T}_{K/F}$ is central for every place
$v$ of
$K$. 
\end{itemize}
\end{proof}

\subsection{Central extensions by tori} 

For any linear algebraic $F$-group $G'$ and central subgroup $Z$, there is an
obvious action of $B(F,Z)$ on $B(F,G')$. It comes from the actions discussed
just before Lemma \ref{lem.ZG'G}.

\begin{proposition}\label{prop.ZactB}
 Let 
\begin{equation}\label{SES.ZG'GB}
1 \to Z \xrightarrow{i} G' \xrightarrow{p} G \to 1
\end{equation}  
be a short exact sequence of linear algebraic $F$-groups in which $Z$ is a
torus that  is central in $G'$. Then the natural map 
\begin{equation}\label{eq.GG'ppB}
p:B(F,G') \to B(F,G) 
\end{equation}
is surjective. Moreover the map \eqref{eq.GG'ppB} induces a bijection 
between $B(F,G)$ and the quotient of
$B(F,G')$ by the action of $B(F,Z)$. Similarly \eqref{eq.GG'ppB} induces a
bijection  between $B(F,G)_{\bsc}$ and the quotient of
$B(F,G')_{\bsc}$ by the action of $B(F,Z)$.
\end{proposition}

\begin{proof}
The first two statements  follow from the analogous ones in
Proposition  \ref{prop.Zact}. The statement concerning basic elements uses
the additional  fact that an element  $b' \in B(F,G')$ is basic if and only
if its image $b$ in $B(F,G)$ is basic. Indeed, it is clear that $b$ is basic
if $b'$ is, and to prove the converse one just applies the next lemma to the
central torus $T$ in $G$ obtained as the image of the Newton homomorphism
$\mathbb D_F \to Z(G)$ for $b \in B(F,G)_{\bsc}$. 
\end{proof}

In the next lemma $k$ is an arbitrary  field. 
\begin{lemma}
 Let 
\begin{equation}\label{SES.ZG'GBTech}
1 \to Z \xrightarrow{i} G' \xrightarrow{p} G \to 1
\end{equation}  
be a short exact sequence of linear algebraic $k$-groups in which $Z$ is a
torus that  is central in $G'$. Let $T$ be a central torus in $G$. Then the
preimage $T'$ of $T$ under $p:G' \twoheadrightarrow G$ is a central torus in
$G'$. 
\end{lemma}

\begin{proof} 
 
This is clear when $G'$ is connected reductive, but the general case
requires an argument.  It follows from Lemma \ref{lem.TorExt} that $T'$ is a
torus. In proving the lemma it is harmless to assume that $k$ is
algebraically closed. Because $T$ is central in $G$, the commutator morphism
$G' \times T'
\to G'$ actually takes values in $Z \hookrightarrow G'$, and a simple
computation shows that the  morphism $G' \times T' \to Z$ (given by
$(g',t')\mapsto g't'g'^{-1}t'^{-1}$) is bimultiplicative. In other words we
have a
$Z$-valued pairing between $G'$ and $T'$, and, to show that $T'$ is central,
we must show that the pairing is trivial. 

This pairing can be viewed as a homomorphism $f$ 
from $G'$ to the constant group scheme over $k$ obtained from the abstract
abelian group $A:=\Hom(X^*(Z),X^*(T'))$. 
 For this we used Cor.~1.5 in  SGA
3, Tome II, Expos\'e VIII. Of course $A$ is a free abelian group of finite
rank, and is therefore torsion-free.  
 We just need to show that $f$ is trivial, and
this follows from the fact that $G'$ is a scheme of finite type over $k$.
Indeed, if $f^{-1}(a)$ were
nonempty for some nonzero $a \in A$, then $f^{-1}(na)$ ($n \in \mathbb Z$)
would be an infinite disjoint collection of nonempty open and
closed subsets of
$G'$, contradicting the fact that $G'$ is noetherian.  
\end{proof}

\section{The natural transformations $\kappa_G$ and $\bar\kappa_G$ for $B(F,G)$} 

In this section we  retain the notation of the previous one. We consider only 
connected reductive $F$-groups $G$. 

\subsection{A preliminary discussion of  $\kappa_G$}  

For a given  finite Galois extension $K/F$
Proposition \ref{prop.NatG} provides a functorial map 
\begin{equation}\label{eq.KappaG}
\kappa_G:H^1_{\alg}(\mathcal E(K/F), G(K)) \to (\Lambda_{G} \otimes
X(K))_{G(K/F)}
\end{equation}
for  any connected reductive $F$-group $G$
split by $K$. Our next task is to show that \eqref{eq.KappaG} 
is compatible with  inflation and then to define and study a map 
\begin{equation}\label{eq.BA}
\kappa_G:B(F,G) \to A(F,G)
\end{equation} 
obtained from \eqref{eq.KappaG} by passing to the limit over $K$.

\subsection{Compatibility of \eqref{eq.KappaG} with inflation} 
Once more let $L \supset K$ be two finite Galois extensions of $F$ in $\bar
F$.  
\begin{lemma}\label{lem.KappaInf} 
Let $G$ be a connected reductive group over $F$ that splits over $K$. Then
the diagram  
\begin{equation}
\begin{CD}
H^1_{\alg}(\mathcal E(K/F), G(K)) @>{\kappa_G}>> (\Lambda_{G} \otimes
X(K))_{G(K/F)} \\
@VVV @A{\id \otimes j}AA \\
H^1_{\alg}(\mathcal E(L/F), G(L)) @>{\kappa_G}>> (\Lambda_{G} \otimes
X(L))_{G(L/F)}
\end{CD}
\end{equation}
commutes and the right vertical arrow is an isomorphism. Here the left
vertical arrow is the  inflation map \eqref{eq.CurrentInf}. 
\end{lemma}
\begin{proof}
Lemma \ref{lem.pandjNew}(3) shows that the right vertical arrow is an
isomorphism, so we need only prove that the diagram commutes. In other
words we must show that $\kappa_G=\kappa_G'$, where $\kappa_G'$ is the map
defined by going the long way around the square. We prove that
$\kappa_G=\kappa_G'$ in three steps.  

When $G$ is a torus, we simply appeal to Lemmas \ref{lem.LocInf} 
and \ref{CD.GlobInflation}. When the derived group  $G_{\der}$ is simply
connected, we write $D$ for the torus obtained as the quotient
$G/G_{\der}$. 
Then the square 
\begin{equation}
\begin{CD}
H^1_{\alg}(\mathcal E(K/F), G(K)) @>{\kappa_G}>> (\Lambda_{G} \otimes
X(K))_{G(K/F)} \\
@VVV @| \\
H^1_{\alg}(\mathcal E(K/F), D(K)) @>{\kappa_D=\kappa_D'}>> (\Lambda_{D}
\otimes X(K))_{G(K/F)}
\end{CD}
\end{equation}
commutes, and the same is true with 
$\kappa_G$ replaced by $\kappa_G'$. Therefore $\kappa_G'=\kappa_G$ when
$G_{\der}$ is simply connected.  

In the general case we choose a $z$-extension $G'\twoheadrightarrow G$
whose kernel is a torus split by $K$.  Then the square 
\begin{equation}
\begin{CD}
H^1_{\alg}(\mathcal E(K/F), G'(K)) @>{\kappa_{G'}=\kappa'_{G'}}>>
(\Lambda_{G'}
\otimes X(K))_{G(K/F)} \\
@VVV @VVV \\
H^1_{\alg}(\mathcal E(K/F), G(K)) @>{\kappa_G}>> (\Lambda_{G}
\otimes X(K))_{G(K/F)}
\end{CD}
\end{equation}
commutes, and the same is true with 
$\kappa_G$ replaced by $\kappa_G'$. Using that the left vertical arrow is
surjective (Proposition \ref{prop.Zact}), we conclude that
$\kappa_G'=\kappa_G$.

\end{proof}

\subsection{Discussion of $C(G)$} 
In this section  
we are concerned only with connected reductive $F$-groups $G$. So 
the center $Z(G)$ is  an $F$-group of multiplicative type. The biggest
torus in $Z(G)$ will be denoted by $C(G)$; it is the subgroup of $Z(G)$
corresponding to the quotient of $X^*(Z(G))$ by its
torsion subgroup. The inclusion $C(G) \hookrightarrow G$ induces a  
natural injection  
\begin{equation}\label{eq.CGL}
\Lambda_{C(G)} \to \Lambda_G.
\end{equation}

\subsection{Discussion of $A(F,G)$}
Now consider any $\Gamma$-module $\Lambda$ on which some open subgroup of
$\Gamma$ acts trivially. Let $K \subset L$ be finite Galois extensions of
$F$ contained in $\bar F$.  
\begin{lemma}\label{lem.LambdaUpDown} 
Suppose that the action of $\Gamma$ on
$\Lambda$ factors through $G(K/F)$.  
Then the square 
\begin{equation}
\begin{CD} 
(\Lambda \otimes X(K))_{G(K/F)}
 @>{N_{K/F}}>> (\Lambda \otimes
X(K))^{G(K/F)} \\
 @A{\id \otimes j}AA @V{\id \otimes p}VV \\
 (\Lambda
\otimes X(L))_{G(L/F)} @>{N_{L/F}}>> (\Lambda
\otimes X(L))^{G(L/F)}
\end{CD}
\end{equation}
commutes, and the left vertical arrow is an isomorphism. 
\end{lemma}
\begin{proof}
This follows directly from Lemma \ref{lem.pandjNew}. 
\end{proof}

We are interested in the $\Gamma$-module $\Lambda_G$ with $G$
connected reductive over $F$. Let $\mathcal K$ be the set of finite Galois
extensions $K$ of $F$ in $\bar F$ such that the action of $\Gamma$ on
$\Lambda_G$ factors through $G(K/F)$. We put 
\begin{equation}
A(F,G):=\projlim_{K \in \mathcal K} \, (\Lambda_G \otimes X(K))_{G(K/F)},
\end{equation} 
where the transition maps are the isomorphisms $\id \otimes j$
appearing in the lemma above. Because these transition maps are
isomorphisms, we can equally well say that 
\begin{equation}\label{eq.Ndef1}
A(F,G) = \injlim_{K \in \mathcal K} \,(\Lambda_G \otimes X(K))_{G(K/F)},
\end{equation} 
where the transition maps are now the inverses of the isomorphisms $\id
\otimes j$ in the lemma. 

Recall that $\mathbb D_F$ is the $F$-group of multiplicative type whose
character group is 
\[
X^*(\mathbb D_F):=\injlim_{K \in \mathcal K} \, X(K)
\]
with transition maps  $p:X(K) \hookrightarrow X(L)$. 
 Observe that 
\begin{equation}\label{eq.Ndef2}
(\Lambda_G \otimes X^*(\mathbb D_F))^{\Gamma}=\injlim_{K \in \mathcal K} \,
(\Lambda_G
\otimes X(K))^{G(K/F)},
\end{equation}
and that, for any torus $T$, there is a canonical bijection   
\begin{equation}\label{eq.NewtT}
(\Lambda_T \otimes X^*(\mathbb D_F))^{\Gamma}=\Hom_F(\mathbb D_F,T). 
\end{equation}

\begin{definition}
For any connected reductive $F$-group $G$ we define 
\begin{equation}\label{eq.DefOfMapN} 
N:A(F,G) \to (\Lambda_G \otimes X^*(\mathbb D_F))^{\Gamma}
\end{equation}
to be the map resulting from taking the injective limit (over $K \in \mathcal
K$) of the norm maps
$N_{K/F}$ appearing in Lemma \ref{lem.LambdaUpDown}. Here we are using 
\eqref{eq.Ndef1}  and \eqref{eq.Ndef2} to view the source and
target of $N$ as injective limits.   
\end{definition} 

We will need the map $N$ in the next proposition. 

\subsection{The canonical map $\kappa_G:B(F,G) \to A(F,G)$} \label{sec.KappaGforB}
Lemma  \ref{lem.KappaInf} shows that, in the injective limit over 
$K \in \mathcal K$,  the natural transformations
$\kappa_G$ of  Proposition \ref{prop.NatG} fit together to give a natural
transformation 
\[
\kappa_G:B(F,G) \to A(F,G).  
\]

\subsection{The canonical map $\bar\kappa_G$} 
For each $K \in \mathcal K$ there is a map (see \eqref{eq.KappaZero})  
\begin{equation}\label{eq.KappaZeroAgain}
\bar\kappa_G:\bigl[\Hom_K(\mathbb D_{K/F},G)/G(K)\bigr]^{G(K/F)} \to (\Lambda_G \otimes X(K))^{G(K/F)}
\end{equation} 

We claim that the maps \eqref{eq.KappaZero} are compatible as $K$ 
 varies through $\mathcal K$. Indeed, 
using the usual procedure involving $z$-extensions, we reduce to the case of tori, 
for which the claim is obvious. So the maps \eqref{eq.KappaZeroAgain} fit together to 
give a map 
\begin{equation}\label{eq.KappaZeroBG}
\bar\kappa_G:\bigl[\Hom_{\bar F}(\mathbb D_F,G)/G(\bar F)\bigr]^{\Gamma} \to (\Lambda_G 
\otimes X^*(\mathbb D_F))^{\Gamma}. 
\end{equation}

\subsection{A relation between $\kappa_G(b)$ and  the  Newton point of  $b$}\label{sub.NewtonAndKappa} 

The two propositions in this subsection were  inspired by an  exchange of email  with  
T.~Kaletha and M.~Rapoport, who pointed out to me 
that such results might hold in the framework of this paper. 

The next proposition generalizes  part of Theorem 1.15 in \cite{RapRich}, for which a 
reference to \cite{IsoI} is given. The proof given here should make it clear which 
results from \cite{IsoI}  justify the relevant part of Theorem 1.15 of 
Rapoport-Richartz. The propositions make use of the map $N$ (see \eqref{eq.DefOfMapN}).

\begin{proposition}\label{prop.NewtKappaBG}
The square 
\begin{equation}
\begin{CD}
B(F,G) @>{\kappa_G}>> A(F,G)\\
@V{Newton}VV @VNVV \\
\bigl[\Hom_{\bar F}(\mathbb D_F,G)/G(\bar F)\bigr]^{\Gamma} @>{\bar\kappa_G}>> (\Lambda_G \otimes
X^*(\mathbb D_F))^{\Gamma}
\end{CD}
\end{equation}
commutes. 
\end{proposition} 

\begin{proof}
This follows from Lemma \ref{lem.NewtonKappaPrelim}. 
\end{proof}

Before stating the next proposition, we observe that the Newton point of a basic
element in $B(F,G)$ lies in $\Hom_F(\mathbb D_F,Z(G))=\Hom_F(\mathbb
D_F,C(G))$, a group that the isomorphism \eqref{eq.NewtT} identifies with 
$(\Lambda_{C(G)} \otimes X^*(\mathbb D_F))^{\Gamma}$. The proposition 
makes use of the inclusion  
$i:\Lambda_{C(G)}
\hookrightarrow \Lambda_G$ (see \eqref{eq.CGL}).  

\begin{proposition}\label{prop.NewtKappa}
The square 
\begin{equation}\label{CD.NewtKappaABC}
\begin{CD}
B(F,G)_{\bsc} @>{\kappa_G}>> A(F,G)\\
@V{Newton}VV @VNVV \\
(\Lambda_{C(G)} \otimes X^*(\mathbb D_F))^{\Gamma} @>i>> (\Lambda_G \otimes
X^*(\mathbb D_F))^{\Gamma}
\end{CD}
\end{equation}
commutes.  Moreover the bottom arrow in the square is injective, so the 
square lets us read off the Newton 
point of $b \in B(F,G)_{\bsc}$  from  $\kappa_G(b)$.  
\end{proposition} 

\begin{proof}

It follows easily from the previous proposition that the diagram commutes. 
Now we prove that the bottom arrow in the square is injective. We have already 
mentioned that $i:\Lambda_{C(G)} \to \Lambda_G$ is injective. Tensoring with the 
torsion-free abelian group $X^*(\mathbb D_F)$ preserves injectivity, and so does taking 
$\Gamma$-invariants. So the bottom arrow is indeed injective.  
\end{proof} 

\begin{remark}\label{rem.im=A0}
It is clear from Proposition  \ref{prop.NewtKappa} that 
\begin{equation}\label{eq.ImKappa} 
\im[B(F,G)_{\bsc}\xrightarrow{\kappa_G} A(F,G)] \subset A_0(F,G),
\end{equation} 
where $A_0(F,G)$ denotes the preimage under $N$ of the subset 
$(\Lambda_{C(G)} \otimes X^*(\mathbb D_F))^{\Gamma}$ of $(\Lambda_G \otimes
X^*(\mathbb D_F))^{\Gamma}$. 
In Propositions \ref{prop.LocImKap} and \ref{prop.ImKapGlob} it will be
shown that  the inclusion
\eqref{eq.ImKappa} is in fact an equality. 
\end{remark}

\subsection{Compatibility of $\kappa_G$ with localization}

In this subsection we consider a global field $F$, a place $u$ of $F$, and 
a connected reductive $F$-group $G$. For convenience we  
 fix an $F$-embedding $\bar F \to 
\bar F_u$ of separable closures of $F$ and $F_u$. 

There is a natural localization map 
\begin{equation}\label{eq.MapLocForA}
A(F,G) \to A(F_u,G), 
\end{equation}
defined as follows. 
Once again let  $\mathcal K$ be the set of finite Galois extensions $K$ of $F$ in 
$\bar F$ such that $\Gal(\bar F/K)$ acts trivially on $\Lambda_G$. 
For $K \in \mathcal K$ there is a natural homomorphism 
\begin{equation}\label{eq.MapLocForAA}
(\Lambda_G \otimes X(K))_{G(K/F)} \to (\Lambda_G)_{G(K_v/F_u)}, 
\end{equation}
where $v$ is the place of $K$ determined by our chosen embedding 
$\bar F \to \bar F_u$. When $G$ is a torus $T$ split by $K$, then $\Lambda_G$ is 
the cocharacter group $M$, and the map \eqref{eq.MapLocForAA} 
was defined in  subsection \ref{sub.LocForTorComm}. In general 
we define  \eqref{eq.MapLocForAA} in exactly the same way, simply 
replacing $M$ by $\Lambda_G$ everywhere. We then obtain 
\eqref{eq.MapLocForA} by taking the colimit over $\mathcal K$ 
(see \eqref{eq.Ndef1})  of the maps 
\eqref{eq.MapLocForAA}. 

We are now going to prove the following compatibility between $\kappa_G$ 
and localization. 

\begin{lemma}\label{lem.KapLocCompat}
For every connected reductive $F$-group $G$ the square  
\begin{equation}
\begin{CD}
B(F,G)  @>{\eqref{eq.BLCzero}}>> B(F_u,G) \\
@V{\kappa_G}VV @V{\kappa_G}VV \\
A(F,G)  @>{\eqref{eq.MapLocForA}}>> A(F_u,G)
\end{CD}
\end{equation}
commutes. 
\end{lemma} 

\begin{proof}
All four maps in the square are functorial in $G$.  Therefore we may reduce 
 to the case in 
which the derived group is simply connected, and from there to the case of  a torus. 
(The two reduction steps follow the same pattern as in the proofs of 
Lemma \ref{lem.KappaInf}  and Proposition \ref{prop.NewtKappa}.) 
Tori can be handled using Lemma \ref{lem.KappaTLocCompat}. 
\end{proof}

\section{A generalization of Shapiro's lemma}\label{sec.Shapiro} 
In this section we are going to prove a version of Shapiro's lemma 
for sets $H^1_Y(E,M)$ like the ones studied before, but
with $M$ now allowed to be nonabelian. Then we will give applications
involving 
$B(F,G)$, including a discussion of corestriction maps in the case of 
tori.

\subsection{Two definitions} 
The two definitions that follow are standard in 
the theory of nonabelian cohomology for a group $G$. 
\begin{itemize}
\item A \emph{$G$-group} $M$ is a group $M$ equipped with an action
of $G$ by automorphisms of $M$. 
\item A \emph{$G$-action of a $G$-group $M$ on a $G$-set $Y$} is an action
of 
$M$ on $Y$ such that the action map $M \times Y \to Y$ is $G$-equivariant. 
\end{itemize}
 Giving a $G$-action of a $G$-group $M$ on $Y$ is the same as giving an
action of
$M
\rtimes G$ on $Y$. 
 
\subsection{The set $H^1_Y(E,M)$ 
for nonabelian $G$-groups $M$}\label{sub.NonAbH1Y}

In subsection \ref{sub.H1Y} we defined sets $H^1_Y(E,M)$. Now we want
to generalize the definition by allowing $M$ to be nonabelian. As before
our starting point is an extension 
\[
1 \to A \to E \to G \to 1
\] 
of $G$ by a $G$-module $A$. We still insist that $A$ be abelian, but we are
going to consider an arbitrary (possibly nonabelian) $G$-group $M$, which
we also view as an $E$-group, with $A$ acting trivially.

We  regard $\Hom(A,M)$ as a
$G$-set in the usual way:  $\sigma \in G$ acts on $f\in
\Hom(A,M)$ by the rule  $(\sigma f)(a)=\sigma(f(\sigma^{-1}a))$. 
 There are natural $G$-actions  of the $G$-group $M$ on itself and on the
$G$-set
$\Hom(A,M)$. Equivalently, there are natural
actions of
$M
\rtimes G$ on both
$M$ and
$\Hom(A,M)$. These actions are spelled out in the following definition. 

\begin{definition}\label{def.2actions}
\hfill 
\begin{enumerate}
\item The $G$-group $M$ acts on itself by conjugation. The corresponding
action of
$M\rtimes G$ on $M$ is as follows: $m\sigma \in M\rtimes G$ transforms $m_1
\in M$ into $m\sigma(m_1)m^{-1}$. 
\item The $G$-group $M$ acts on the $G$-set $\Hom(A,M)$ through its action
on $M$. The corresponding action of $M\rtimes G$ is as follows: 
$m\sigma \in M\rtimes G$ transforms $f \in \Hom(A,M)$ 
 into $\Int(m) \circ \sigma(f)$. 
\end{enumerate} 
We will also use the canonical surjective homomorphism $M \rtimes E
\twoheadrightarrow M \rtimes G$ (given by the identity on $M$ and the
canonical surjection $E \twoheadrightarrow G$ on $E$) to make 
$M\rtimes E$ act on 
$M$ and
$\Hom(A,M)$. 
\end{definition}

In order to define $H^1_Y(E,M)$ we need two more ingredients, namely an 
$(M\rtimes G)$-set $Y$ and an 
$(M\rtimes G)$-map $\xi:Y \to \Hom(A,M)$. 
We require that $(Y,\xi)$ satisfy the following condition: 
\begin{equation}\label{eq.CondXiY}
\xi(y) (A) \subset M_y \text{ for all $y \in Y$}.
\end{equation}
Here we are writing $M_y$ for the stabilizer of $y$ in $M$. 

\begin{example}
Suppose that $M$ is abelian. As in subsection \ref{sub.H1Y}, let us 
consider a $G$-module $Y$ 
and $G$-module map $\xi:Y \to \Hom(A,M)$. Because $M$ is abelian, 
it acts  trivially on $\Hom(A,M)$, and so $\xi$ becomes an $(M \rtimes G)$-map 
if we make $M$ act trivially on $Y$. The condition \eqref{eq.CondXiY} is then 
automatically satisfied.  
\end{example}

\begin{example}
Consider a Galois gerb 
$
1 \to D(K) \to \mathcal E \to G(K/F) \to 1
$ 
(see subsection \ref{sub.RvGaGe}). 
 Let $G$ be a linear algebraic
group over $F$. Then  
\begin{itemize}
\item take $M=G(K)$, 
\item take $Y=\Hom_K(D,G)$, and  
\item take  $\xi$ to  be the natural map $\Hom_K(D,G) \to \Hom(D(K),G(K))$. 
\end{itemize} 
The  condition \eqref{eq.CondXiY} is 
automatically satisfied.  
\end{example} 

Returning to the general discussion, we now want to define $H^1_Y(E,M)$ in
such a way that it agrees with the previously (in subsection \ref{sub.H1Y})
defined  notion in the first
example  and with $H^1_{\alg}(\mathcal E,G(K))$ in the second example. It
is clear how to do this. We begin by defining suitable $1$-cocycles, the
set of which will be denoted by $Z^1_Y(E,M)$. By definition, an element in
$Z^1_Y(E,M)$ is a pair $(\nu,x)$ consisting of $\nu \in Y$ and $x \in
Z^1(E,M)$ satisfying the following two conditions: 
\begin{enumerate}
\item The restriction $x_0$ of $x$ to $A$ is the homomorphism $A \to M$
obtained as the image of $\nu$ under $\xi$. 
\item $x_w  \sigma(\nu)=\nu$ for any $w \in E$, with $\sigma$
denoting the image of $w$ under $E \twoheadrightarrow G$. 
\end{enumerate}

Use the $1$-cocycle $x$ to define a homomorphism $\varphi_x:E \to M\rtimes
E$  (thus $\varphi_x(w):=x_w w$ for all $w \in E$).  The $1$-cocycle
condition for
$x$ shows that the element $x_0 \in \Hom(A,M)$ is fixed by the subgroup
$\varphi_x(E)$ of
$M \rtimes E$. 
Observe that (2) can be reformulated as the condition that $\nu$ be
fixed by $\varphi_x(E) \subset M\rtimes E$, with $M \rtimes E$ acting on $Y$
through the canonical surjection $M \rtimes E \twoheadrightarrow M\rtimes G$
defined earlier.  

The group $M$ acts on  $Z^1_Y(E,M)$ in the obvious way (the
action of
$m
\in M$  sends 
$(\nu,x)$ to $(m \nu, w \mapsto mx_w w(m)^{-1})$), and
$H^1_Y(E,M)$ is by definition the quotient of $Z^1_Y(E,M)$ by the action
of $M$. 

\begin{remark} 
We now comment on the significance of the 
 condition  \eqref{eq.CondXiY} we  imposed on $(Y,\xi)$. 
Suppose for a moment that we did not impose it. 
We could still define $Z^1_Y(E,M)$ and $H^1_Y(E,M)$, in exactly 
the same way. Now, for any $\nu \in Y$ for which there exists 
$x \in Z^1(E,M)$ with $(\nu,x) \in Z^1_Y(E,M)$, the conditions 
 (1) and (2) would force the inclusion 
 $\xi(\nu) (A) \subset M_\nu$  to hold.   In other 
words, elements $\nu \in Y$ for which $\xi(\nu)(A)$ is not contained in $M_\nu$ 
are irrelevant when forming $H^1_Y(E,M)$. 
So  we lose nothing by imposing 
condition  \eqref{eq.CondXiY}, and in fact we even gain something,
 because doing so  will make 
  the  discussion of the maps $\Phi(f,g,\tilde h)$ in subsection \ref{sub.MapPhifgh} 
 a bit simpler. This is the main reason for imposing  \eqref{eq.CondXiY}. 
\end{remark}

There is an abstract Newton map 
\[
H^1_Y(E,M) \to (M\backslash Y)^G
\] 
in this context (induced by $(\nu,x) \mapsto \nu$), and it would be easy
enough to analyze its fibers (as we did before in the special case of
$H^1_{\alg}(\mathcal E,G(K))$. However, in proving a version of Shapiro's
lemma for
$H^1_Y(E,M)$, it is more useful to analyze the fibers of the  map 
\begin{equation}\label{eq.algabs} 
H^1_Y(E,M) \to H^1(E,M)
\end{equation}
induced by $(\nu,x) \mapsto x$. 

\subsection{Fibers of the map \eqref{eq.algabs}} 
Let $x \in Z^1(E,M)$, and let $[x]$ denote its class in $H^1(E,M)$. As above
we use
$x$ to define
$x_0 \in \Hom(A,M)$ and a homomorphism $\varphi_x:E \to M \rtimes E$. 
For any $(M \rtimes G)$-set $X$  we denote by $X_*$ the twisted $E$-set
obtained  by making $E$ act on
$X$ through the homomorphism 
\[
 E \xrightarrow{\varphi_x} M \rtimes E \twoheadrightarrow M \rtimes G. 
\]
The examples we have in mind for $X$ are   $M$ and 
$\Hom(A,M)$  (with the actions in  Definition \ref{def.2actions}), as well
as
$Y$.  We may view $\xi$ as an $E$-map 
$\xi:Y_* \to \Hom(A,M)_*$. Observe that $x_0$ lies in the fixed-point set
$(\Hom(A,M)_*)^E$, so the fiber $\xi^{-1}(x_0)$ is stable under the action
of $E$ on $Y_*$. 
\begin{lemma}\label{lem.FiberH1}
 The fiber of
\eqref{eq.algabs} over the class $[x]$ of 
$x$ is equal to the quotient set 
$M_x\backslash Y_x$, where  
\begin{itemize}
\item $M_x:=(M_*)^E$,
\item $Y_x:=(\xi^{-1}(x_0))^E \subset Y_*$. 
\end{itemize}
\end{lemma}
\begin{proof}
From the definition of $Z^1_Y(E,M)$, we see that 
the set of $\nu \in Y$ such that $(\nu,x) \in
Z^1_Y(E,M)$ is equal to $Y_x$. 
Two such pairs $(\nu,x),(\nu',x)$ are cohomologous if and only
if there exists $m\in M$ such that $m\nu=\nu'$ and $mx_ww(m)^{-1}=x_w$, and
the second of these equalities just says that $m$ is fixed by the
(twisted) action of $E$.  
\end{proof}

\subsection{Naturality of $H^1_Y(E,M)$ with respect to $(M,Y,\xi)$} 
 Let $(M',Y',\xi')$ be another triple like $(M,Y,\xi)$, and suppose we have 
a
$G$-homomorphism 
$f:M
\to M'$ and an
$(M\rtimes G)$-map
$g:Y \to Y'$ such that 
\begin{equation}
\begin{CD}
Y @>\xi>> \Hom(A,M) \\
@VgVV @VfVV \\
Y' @>\xi'>> \Hom(A,M')
\end{CD}
\end{equation}
commutes. Then there is an induced map 
\begin{equation}
H^1_Y(E,M) \to H^1_{Y'}(E,M')
\end{equation}
sending the class of $(\nu,x)$ to that of $(g(\nu),f(x))$. 

\subsection{Restriction maps for $H^1_Y(E,M)$} \label{sub.ResAbsNA}
Let $H$ be a subgroup of $G$, and form 
an extension 
\[
1 \to A \to E' \to H \to 1
\]
by taking $E'$ to be the preimage of $H$ under $E \twoheadrightarrow G$.
We then obtain a  restriction map 
\begin{equation}
\Res:H^1_Y(E,M) \to H^1_Y(E',M)
\end{equation}
by sending the class of $(\nu,x)$ to the class of $(\nu,x')$, where $x'$ is
the restriction of $x$ to the subgroup $E'$. 

\subsection{More forms of naturality}
The next three subsections will study more general forms of naturality 
in which $E$ is allowed to vary.  These will be needed in 
 section \ref{sec.Finite}.

\subsection{The map $\Phi(f,g,\tilde h)$}\label{sub.MapPhifgh} 
This subsection 
is  a generalization to nonabelian $M$ of  subsection \ref{sub.NatRespG}. 
All the maps in that subsection generalize easily, but, to keep things 
a little simpler, here we consider only the situation in which the homomorphism $\rho:G' \to G$ in 
subsection \ref{sub.NatRespG} is the identity map on $G$. This special case 
suffices for the needs of section \ref{sec.Finite}.

We consider $1 \to A \to E \to G \to 1$, a $G$-group $M$, 
an $(M \rtimes G)$-set $Y$ 
 and an $(M \rtimes G)$-map 
$\xi:Y \to \Hom(A,M)$ satisfying \eqref{eq.CondXiY}. We may then form 
the pointed set $H^1_Y(E,M)$.  
In addition we consider another such collection of objects: $1 \to A' \to E' \to G \to 1$, a $G$-group $M'$, 
an $(M' \rtimes G)$-set $Y'$ 
 and an $(M' \rtimes G)$-map 
$\xi':Y' \to \Hom(A',M')$ satisfying \eqref{eq.CondXiY}. We may then form 
the pointed set $H^1_{Y'}(E',M')$. 

Given some additional data  $f$, $g$, $\tilde h$, we are going 
to define a map 
\[
\Phi(f,g,\tilde h):H^1_Y(E,M) \to H^1_{Y'}(E',M').
\] 
These data are as follows: 
\begin{itemize}
\item a $G$-homomorphism $f:M  \to M'$, 
\item an $(M \rtimes G)$-map $g:Y \to Y'$, where we are using $f$ to view $Y'$ as $M$-set,
\item a homomorphism $\tilde h:E \to E'$ of extensions,
\end{itemize}
satisfying the requirement that the diagram 
\begin{equation}\label{CD.gXifh}
\begin{CD}
Y @>{\xi}>> \Hom(A,M) \\
@| @VfVV \\
Y @. \Hom(A,M') \\
@VgVV @AhAA \\
Y' @>{\xi'}>> \Hom(A',M')
\end{CD}
\end{equation}
commute, where $h$ is the unique map $A \to A'$ such that 
\begin{equation}\label{CD.ExtWithEandE'}
\begin{CD}
1 @>>> A @>>> E @>>> G @>>> 1 \\
@. @VhVV @V{\tilde h}VV @| @. \\
1 @>>> A' @>>> E' @>>> G @>>> 1
\end{CD}
\end{equation}  
commutes.  

We define $\Phi(f,g,\tilde h)$ to be the map 
 sending the class of $(\nu,x)$ to the class of 
$(g(\nu),x')$,  where $x'$ is the unique $1$-cocycle of $E'$ in $M'$ such that 
\begin{itemize}
\item the restriction of $x'$ to $A'$ is equal to  the map $\xi'(g(\nu)):A'
\to M'$, and 
\item 
the pullback of $x'$ to $E$ (via $\tilde h$) is equal to   $f(x)$. 
\end{itemize} 
In checking that $(g(\nu),x')$ satisfies condition (2) in 
the definition of $Z^1_{Y'}(E',M')$, one needs to use 
 \eqref{eq.CondXiY} (for $(Y',\xi')$) in addition to  the fact that $(\nu,x)$ 
satisfies (2). 

It is easy to see that the map $\Phi(f,g,\tilde h)$ depends only on the 
$A'$-conjugacy class of $\tilde h$. So, when $H^1(G,A')$ vanishes, 
the dependence of $\Phi(f,g,\tilde h)$ on $\tilde h$ is through $h$. 

When $H^1(G,A)$, $H^1(G,A')$ both vanish, the sets $H^1_Y(E,M)$,  
 $H^1_{Y'}(E',M')$ depend (up to canonical isomorphism) only on the 
cohomology classes $\alpha \in H^2(G,A)$, $\alpha' \in H^2(G,A')$ 
associated to $E$, $E'$, and, whenever we have $f,g,h$ such that 
\begin{itemize}
\item 
\eqref{CD.gXifh} commutes, and 
\item 
$h(\alpha)=\alpha'$,
\end{itemize}
we obtain a well-defined map 
\[
\Phi(f,g, h):H^1_Y(E,M) \to H^1_{Y'}(E',M')
\] 
by putting $\Phi(f,g, h):=\Phi(f,g,\tilde h)$ for any homomorphism $\tilde h$ making 
\ref{CD.ExtWithEandE'} commute. 

The next lemma concerns compositions of maps of type $\Phi$. We consider triples 
$(f_1,g_1,\tilde h_1)$ and $(f_2,g_2,\tilde h_2)$  such that 
$\Phi(f_1,g_1,\tilde h_1):H^1_Y(E,M) \to H^1_{Y'}(E',M')$ and $\Phi(f_2,g_2,\tilde h_2): H^1_{Y'}(E',M') \to  H^1_{Y''}(E'',M'')$ 
are defined.  
It is easy to check that 
the triple 
$(f_2 \circ f_1,g_2 \circ g_1,\tilde h_2 \circ \tilde h_1)$
satisfies the requirement needed in order to define the map 
$\Phi(f_2 \circ f_1,g_2 \circ g_1,\tilde h_2 \circ \tilde h_1)$.

\begin{lemma}\label{lem.CompPhi}
The composed map 
\[
H^1_Y(E,M) \xrightarrow{\Phi(f_1,g_1,\tilde h_1)} H^1_{Y'}(E',M') \xrightarrow{\Phi(f_2,g_2,\tilde h_2)} H^1_{Y''}(E'',M'')
\] 
is equal to $\Phi(f_2 \circ f_1,g_2 \circ g_1,\tilde h_2 \circ \tilde h_1)$.  
\end{lemma} 

\begin{proof}
Easy. 
\end{proof}

\subsection{The map $\Psi(g,\tilde h)$}\label{sub.DefnPsiMap}
We continue to  consider 
 $1 \to A \to E \to G \to 1$, $M$, 
 $Y$  and 
$\xi:Y \to \Hom(A,M)$, as well as  their primed versions.  
In this subsection, however, we 
make the further assumption that $M'$ coincides with $M$.

 Given some additional data   $g$, $\tilde p$, we are going 
to define a pullback map 
\[
\Psi(g,\tilde p):H^1_Y(E,M) \to H^1_{Y'}(E',M).
\] 
These data are as follows: 
\begin{itemize}
\item an $(M \rtimes G)$-map $g:Y \to Y'$, 
\item a homomorphism $\tilde p:E' \to E$ of extensions,
\end{itemize}
satisfying the requirement that the diagram 
\begin{equation}\label{CD.gXifhNew}
\begin{CD}
Y @>\xi>> \Hom(A,M) \\
@VgVV @VpVV \\
Y' @>{\xi'}>> \Hom(A',M)
\end{CD}
\end{equation}
commute, where $p$ is the unique map $A' \to A$ such that 
\begin{equation}\label{CD.ExtWithEandE'New}
\begin{CD}
1 @>>> A @>>> E @>>> G @>>> 1 \\
@. @ApAA @A{\tilde p}AA @| @. \\
1 @>>> A' @>>> E' @>>> G @>>> 1
\end{CD}
\end{equation}  
commutes.  

We define $\Psi(g,\tilde p)$ to be the map 
 sending the class of $(\nu,x)$ to the class of 
$(g(\nu),x')$,  where $x'$ is 
the pullback of $x$ to $E'$ (via $\tilde p$).   
It is easy to see that the map $\Psi(g,\tilde p)$ depends only on the 
$A$-conjugacy class of $\tilde p$;  checking this involves using 
\eqref{eq.CondXiY} for $(Y,\xi)$.  
So, when $H^1(G,A)$ vanishes, 
the dependence of $\Psi(g,\tilde p)$ on $\tilde p$ is through $p$.

When $H^1(G,A)$, $H^1(G,A')$ both vanish, the sets $H^1_Y(E,M)$,  
 $H^1_{Y'}(E',M)$ depend (up to canonical isomorphism) only on the 
cohomology classes $\alpha \in H^2(G,A)$, $\alpha' \in H^2(G,A')$ 
associated to $E$, $E'$,
 and, whenever we have $g,p$ such that 
\begin{itemize}
\item 
\eqref{CD.gXifhNew} commutes, and 
\item 
$p(\alpha')=\alpha$,
\end{itemize}
we obtain a well-defined map 
\[
\Psi(g, p):H^1_Y(E,M) \to H^1_{Y'}(E',M)
\] 
by putting $\Psi(g, p):=\Psi(g,\tilde p)$ for any homomorphism $\tilde p$ making 
\ref{CD.ExtWithEandE'New} commute. 

The next lemma concerns compositions of maps of type $\Psi$. We consider pairs 
$(g_1,\tilde p_1)$ and $(g_2,\tilde p_2)$  such that 
$\Psi(g_1,\tilde p_1):H^1_Y(E,M) \to H^1_{Y'}(E',M)$ and $\Psi(g_2,\tilde p_2): H^1_{Y'}(E',M) \to  H^1_{Y''}(E'',M)$ 
are defined. It is easy to check that 
$(g_2 \circ g_1,\tilde p_1 \circ \tilde p_2)$ 
is such that $\Psi(g_2 \circ g_1,\tilde p_1 \circ \tilde p_2)$ is defined.

\begin{lemma}\label{lem.CompPsi}
The composed map 
\[
H^1_Y(E,M) \xrightarrow{\Psi(g_1,\tilde p_1)} H^1_{Y'}(E',M) \xrightarrow{\Psi(g_2,\tilde p_2)} H^1_{Y''}(E'',M)
\] 
is equal to $\Psi(g_2 \circ g_1,\tilde p_1 \circ \tilde p_2)$.  
\end{lemma} 

\begin{proof}
Easy. 
\end{proof}

\subsection{A compatibility between maps of type $\Phi$ and $\Psi$} 
In the next lemma we suppose that we are given a commutative diagram 
\begin{equation}
\begin{CD}
 E @<{\tilde p}<< E_1   \\
 @V{\tilde h}VV @V{\tilde h_1}VV   \\
 E' @<{\tilde p'}<< E'_1 
\end{CD}
\end{equation} 
of extensions, 
a $G$-homomorphism $f:M \to M'$, and a commutative diagram 
\begin{equation}
\begin{CD}
 Y @>g'>> Y_1   \\
 @V{g}VV @V{g_1}VV   \\
Y' @>g''>> Y'_1 
\end{CD}
\end{equation} 
in which the top arrow is a map of $(M \rtimes G)$-sets, the bottom arrow 
is a map of $(M' \rtimes G)$-sets, and the two vertical arrows are maps 
of $(M \rtimes G)$-sets. We further assume that the triples 
$(f,g,\tilde h)$ and $(f,g_1,\tilde h_1)$ satisfy the requirements needed 
to define $\Phi(f,g,\tilde h)$ and $\Phi(f,g_1,\tilde h_1)$. Finally, we 
assume that the pairs $(g',\tilde p)$ and $(g'',\tilde p')$ satisfy the requirements 
needed to define  $\Psi(g',\tilde p)$ and $\Psi(g'',\tilde p')$. 

\begin{lemma}\label{lem.PhiPsi}
Under the assumptions above  the square 
\begin{equation}
\begin{CD}
H^1_Y( E,M) @>{\Psi(g',\tilde p)}>> H^1_{Y_1}( E_1,M)   \\
 @V{\Phi(f,g,\tilde h)}VV @VV{\Phi(f,g_1,\tilde h_1)}V   \\
H^1_{Y'}( E',M')  @>{\Psi(g'',\tilde p')}>> H^1_{Y'_1}(E'_1,M')
\end{CD}
\end{equation} 
commutes. 
\end{lemma} 

\begin{proof}
Easy. 
\end{proof}

\subsection{Coinduction for $H$-sets and $H$-groups}
Let $G$ be a group and $H$ a subgroup. The forgetful functor
from $G$-$\mathbf{Sets}$ to $H$-$\mathbf{Sets}$ has
a left adjoint $L$ and a  right adjoint $R$. It is customary to
refer to $L$ as induction and $R$ as coinduction.  The values of these two
functors on an $H$-set $Y$ are given by 
\begin{itemize}
\item $L(Y):=G\overset{H}{\times} Y$, 
\item $R(Y)=\{f:G \to Y:f(\tau\sigma)=\tau(f(\sigma)) \quad \forall \,
\sigma
\in G, \tau
\in H\}$. 
\end{itemize}

Here $G\overset{H}{\times} Y$ is  the quotient of $G\times Y$ by the
$H$-action
$\tau(\sigma,y):=(\sigma\tau^{-1},\tau y)$, and  $\sigma_1 \in G$ acts by
$\sigma_1(\sigma,y):=(\sigma_1\sigma,y)$.   The action of $\sigma_1$ on
$R(Y)$ is given by right-translation,
i.e.~$(\sigma_1f)(\sigma):=f(\sigma\sigma_1)$. The more useful of the two
adjunction morphisms for $R$ is the $H$-map
$\epsilon:R(Y) \to Y$ given by evaluation at the identity element of $G$.
Clearly $\epsilon$ restricts to a bijection 
\begin{equation}\label{eq.fpbij} 
R(Y)^G \to Y^H
\end{equation}
between fixed-point sets. 

Now suppose that $M$ is an $H$-group. Then of course $M$ is an $H$-set and
so we may form the
$G$-set $R(M)$. In fact $R(M)$ becomes a $G$-group for the group structure
given by pointwise multiplication of maps:
$(ff')(\sigma):=f(\sigma)f'(\sigma)$, the product on the right being taken
in the group $M$. The functor $R$ from
$H$-$\mathbf{Groups}$ to $G$-$\mathbf {Groups}$  is right adjoint to the
forgetful functor. When $G$ is a Galois group $G(K/F)$,  groups
coinduced from $H=G(K/E)$ appear naturally in the context of Weil
restriction of scalars from $E$ to $F$, where $E$ is an intermediate field
for $K/F$ that is finite over $F$.

\subsection{Shapiro's lemma for $H^1_Y(E,M)$} 
Again let $H$ be a subgroup of $G$ and form $E' \subset E$ as in  
 subsection \ref{sub.ResAbsNA}. Before we discuss Shapiro's lemma, we need
to analyze the following situation. Suppose we are given an $H$-action 
\begin{equation}\label{eq.HAct}
M\times X \to X
\end{equation}
of an $H$-group $M$ on an $H$-set $X$. Applying the functor
$R$ (of coinduction from $H$ to $G$) to the action map \eqref{eq.HAct}, 
we obtain a $G$-action 
\begin{equation}\label{eq.GAct}
R(M)\times R(X) \to R(X)
\end{equation}
of the $G$-group $R(M)$ on the $G$-set $R(X)$. (Here we used that
coinduction preserves products. Indeed, because it is a right adjoint, it
preserves all small limits.) In other words,
$R(X)$ is  an $(R(M) \rtimes G)$-set. The $H$-action of $M$ on $X$ can be
viewed as an action of $M \rtimes H$ on $X$, and the natural map $R(M)
\rtimes H
\to M
\rtimes H$ (given by $f\tau \mapsto \epsilon(f)\tau$)  lets us view $X$ as
an $(R(M)\rtimes H)$-set. 

Suppose further that we are given a $1$-cocycle $x$ of $E$ in $R(M)$. 
We then obtain a $1$-cocycle $y$ of $E'$ in $M$ by putting 
\[
y_{w'}:=\epsilon(x_{w'}). 
\]
The map $x \mapsto y$ on $1$-cocycles induces the classical Shapiro
isomorphism 
\[
H^1(E,R(M)) \to H^1(E',M). 
\] 
We use $x$ (resp.~$y$) to form a homomorphism $\varphi_x:E \to R(M)\rtimes
G$ (resp.~$\varphi_y:E' \to M\rtimes H)$. Using these homomorphisms, 
we obtain a twisted $E$-set $R(X)_*$ and a twisted $E'$-set $X_*$. 

\begin{lemma}\label{lem.ShAid}
The $E$-sets  $R(X)_*$ and
$R^E_{E'}(X_*)$ are canonically isomorphic. Here we are denoting
coinduction from $E'$ to $E$ by $R^E_{E'}$ in order to distinguish it from
coinduction from $H$ to $G$. Under this isomorphism $f_1 \in R(X)_*$
corresponds to $f_2 \in R^E_{E'}(X_*)$ when
$f_2(w)=\epsilon(x_w)f_1(\sigma)$, with $\sigma$ denoting the image of $w$
under $E \twoheadrightarrow G$. 
\end{lemma} 
\begin{proof}
Easy. 
\end{proof}

Now we are ready to tackle Shapiro's lemma. We start with a triple
$(M,Y,\xi)$ relevant to $E'$ rather than $E$. So $M$ is an $H$-group, $Y$
is an $H$-set equipped with an $H$-action of $M$, and $\xi:Y \to \Hom(A,M)$
is
$(M\rtimes H)$-equivariant.  We may therefore form the set 
$
H^1_Y(E',M).
$

Applying the functor $R$ of coinduction from $H$ to  $G$ to  
the map $\xi$, we obtain an $(R(M) \rtimes G)$-equivariant map 
\[
R(\xi):R(Y) \to R(\Hom(A,M)).
\]
Now observe that $R(\Hom(A,M))\simeq\Hom(A,R(M))$ as $(R(M) \rtimes
G)$-sets.  Here $f_1 \in R(\Hom(A,M))$ corresponds to $f_2 \in \Hom(A,R(M))$
when
\[
f_1(\sigma)(a)=f_2(\sigma^{-1}(a))(\sigma). 
\] 
Therefore we may equally well regard $R(\xi)$ as an $(R(M)\rtimes G)$-map 
\[
R(\xi):R(Y) \to \Hom(A,R(M))),
\]
and so we may form the set 
$
H^1_{R(Y)}(E,R(M)). 
$

We then have a restriction map 
\begin{equation}\label{eq.ShRes}
H^1_{R(Y)}(E,R(M)) \to H^1_{R(Y)}(E',R(M)).
\end{equation}
Moreover, naturality with respect to the commutative diagram 
\begin{equation*}
\begin{CD}
R(Y) @>R(\xi)>> R(\Hom(A,M)) \\
@V{\epsilon}VV @V{\epsilon}VV \\
Y @>{\xi}>> \Hom(A,M)
\end{CD}
\end{equation*}
provides us with a map 
\begin{equation}\label{eq.ShFunct}
 H^1_{R(Y)}(E',R(M)) \to H^1_Y(E',M).
\end{equation}

The next result is our generalized version of Shapiro's lemma.  
\begin{lemma}
The composed map 
\[
H^1_{R(Y)}(E,R(M)) \xrightarrow{\eqref{eq.ShRes}} H^1_{R(Y)}(E',R(M))
\xrightarrow{\eqref{eq.ShFunct}} H^1_Y(E',M)
\]
 is bijective. 
\end{lemma}
\begin{proof}
Consider the commutative square 
\begin{equation}
\begin{CD}
H^1_{R(Y)}(E,R(M)) @>>> H^1_Y(E',M) \\
@VVV @VVV \\
H^1(E,R(M)) @>>> H^1(E',M)
\end{CD}
\end{equation}
We must prove that the top arrow is bijective. Now the classical form of 
Shapiro's lemma asserts that the bottom arrow is bijective. So we are
reduced to proving the following. Fix $x \in Z^1(E,R(M))$ and let $y$ be
its  image under the cocycle-level Shapiro map
\[
Z^1(E,R(M)) \to Z^1(E',R(M)) \to Z^1(E',M),
\]
the first arrow being restriction from $E$ to $E'$, and the second being
the map induced by $\epsilon:R(M) \to M$. What we must prove is that the
top arrow restricts to a bijection from the fiber of the left arrow over 
$[x]$ to the fiber of the right arrow over $[y]$. 

These fibers were described in Lemma \ref{lem.FiberH1}.  The fiber 
on the right  is
$M_{y}\backslash Y_{y}$ and the one on the left is
$R(M)_x\backslash R(Y)_x$.    
   We write $x_0$ for the restriction of $x$ to $A$, and $y_0$ for the
restriction of $y$ to $A$. As usual we write $\varphi_x$ for the
homomorphism $E \to R(M)\rtimes E$ obtained from
$x$, and $\varphi_y$ for the
homomorphism $E' \to M\rtimes E'$ obtained from
$y$.  To prove the lemma we just need to produce (compatible) canonical 
bijections 
$M_y= R(M)_x$ and $Y_y= R(Y)_x$. For this we use Lemma \ref{lem.ShAid}. 

The first bijection is clear: we have 
\[
R(M)_x=(R(M)_*)^E=\bigl(R^E_{E'}(M_*)\bigr)^E=(M_*)^{E'}=M_{y}.  
\]
For the second one we begin by  applying the functor $R^E_{E'}$ to the
cartesian square  
\begin{equation}
\begin{CD}
\xi^{-1}(y_0) @>>> Y_* \\
@VVV @V{\xi}VV \\
\{y_0\} @>>> \Hom(A,M)_*
\end{CD} 
\end{equation}
of $E'$-sets. Since $R^E_{E'}$ is a right adjoint, it preserves cartesian
squares and final objects, and we conclude that 
the $E$-set obtained as the fiber of $R(\xi)$ over $x_0$ is coinduced 
from the
$E'$-set $\xi^{-1}(y_0)$. 
  It follows that 
\[
(R(\xi)^{-1}(x_0)_*)^E 
=(R^E_{E'}(\xi^{-1}(y_0)_*))^{E}=(\xi^{-1}(y_0)_* )^{E'}
\] 
and this is precisely the canonical bijection $R(Y)_x=Y_y$ we needed to
construct. So the lemma is proved. 
\end{proof} 

\subsection{Application to $H^1_{\alg}(\mathcal E,G(K))$} 
Let  
\[
1 \to D(K) \to \mathcal E \to G(K/F) \to 1
\]
be a Galois gerb for $K/F$. Let $E$ be an intermediate field for $K/F$, and
let $\mathcal E'$ be the preimage of $G(K/E)$ in $\mathcal E$. 

Let $G_0$ be a linear algebraic group over $E$ and put $G=R_{E/F}G_0$ (Weil
restriction of scalars). Then $G(K)=R(G_0(K))$ and
$\Hom_K(D,G)=R(\Hom_K(D,G_0))$, where $R$ denotes coinduction from $G(K/E)$
to $G(K/F)$. So there is a Shapiro isomorphism 
\begin{equation}
H^1_{\alg}(\mathcal E,G(K)) = H^1_{\alg}(\mathcal E',G_0(K)).
\end{equation}

\subsection{Application to $B(F,G)$}
Now let $F$ be a local or global field, and let $E/F$ be a finite separable
extension. Again consider $G=R_{E/F}(G_0)$ for some linear algebraic
$E$-group $G_0$. Then there is a Shapiro isomorphism 
\begin{equation}\label{eq.ShB}
B(F,G) = B(E,G_0).
\end{equation} 

\subsection{Application to $B_i(F,T)$ for $i=1,2,3$}
Let $E/F$ be a finite separable extension of global fields, let $T_0$ be a
torus over $E$, and put $T=R_{E/F}(T_0)$. Then for $i=1,2,3$ there are
Shapiro isomorphisms 
\begin{equation}\label{eq.ShBi}
B_i(F,T) = B_i(E,T_0). 
\end{equation} 
For $i=3$ this is just \eqref{eq.ShB} in different notation. (The groups 
$B_i(F,T)$
were defined in subsection \ref{sub.DefBiT}.)

\subsection{Corestriction and restriction for $B_i(F,T)$}
\label{sub.ResAndCor} 
Let $E/F$ be a finite separable extension of global
fields, and let $T$ be an $F$-torus. The Shapiro isomorphism makes it easy
to define corestriction maps for $T$. Put $\tilde T:=R_{E/F}(T)$. Because we
started with a torus $T$ over $F$ (not $E$), there is a norm map
$N_{E/F}:\tilde T \to T$. For $i=1,2,3$ we define a corestriction map 
\begin{equation}
\Cor:B_i(E,T) \to B_i(F,T)
\end{equation}
as the composed map 
\[
B_i(E,T) \overset{\eqref{eq.ShBi}}{=}B_i(F,\tilde T) \xrightarrow{N_{E/F}}
B_i(F,T). 
\]
\begin{lemma} Let $K/E$ be a finite
extension such that $K/F$ is Galois and $T$ is split by $K$.  Put
$Y_i(K):=X_*(T) \otimes X_i(K)$.  
\begin{enumerate}
\item 
For $i=1,2,3$  there is a commutative diagram 
\begin{equation*}
\begin{CD}
Y_i(K)_{G(K/E)} @>{\simeq}>> B_i(E,T) @>>> Y_i(K)^{G(K/E)}
\\ @VVV @V{\Cor}VV @VVV \\
Y_i(K)_{G(K/F)} @>{\simeq}>> B_i(F,T) @>>> Y_i(K)^{G(K/F)}
\end{CD}
\end{equation*}
The left vertical arrow is induced by the identity map on $Y_i(K)$. The
middle vertical arrow is corestriction for $E/F$. The right vertical arrow
is given by $y \mapsto \sum_{\sigma \in G(K/F)/G(K/E)}
\sigma(y)$. 
\item
For $i=1,2,3$  there is a commutative diagram 
\begin{equation*}
\begin{CD}
Y_i(K)_{G(K/E)} @>{\simeq}>> B_i(E,T) @>>> Y_i(K)^{G(K/E)}
\\ @AAA @A{\Res}AA @AAA \\
Y_i(K)_{G(K/F)} @>{\simeq}>> B_i(F,T) @>>> Y_i(K)^{G(K/F)}
\end{CD}
\end{equation*} 
The left vertical arrow is  given
by $y \mapsto \sum_{\sigma \in G(K/E)\backslash G(K/F)}
\sigma(y)$. 
 The middle vertical arrow is restriction for $E/F$. The right
vertical arrow is induced by the identity map on $Y_i(K)$. 
\end{enumerate}
\end{lemma}
\begin{proof}
Both parts of the lemma  can be proved in the same way as 
Lemma
\ref{lem.Trick}. The outer rectangles and right squares clearly commute.
Therefore the left squares also commute when $X_*(T)$ is free as
$\mathbb Z[G(K/F)]$-module. The general case is then reduced to this
special one by choosing $T' \to T$ with $X_*(T') \to X_*(T)$ surjective and
$X_*(T')$ free as $\mathbb Z[G(K/F)]$-module. 
\end{proof}

\subsection{Corestriction and restriction for $B(F,T)$ when $F$ is local} 
Let $E/F$ be a finite separable extension of local fields, and let $T$ be
an $F$-torus. As in the global case we define a corestriction map as the
composed map 
\[
B(E,T) =B(F,\tilde T) \xrightarrow{N_{E/F}}
B(F,T). 
\]
\begin{lemma} \label{lem.LocalCorRes}
 Let $K/E$ be a finite
extension such that $K/F$ is Galois and $T$ is split by $K$.  Put
$Y:=X_*(T)$. 
\begin{enumerate}
\item 
There is a commutative diagram 
\begin{equation*}
\begin{CD}
Y_{G(K/E)} @>{\simeq}>> B(E,T) @>>> Y^{G(K/E)}
\\ @VVV @V{\Cor}VV @VVV \\
Y_{G(K/F)} @>{\simeq}>> B(F,T) @>>> Y^{G(K/F)}
\end{CD}
\end{equation*}
The left vertical arrow is induced by the identity map on $Y$. The
middle vertical arrow is corestriction for $E/F$. The right vertical arrow
is given by $y \mapsto \sum_{\sigma \in G(K/F)/G(K/E)}
\sigma(y)$. 
\item
There is a commutative diagram 
\begin{equation*}
\begin{CD}
Y_{G(K/E)} @>{\simeq}>> B(E,T) @>>> Y^{G(K/E)}
\\ @AAA @A{\Res}AA @AAA \\
Y_{G(K/F)} @>{\simeq}>> B(F,T) @>>> Y^{G(K/F)}
\end{CD}
\end{equation*} 
The left vertical arrow is  given
by $y \mapsto \sum_{\sigma \in G(K/E)\backslash G(K/F)}
\sigma(y)$. 
 The middle vertical arrow is restriction for $E/F$. The right
vertical arrow is induced by the identity map on $Y$. 
\end{enumerate}
\end{lemma}
\begin{proof}
Same as in global case. 
\end{proof}

\section{$B(F,G)_{\bsc}$ in the local case}\label{sec.BFGlocal} 
\subsection{Notation} 
Let $F$ be a local field. We fix a separable closure $\bar F$ of $F$ and
put $\Gamma:=\Gal(\bar F/F)$. Let $G$ be a connected reductive $F$-group. We
are going to study $B(F,G)_{\bsc}$ (see section \ref{sec.DiscSetBFG})
 and the map $\kappa_G:B(F,G)_{\bsc} \to
A(F,G)=(\Lambda_G)_{\Gamma}$  
(see section \ref{sec.KappaGforB}). As in  sections \ref{sec.DiscSetBFG}
and \ref{sec.KappaGforB}  
 we write $Z(G)$ for the center of $G$, and $C(G)$ for the biggest torus in $Z(G)$.

\subsection{The case of tori}  From 
Lemma \ref{lem.LocInf} it follows  that, for every  $F$-torus $T$,
the map  
\[
\kappa_T:B(F,T) \to A(F,T)=(X_*(T))_{\Gamma}
\] 
is an isomorphism. 
In this simple case the colimit defining $B(F,T)$ is already attained when
$K$ is big enough to split $T$. 

\subsection{$B(F,G)$ in the nonarchimedean case} 
Assume that $F$ is  nonarchimedean. For any linear algebraic
group
$G$ over $F$, there is a canonical identification of  $B(F,G)$  with the set
denoted by $\mathbf B(G)$ in \cite{IsoII}. Strictly speaking \cite{IsoII}
treats only the $p$-adic case, but the definition of $\mathbf B(G)$ given
there makes sense for all nonarchimedean $F$.  

\begin{proposition}\label{prop.BscLocalMain}
Let $G$ be a  connected reductive $F$-group. Then the following statements
hold. 
\begin{enumerate}
\item The map $\kappa_G:B(F,G) \to A(F,G)$ restricts to a bijection 
\begin{equation}\label{eq.LocalBijection}
\kappa_G:B(F,G)_{\bsc}  \to A(F,G)
\end{equation}
\item If $T$ is an elliptic maximal $F$-torus in $G$, then the natural map 
\[
B(F,T) \to B(F,G)_{\bsc}
\] 
is surjective. 
\end{enumerate}
\end{proposition} 

\begin{proof}
First we prove that \eqref{eq.LocalBijection} is injective. 
When the derived
group of $G$ is simply connected, this follows easily from the vanishing of
$H^1$ for simply connected semisimple groups (due to Kneser \cite{KnI,KnII}
in the
$p$-adic case  and  Bruhat-Tits \cite{BT} in general). The general case is 
then 
treated using $z$-extensions and Proposition \ref{prop.ZactB}. The reader
who finds these indications too brief can look ahead to 
the proof of Proposition \ref{prop.BAMain}, where the corresponding steps are
treated in much greater detail. 

Next we prove part (2) of the proposition. The image of the natural map
$B(F,T) \to B(F,G)$ is contained in the subset $B(F,G)_{\bsc}$, simply 
because
$T$ is elliptic.  (The natural injection $\Hom_F(\mathbb D_F,Z(G))
\hookrightarrow \Hom_F(\mathbb D_F,T)$ is actually bijective, since the
image of any $F$-homomorphism $\mathbb D_F \to T$ is a split subtorus of
$T$.) The functoriality of
$\kappa_G$ guarantees that the diagram 
\begin{equation}
\begin{CD}
B(F,T) @>{\kappa_T}>> (\Lambda_T)_{\Gamma} \\ 
@VVV @VVV \\
B(F,G)_{\bsc} @>{\eqref{eq.LocalBijection}}>> (\Lambda_G)_{\Gamma}
\end{CD}
\end{equation} 
commutes. Since $\kappa_T$ is an isomorphism, 
the surjectivity of the left vertical map  follows from that of the right
vertical map and the (already established) injectivity of \eqref{eq.LocalBijection}. 

Finally we recall that  elliptic maximal $F$-tori
$T$ in $G$ are known to exist. This is due to Kneser \cite[\S 15]{KnII} in the
$p$-adic case and
DeBacker \cite{DeB} in general.  The surjectivity of
\eqref{eq.LocalBijection} now follows from part (2). 
\end{proof}

\subsection{$B(F,G)$ and $B(F,G)_{\bsc}$ in the complex
case}\label{sub.Cmplx} 
The complex case is very simple:  the map
\eqref{eq.SlopeMap} is bijective, which just says that $B(\mathbb C,G)$ is 
the set of
$G(\mathbb C)$-conjugacy classes of homomorphisms from $\mathbb G_m$ to
$G$. In particular we have 
\[
B(\mathbb C,G)_{\bsc}=\Lambda_{C(G)}. 
\]

\subsection{$B(F,G)_{\bsc}$ in the real case} We can analyze $B(\mathbb
R,G)_{\bsc}$ using some results of Shelstad
\cite{Sh}.  We choose a fundamental maximal $\mathbb R$-torus $T$ in $G$. We
write
$\Omega$ for its absolute Weyl group,  $\Omega(\mathbb R)$ for the fixed
points of complex conjugation on $\Omega$, and $\Omega_{\mathbb R}$ for the
subgroup of
$\Omega(\mathbb R)$ consisting of elements that can be represented by an
element in the normalizer of $T$ in $G(\mathbb R)$. As Shelstad shows,
\begin{itemize}
\item $T$ transfers to every inner form of $G$, and  
\item   there is a natural bijection 
\begin{equation}\label{eq.Shel}
\Omega_{\mathbb R}\backslash \Omega(\mathbb R) \xrightarrow{\simeq} 
\ker[H^1(\mathbb R, T) \to H^1(\mathbb R,G)] 
\end{equation}
\end{itemize}

Borovoi \cite[Theorem 1]{B1} observes that Shelstad's results  lead to a 
useful description of $H^1(\mathbb R,G)$. For this Borovoi uses the
following (right) action of $\Omega(\mathbb R)$ on $H^1(\mathbb R, T)$.
Given an element $\omega \in\Omega(\mathbb R)$ and a $1$-cocycle $t$
of
$\Gal(\mathbb C/\mathbb R)$ in $T$, the action of $\omega$ sends 
the class of $t$ to the class of the $1$-cocycle $t'$ given
by 
$t'_\sigma=\dot\omega^{-1}t_\sigma\sigma(\dot\omega)$, where $\dot\omega$ is
a representative for $\omega$ in the normalizer of
$T$ in $G(\mathbb C)$. Obviously the $\Omega(\mathbb R)$-orbit of the class
of $t$ is equal to the quotient $\Omega^t_{\mathbb R}\backslash
\Omega(\mathbb R)$, where $\Omega^t$ denotes the twist of $\Omega$ by $t$.
(So $\Omega^t$ is the Weyl group of $T$ in the pure inner form $G^t$ of $G$ 
obtained as the twist  by 
$t$.) Now \eqref{eq.Shel}, applied to the inner form $G^t$, implies (by the
usual twisting argument in Galois cohomology) that the fiber of 
$H^1(\mathbb R,T) \to H^1(\mathbb R,G)$ through $t$ is equal to the 
$\Omega(\mathbb R)$-orbit of $t$. 
 (When we twist, $\Omega_{\mathbb
R}$ changes, but
$\Omega(\mathbb R)$ does not.) Moreover, the fact that $T$ transfers 
to every inner form of $G$ implies (see, e.g., \cite{EST}) that
$H^1(\mathbb R,T) \to H^1(\mathbb R,G)$ is surjective. Putting these
observations together, Borovoi concludes that $H^1(\mathbb R,G)$ is the
quotient of $H^1(\mathbb R,T)$ by the above action of $\Omega(\mathbb R)$. 

We are now going to follow the  same line of reasoning to describe
$B(\mathbb R,G)_{\bsc}$ in terms of the subset $B(\mathbb R,T)_{G-\bsc}$
of $B(\mathbb R,T)$ consisting of all elements whose Newton point
$\nu:\mathbb G_m \to T$ is central in $G$. There is a natural action of
$\Omega(\mathbb R)$ on $B(\mathbb R,T)$, induced by the following action on
algebraic $1$-cocycles. Let $\omega \in \Omega(\mathbb R)$, and choose a
representative $\dot\omega$ of $\omega$ in the normalizer of $T$ in
$G(\mathbb C)$. Let $b=(\nu,x)$ be an algebraic $1$-cocycle in $T$. Then the
action of $\omega$ sends the class of $b$ to the class of the algebraic
$1$-cocycle $b':=(\omega^{-1}(\nu),w \mapsto \dot\omega^{-1}x_w
w(\dot\omega))$.  When $b$ is basic, so that $\nu$ is central,
$\omega^{-1}(\nu)$ is of course equal to $\nu$. In particular, the action of
$\Omega(\mathbb R)$ preserves the subset $B(\mathbb R,T)_{G-\bsc}$ of
$B(\mathbb R,T)$.   

\begin{lemma}\label{lem.BscReal}
The natural map $B(\mathbb R,T)_{G-\bsc} \to B(\mathbb R,G)_{\bsc}$ induces
a bijection between $B(\mathbb R,G)_{\bsc}$ and the quotient of 
$B(\mathbb R,T)_{G-\bsc}$ by the action of $\Omega(\mathbb R)$. 
\end{lemma}
\begin{proof} We claim that $B(\mathbb R,T)_{G-\bsc} \to B(\mathbb
R,G)_{\bsc}$ is surjective. Indeed, consider an element $b \in B(\mathbb
R,G)_{\bsc}$. Its image in $B(\mathbb R,G_{\ad})_{\bsc}=H^1(\mathbb R,G_{\ad})$ 
lies in the image of $H^1(\mathbb R,T_{\ad})$, because 
$H^1(\mathbb R,T_{\ad}) \to H^1(\mathbb R,G_{\ad})$ is surjective. 
Therefore $b$ can be represented by an algebraic $1$-cocycle $(\nu,x)$ 
for which the image of $x$ in the adjoint group takes values in $T_{\ad}$. 
It follows that $x$ itself takes values in $T$. This, together with the 
fact that $\nu$ is central in $G$, 
 shows that $(\nu,x)$ is 
the image of an algebraic $1$-cocycle in $T$, and the claim follows.

It remains to examine the fibers of our surjection.  Just as for Galois
cohomology, the fiber of 
$B(\mathbb R,T)_{G-\bsc} \twoheadrightarrow B(\mathbb R,G)_{\bsc}$ through
the class in 
$B(\mathbb R,T)_{G-\bsc}$ represented by the algebraic $1$-cocycle
$b=(\nu,x)$ in $T$ (with $\nu$ central in $G$) can be identified with the
kernel of 
\[
B(\mathbb R,T)_{G-\bsc} \twoheadrightarrow B(\mathbb R,J_b)_{\bsc}, 
\] where
$J_b$ is the inner form of $G$ obtained as the twist by $b$. Now the
kernel of $B(\mathbb R,T) \to B(\mathbb R,J_b)$ is equal to the kernel of 
$H^1(\mathbb R,T) \to H^1(\mathbb R,J_b)$, because $\Hom(\mathbb G_m,T) \to
\Hom(\mathbb G_m,J_b)$ is obviously injective. Therefore, by 
 the second
of Shelstad's results reviewed above,  the fiber of 
$B(\mathbb R,T)_{G-\bsc} \to B(\mathbb R,G)_{\bsc}$ through the class of
$b$ can be identified with $\Omega^b_{\mathbb R}\backslash
\Omega(\mathbb R)$, where $\Omega^b$ is the Weyl group of $T$ in the twist
$J_b$. Unwinding the definitions, one sees that $\Omega^b_{\mathbb R}$ is
the stabilizer in $\Omega(\mathbb R)$ of the class of $b$, and we conclude
that the fiber of 
$B(\mathbb R,T)_{G-\bsc} \to B(\mathbb R,G)_{\bsc}$ through the class of
$b$ is the $\Omega(\mathbb R)$-orbit of the class of $b$, as desired. 
\end{proof}

\begin{remark}
When the fundamental torus $T$ is  elliptic (equivalently, when elliptic
maximal tori  exist in $G$), any
$F$-homomorphism
$\mathbb G_m \to T$ is automatically central in $G$, so $B(\mathbb
R,T)_{G-\bsc}=B(\mathbb R,T)$, and the lemma tells us that $B(\mathbb
R,G)_{\bsc}$ is the quotient of $B(\mathbb R,T)$ by the above action of
$\Omega(\mathbb R)$. 
\end{remark} 

\subsection{The image of $B(F,G)_{\bsc} \to A(F,G)$}
The next result involves the subset $A_0(F,G)$ of $A(F,G)$ defined in Remark
\ref{rem.im=A0}. 
\begin{proposition}\label{prop.LocImKap}
For any local field $F$ the image of $\kappa_G:B(F,G)_{\bsc} \to A(F,G)$ is 
$A_0(F,G)$.
\end{proposition}

\begin{proof} We already know from Remark \ref{rem.im=A0} 
that the image of $B(F,G)_{\bsc} \to A(F,G)$ is contained in $A_0(F,G)$. 
So we just need to check that any element in $A_0(F,G)$ lies in 
$\im[B(F,G)_{\bsc} \to A(F,G)]$. In the nonarchimedean case  this is clear
from Proposition \ref{prop.BscLocalMain}(1), and in the complex case it is
clear from  the fact that $B(\mathbb C,G)_{\bsc}=\Lambda_{C(G)}$. The
real case is more interesting. 

Before tackling the real case we need to make a definition. We say that 
a commutative diagram 
\begin{equation}
\begin{CD}
Z @>>> X\\
@VVV @VVV \\
Y @>>> S
\end{CD}
\end{equation} 
of sets is \emph{semicartesian} if the induced map $Z \to X \times_S Y$ is surjective. 

 In the real case the commutative square \eqref{CD.NewtKappaABC} works out to
\begin{equation}
\begin{CD}
B(\mathbb R,G)_{\bsc} @>{\kappa_G}>> (\Lambda_G)_{\Gamma}\\
@V{Newton}VV @V{N_{\mathbb C/\mathbb R}}VV \\
(\Lambda_{C(G)}) ^{\Gamma} @>i>> (\Lambda_G)^{\Gamma}.
\end{CD}
\end{equation} 
To prove the proposition we must prove that this square is
semicartesian.  Let $T$ be a fundamental maximal $\mathbb R$-torus in $G$. The
map 
$B(\mathbb R,T)_{G-\bsc} \to B(\mathbb R,G)_{\bsc}$ is surjective 
by  Lemma \ref{lem.BscReal}.
Moreover, the map 
$i:(\Lambda_{C(G)})^{\Gamma} \to (\Lambda_G)^{\Gamma}$ factors
as 
$(\Lambda_{C(G)})^{\Gamma} \hookrightarrow (\Lambda_T)^{\Gamma} 
\xrightarrow{p}
(\Lambda_G)^{\Gamma}$, where $p$ is (induced by) the canonical surjection in
the short exact sequence 
\begin{equation}
0 \to X_*(T_{\ssc}) \to \Lambda_T \xrightarrow{p} \Lambda_G \to
0.
\end{equation}

 So,   to prove the proposition, it will suffice to
show that the square 
\begin{equation}
\begin{CD}
(\Lambda_T)_{\Gamma} @>p>> (\Lambda_G)_{\Gamma}\\
@V{N_{\mathbb C/\mathbb R}}VV @V{N_{\mathbb C/\mathbb R}}VV \\
(\Lambda_T) ^{\Gamma} @>p>> (\Lambda_G)^{\Gamma}
\end{CD}
\end{equation}
is semicartesian.

Consider $g \in (\Lambda_G)_{\Gamma}$ and $t \in (\Lambda_T) ^{\Gamma}$
such that $N_{\mathbb C/\mathbb R}(g)=p(t)$. We need to construct $\tilde t
\in (\Lambda_T) _{\Gamma}$ such that $p(\tilde t)=g$ and  $N_{\mathbb
C/\mathbb R}(\tilde t)=t$. We begin by choosing any $t_1 \in 
(\Lambda_T) _{\Gamma}$ such that $p(t_1)=g$. Then $y:=t-N_{\mathbb
C/\mathbb R}(t_1)$ lies in the kernel $(X_*(T_{\ssc}))^{\Gamma}$ of
the bottom horizontal arrow in our square. To finish the proof it suffices 
to construct an element $x \in (X_*(T_{\ssc}))_{\Gamma}$ such that 
$N_{\mathbb C/\mathbb R}(x)=y$, since we will then obtain the desired
element $\tilde t$ as the sum $t_1+x$. 

The existence of $x$ (for arbitrary $y$) is just the statement that the Tate
cohomology group $H^0(\Gamma,X_*(T_{\ssc}))$ vanishes. This is indeed the
case, because  $T_{\ssc}$
is isomorphic to a product $T_a \times T_i$, with $T_a$  anisotropic and
$T_i$  of the form $R_{\mathbb C/\mathbb R}(S)$ for some $\mathbb C$-torus
$S$ (see, e.g., the proof of Lemma 10.4 in \cite{EST}). 
\end{proof}

\section{A finiteness theorem} \label{sec.Finite} 

\subsection{Motivation} 

Let $K/F$ be a finite Galois extension of global fields. As usual we write
$V_F$ for the set of all places of $F$. When $u$ is a finite place, we write
$\mathcal O_u$ for the valuation ring in $F_u$. 
 For every subset $S$ of $V_F$ we denote by $S_K$ the preimage
of $S$ under the natural surjection $V_K \twoheadrightarrow V_F$. When $S$
contains $S_{\infty}$, the set of infinite places of $F$, we put 
\begin{align*}
F_S:&=\{x \in F: x \in \mathcal O_u \quad \forall \, u \in V_F \setminus S \}, \\
K_S:&=\{x \in K: x \in \mathcal O_v \quad \forall \, v \in V_K \setminus S_K \}, \\
\mathbb A_{K,S}:&=\{x \in \mathbb A_K: x_v \in \mathcal O_v \quad \forall \, v \in V_K \setminus S_K \}.
\end{align*}
If $S=V_F$, then $F_S=F$. In the number field case, if $S=S_{\infty}$, then
$F_S$ is the ring of integers in $F$.

Now let $G$ be a linear algebraic group over $F$. By way of motivation for
this section we begin by reviewing a  standard finiteness result for
$H^1(G(K/F),G(K))$. To formulate the result we first need to choose an
extension of $G$ to a smooth affine group scheme $\mathcal G$ 
over $F_{S(\mathcal G)}$, where ${S(\mathcal G)}$ is some finite set of
places containing $S_{\infty}$. Given two such extensions $\mathcal G_1$,
$\mathcal G_2$, there exists a finite set $S$ of places such that 
\begin{itemize}
\item $S$ contains both $S(\mathcal G_1)$ and $S(\mathcal G_2)$, and 
\item the identity morphism for $G$ extends (uniquely) 
 to an $F_S$-isomorphism between $\mathcal G_1$, $\mathcal G_2$.
\end{itemize}
For such a set $S$ we have $\mathcal G_1(\mathcal O_v)=\mathcal G_2(\mathcal
O_v)$ when $v \notin S_K$. 

Here is the standard finiteness result, along with its easy proof.   
When dealing with localization maps (as in the next result), we make the following 
convention: $u$ denotes a place of $F$, and $v$ denotes some chosen place of $K$ 
lying over $u$. This will allow us to keep our statements a little more succinct.

\begin{proposition}\label{prop.TotalLocal}
Let $x \in H^1(G(K/F),G(K))$ and write $x_u$ for the image of $x$ under the
localization map 
\[
H^1(G(K/F),G(K)) \to H^1(G(K_v/F_u),G(K_v)).
\] 
Then there exists a finite set $S$ of places of $F$ such that 
\begin{itemize}
\item $S$ contains
$S(\mathcal G)$, and 
\item for all $u \notin S$ the element $x_u$ lies in the image of the map 
\[
H^1(G(K_v/F_u),\mathcal G(\mathcal O_v)) \to H^1(G(K_v/F_u),G(K_v)).
\]  
\end{itemize} 
\end{proposition} 
\begin{proof} Choose a cocycle representing $x$. 
Because $G(K/F)$ is finite, there exists finite $S \supset {S(\mathcal G)}$
such that 
$x$ takes values in $\mathcal G(K_S)$. Clearly $S$ does the job. 
\end{proof} 

\subsection{Goal of this section} We are going to generalize the finiteness
result from $H^1(G(K/F),G(K))$ to the bigger set $H^1_{\alg}(\mathcal
E_3(K/F),G(K))$. 
For $u \notin S(\mathcal G)$ we define 
\begin{equation}\label{eqGOGK}
H^1(G(K_v/F_u),\mathcal G(\mathcal O_v)) \to 
H^1_{\alg}(\mathcal E(K_v/F_u),G(K_v)) 
\end{equation} 
as the composition of the natural map (induced by $\mathcal G(\mathcal O_v) \hookrightarrow G(K_v)$)
\begin{equation*}
H^1(G(K_v/F_u),\mathcal G(\mathcal O_v)) \to H^1(G(K_v/F_u),G(K_v)) 
\end{equation*} 
 and 
the canonical  inclusion 
\begin{equation*}
 H^1(G(K_v/F_u),G(K_v)) 
\hookrightarrow H^1_{\alg}(\mathcal E(K_v/F_u),G(K_v)).
\end{equation*}

\begin{proposition}\label{prop.BigDiag} 
Let $b \in H^1_{\alg}(\mathcal E_3(K/F),G(K))$ and write $b_u$ for the image
of
$b$ under the localization map 
\[
\ell^F_u:H^1_{\alg}(\mathcal E_3(K/F),G(K)) 
\xrightarrow{\eqref{eq.Loc}} H^1_{\alg}(\mathcal E(K_v/F_u),G(K_v)). 
\]
Then there exists a finite set $S$ of places of $F$ such that 
\begin{itemize}
\item $S$ contains
$S(\mathcal G)$, and 
\item for all $u \notin S$ the  element $b_u$ lies in the image of
the  map \eqref{eqGOGK}.
 \end{itemize} 
\end{proposition} 

Before proving the proposition, we make
note of a simple corollary. 
\begin{corollary} \label{cor.ResDirProd}
Assume that $G$ is connected. 
Let $b \in H^1_{\alg}(\mathcal E_3(K/F),G(K))$ and again write $b_u$ for the image
of $b$ under the localization map 
$\ell^F_u$.  
Then there exists a finite set $S$ of places of $F$ such that 
$b_u$ is trivial   
for all $u \notin S$. 
\end{corollary}

\begin{proof} For any finite place $u$ of $F$ we denote by $\kappa(u)$ the
residue field of the valuation ring $\mathcal O_u$.  From Proposition 3.7 in
Expos\'e
$VI_B$ of SGA 3, Tome I  it follows easily that there exists a finite set
$S$ of places with
$S \supset S(\mathcal G)$ and such that $\mathcal G \otimes \kappa(u)$ is
connected for all $u \notin S$. Enlarging $S$ if need be, we may also  assume 
that $K/F$ is unramified outside $S$. Then a standard
argument involving Hensel's Lemma and Lang's Theorem shows that  
 $H^1(G(K_v/F_u),\mathcal G(\mathcal O_v)) $ is trivial for all $u \notin
S$. From this it is clear that  
Corollary \ref{cor.ResDirProd} does follow from 
Proposition \ref{prop.BigDiag}.   
\end{proof}

The next two subsections will provide lemmas that will be used in the proof of Proposition 
\ref{prop.BigDiag}.

\subsection{A vanishing theorem} 
 In this subsection the linear algebraic group $G$ will not
appear, so we will temporarily lighten our notation by using $G$ as an
abbreviation for the Galois group $G(K/F)$. Moreover we denote the
decomposition group at a place $w$ of $K$ simply by $G_w$. (So $G_w$ is the
stabilizer of $w$ in $G$.)

For any set $S$  of places of $F$    we consider the short exact sequence 
\begin{equation}
0 \to X_3(S) \to X_2(S) \to X_1(S) \to 0 \tag{$X(S)$}
\end{equation}
with  $X_1(S):=\mathbb Z$, $X_2(S):=\mathbb Z[S_K]$ and $X_3(S)=\mathbb
Z[S_K]_0$. (When we worked with this sequence before, in subsection 
\ref{sub.GlobTNT},   we did not include 
$S$ in the notation. Now we do, because we need to keep track of which set 
$S$ we are using.) 
Dual to this short exact sequence of $G$-modules is a short exact
sequence 
\[
1 \to \mathbb G_m \to \tilde{\mathbb T}_S \to \mathbb T_S \to 1
\]
of protori over $F$ (that split over $K$).

\begin{lemma}\label{lem.VanS}
Let $S$ be a set of places of $F$ such that 
\[
\{G_w:w \in S_K\}=\{G_w:w \in V_K\}.
\] 
Then the following statements hold. 
\begin{enumerate}
\item The group $H^1(G',\mathbb T_S(K))$ vanishes for every subgroup $G'$ of $G$. 
\item For every place $v$ of $K$ the group $H^1(G_v,\mathbb T_S(K_v))$
vanishes. 
\end{enumerate}
\end{lemma} 
\begin{proof}
  (1) Notice that our hypothesis on $S$ implies  that 
\begin{equation}\label{eq.2SetsWeq}
\{G'_w:w \in S_K\}=\{G'_w:w \in V_K\}.
\end{equation}
 We are going to use the long exact sequence of Tate cohomology for the 
short exact sequence 
\[
1 \to \mathbb G_m(K) \to \tilde{\mathbb T}_S(K) \to \mathbb T_S(K) \to 1
\] 
of $G'$-modules. From Lemma \ref{lem.TateX2} it follows that 
\[
H^r(G',\tilde{\mathbb T}_S(K))=\prod_{v \in G'\backslash S_K}
H^r(G'_v,K^\times).
\] 
So $H^1(G',\tilde{\mathbb T}_S(K))$ vanishes by Hilbert's Theorem 90. To
prove that $H^1(G',\mathbb T_S(K))$ vanishes, we just need to prove that 
the natural map 
\[
H^2(G',K^\times) \to \prod_{v \in G'\backslash S_K}
H^2(G'_v,K^\times)  
\] 
is injective.  By \eqref{eq.2SetsWeq} it is equivalent to prove 
that 
\[
H^2(G',K^\times) \to \prod_{v \in G'\backslash V_K}
H^2(G'_v,K^\times)  
\] 
is injective, 
and this follows from the well-known injectivity of 
\[
H^2(G',K^\times) \to \prod_{v \in G'\backslash V_K}
H^2(G'_v,K_v^\times).  
\] 

(2) We are going to use the long exact sequence of Tate cohomology for the 
short exact sequence 
\[
1 \to \mathbb G_m(K_v) \to \tilde{\mathbb T}_S(K_v) \to \mathbb T_S(K_v) \to
1
\] 
of $G_v$-modules. From Lemma \ref{lem.TateX2} we have  
\[
H^r(G_v,\tilde{\mathbb T}_S(K_v))=\prod_{w \in G_{v}\backslash S_K}
H^r(G_{v,w},K_v^\times),
\] 
where $G_{v,w}=G_v \cap G_w$. 
So $H^1(G_v,\tilde{\mathbb T}_S(K_v))$ vanishes by Hilbert's Theorem 90. To
prove that $H^1(G_v,\mathbb T_S(K_v))$ vanishes, we just need to prove that 
the natural map 
\[
H^2(G_v,K_v^\times) \to \prod_{w \in G_v\backslash S_K}
H^2(G_{v,w},K_v^\times)  
\] 
is injective. This is clear because (by our hypothesis on $S$) there exists
$w \in S_K$ such that $G_w=G_v$, and for this $w$ we have $G_{v,w}=G_v$. 
\end{proof} 

\begin{remark}
There exist finite subsets $S$ of $V_F$ satisfying the hypothesis of the
last lemma. This is obvious,  because the Galois group $G$ is finite and
therefore has only finitely many subgroups.  
\end{remark}

\subsection{Comparison of two cohomology classes} In this subsection the linear
algebraic group $G$ again does not appear, so we continue to use $G$ to
denote $G(K/F)$. In this subsection $S$ denotes any set of places of $F$ 
satisfying the three conditions  in subsection
\ref{sub.TNnot}. As in subsection \ref{sub.GlobTNT}, using
$S$-adeles we obtain  another short exact sequence of
$G$-modules, namely 
\begin{equation}
1 \to A_3(S) \to A_2(S) \to A_1(S) \to 1. \tag{$A(S)$}
\end{equation}
 (Before we did not  include $S$ in  the notation.) Recall that 
\[
A_2(S)=\mathbb A_{K,S}^\times=\bigl(\prod_{v \notin S_K} \mathcal O_v^\times \bigr) \times \bigl(
\prod_{v \in S_K} K_v^\times \bigr), \qquad A_3(S)=K_S^\times,  
\] 
and our assumptions on $S$ imply that $A_1(S)$ is the group of idele classes
of $K$. 

We denote by $\alpha(S)$ Tate's canonical class in $H^2(G,\Hom(X(S),A(S)))$.  
Its projections under $\pi_i:\Hom(X(S),A(S)) \to
\Hom(X_i(S),A_i(S))$ will be denoted by $\alpha_i(S)$.  When $S=V_F$, we
write $X$, $ X_i$, $A$, $A_i$,  $\alpha$, $\alpha_i$, $\mathbb T$,  $\tilde{\mathbb T}$
instead of
$X(V_F)$, $X_i(V_F)$, $A(V_F)$, 
$A_i(V_F)$, $\alpha(V_F)$, $\alpha_i(V_F)$, $\mathbb T_{V_F}$, $\tilde{\mathbb T}_{V_F}$. 
 
There is an obvious morphism $p^S$ from the exact sequence $X(S)$ to the exact sequence $X=X(V_K)$. 
This morphism is given by the vertical maps in the 
commutative diagram 
 \begin{equation*}
\begin{CD}
0 @>>> X_3(S) @>>> X_2(S) @>>> X_1(S) @>>> 0 \\
@. @V{p^S_3}VV @V{p^S_2}VV @V{p^S_1}VV @. \\
0 @>>> X_3 @>>> X_2 @>>> X_1 @>>> 0.
 \end{CD}
\end{equation*} 
The maps $p^S_2$  and $p^S_3$ are induced by the inclusion $S_K \subset V_K$, and 
$p^S_1$ is the identity map on $X_1(S)=\mathbb Z=X_1$. 

Dual to this is another commutative diagram
 \begin{equation*}
\begin{CD}
1 @>>> \mathbb G_m @>>> \tilde{\mathbb T} @>>> \mathbb T @>>> 1 \\
@. @| @V{p^S_2}VV @V{p^S_3}VV @. \\
1 @>>> \mathbb G_m @>>> \tilde{\mathbb T}_S @>>> \mathbb T_S @>>> 1.
 \end{CD}
\end{equation*} 
with exact rows.

Similarly there is an obvious inclusion morphism $k^S:A(S) \hookrightarrow A$, given by the 
vertical inclusions in the  
commutative diagram 
  \begin{equation*}
\begin{CD}
1 @>>> A_3(S) @>>> A_2(S) @>>> A_1(S) @>>> 1 \\
@. @V{k^S_3}VV @V{k^S_2}VV @V{k^S_1}VV @. \\
1 @>>> A_3 @>>> A_2 @>>> A_1 @>>> 1.
 \end{CD}
\end{equation*} 
Concretely, these vertical inclusions are (reading from left to right) 
$K_S^\times  \hookrightarrow K^\times$, 
$\mathbb A_{K,S}^\times  \hookrightarrow \mathbb A_K^\times$, and 
$\mathbb A_{K,S}^\times /K_S^\times  \hookrightarrow \mathbb A_K^\times/K^\times$. 
Observe that the last of these inclusions is actually an isomorphism by virtue of one of the 
conditions imposed on $S$ in \ref{sub.TNnot}. 

The purpose of the next lemma is to compare the Tate classes $\alpha(S)$ and $\alpha$. 
 The relevant groups, maps and Tate classes are shown in the following diagram. 
  \begin{equation*}
\begin{CD}
\alpha(S) \quad \Hom(X(S), A(S)) @>{k^S}>> \Hom(X(S),A) @<{p^S}<< \Hom(X,A) \quad \alpha  \\
 @V{\pi_i}VV @V{\pi_i}VV @V{\pi_i}VV  \\
\alpha_i(S) \quad \Hom(X_i(S), A_i(S)) @>{k^S_i}>> \Hom(X_i(S),A_i) @<{p^S_i}<< \Hom(X_i,A_i) \quad \alpha_i .
 \end{CD}
\end{equation*} 
The proximity of $\alpha(S)$ to $\Hom(X(S),A(S))$ is meant as a reminder that 
$\alpha(S)$ lies in $H^2(G,\Hom(X(S), A(S)))$, and so on. 

\begin{lemma}\label{lem.kSpS}
There are equalities 
\begin{align}
k^S(\alpha(S))&=p^S(\alpha), \label{FirstEq} \\
 k_i^S(\alpha_i(S))&=p_i^S(\alpha_i) \quad (i=1,2,3). \label{SecondEq}
\end{align}
\end{lemma}

\begin{proof} 

Because of the way the Tate classes are defined, we need to  prove 
\eqref{SecondEq} for $i=1,2$, then deduce \eqref{FirstEq}, and from that 
obtain \eqref{SecondEq} for $i=3$.

It is clear that \eqref{SecondEq} holds for $i=1$, because both $\alpha_1(S)$ 
and $\alpha_1$ are the fundamental class in
$\mathbb A_{K,S}^\times /K_S^\times  \simeq \mathbb A_K^\times/K^\times$. 

Next we show that \eqref{SecondEq} holds for $i=2$. We must check that the 
elements $ k_2^S(\alpha_2(S))$ and  $p_2^S(\alpha_2)$ in $H^2(G,\Hom(X_2(S),A_2))$ are equal. 
This follows from Lemma \ref{lem.TateX2}, since, for all $v \in  S_K$, both $ k_2^S(\alpha_2(S))$ and  $p_2^S(\alpha_2)$ 
have the same image in $H^2(G_v,A_2)$, namely the image of the local fundamental class under 
$K_v^\times \hookrightarrow \mathbb A_K^\times$. 

Now we prove the equality \eqref{FirstEq}. Applying $\pi_i$ to this equality (with $i=1,2$) we obtain 
the equalities \eqref{SecondEq}  for $i=1,2$, and these have already been verified. To prove 
\eqref{FirstEq} it remains only to prove that the  map 
\[
(\pi_1,\pi_2):H^2(G,\Hom(X(S),A)) \to H^2(G,\Hom(X_1(S),A_1)) \oplus H^2(G,\Hom(X_2(S),A_2))
\]
is injective. We use  essentially the same reasoning as in Tate's 
proof of Lemma \ref{lem.cart}; the reader can also consult our proof that the square \eqref{CD.car2} is cartesian. 
The desired injectivity follows from the vanishing of 
$H^1(G,\Hom(X_2(S),A_1))$, and this  is an easy consequence of  
  Lemmas  \ref{lem.TateX2} and  \ref{lem.H1Van}.

The equality \eqref{SecondEq} for $i=3$ is obtained by applying $\pi_3$ to  \eqref{FirstEq}. 
\end{proof}

\subsection{Proof of Proposition \ref{prop.BigDiag}} 

 We begin with a definition. 
We say that a subset $S$ of $V_F$ is\emph{ adequate} if it satisfies the following 
list of conditions: 
\begin{itemize}
\item
$S$ is finite. 
\item 
$S$ contains $S(\mathcal G)$ (and consequently all archimedean places). 
\item $S$ contains all finite places that ramify in $K$. 
\item For every intermediate field $E$ of $K/F$,  every ideal class of $E$
contains an ideal with support in the preimage $S_E$ of
$S$ under $V_E \twoheadrightarrow V_F$. 
\item $S$ satisfies the hypothesis of Lemma \ref{lem.VanS}, 
i.e. 
\[
\{G(K/F)_w:w \in S_K\}=\{G(K/F)_w:w \in V_K\}. 
\]
\end{itemize}
It is clear that, if $S$ is adequate and $S'$ is a finite set of places with $S' \supset S$, then 
$S'$ is adequate. It is also clear that adequate sets $S$ exist. 

When $S$ is adequate, the first and third conditions in the list above imply 
that the 
protorus $\mathbb T_S$ is actually a torus, and that it  
extends uniquely to a torus $\mathcal T_S$ over $F_S$. 
We will soon need the following vanishing theorem. 

\begin{lemma}\label{lem.LangHensTorus}
Let $u$ be a place of $F$ outside $S$, and let $v$ be a place of $K$
lying over~$u$. Then 
$H^r(G(K_v/F_u),\mathcal T_S(\mathcal O_v))=0$ for all $r \in \mathbb Z$. 
\end{lemma}

\begin{proof}
This  is a standard consequence of Hensel's lemma and Lang's Theorem. 
\end{proof}

For adequate $S$ we are now going to construct a pointed set 
$H^1_{\alg}(\mathcal E_3(S),\mathcal G(K_S))$ together with canonical map 
\begin{equation}\label{eq.MapKSG}
H^1_{\alg}(\mathcal E_3(S),\mathcal G(K_S)) \to H^1_{\alg}(\mathcal E_3(K/F),G(K)).
\end{equation}
Here $\mathcal E_3(S)$ denotes an extension 
\[
1 \to \mathcal T_S(K_S) \to \mathcal E_3(S) \to G(K/F) \to 1
\]
with corresponding cohomology class $\alpha_3(S)$. 
(In section \ref{sec.GTN} the set $S$ was fixed and this extension was 
denoted 
 simply by $\mathcal E_3$, 
but now we need to keep track of $S$.)  

The definition of the set $H^1_{\alg}(\mathcal E_3(S),\mathcal G(K_S))$ 
involves  an obvious extension of our usual notion of 
$H^1_{\alg}$ to a situation involving a Galois 
extension of rings (namely $K_S/F_S$) rather than fields.   
 More precisely  $H^1_{\alg}(\mathcal E_3(S),\mathcal G(K_S))$ is 
defined to be  the set 
$H^1_Y(\mathcal E_3(S),\mathcal G(K_S))$ from subsection \ref{sub.NonAbH1Y} 
formed using  $Y=\Hom_{K_S}(\mathcal T_S,\mathcal G)$ and 
the obvious map $\xi:Y \to \Hom(\mathcal T_S(K_S),\mathcal G(K_S))$. 
So an element in $Z^1_{\alg}(\mathcal E_3(S),\mathcal G(K_S))$ 
is a pair $(\nu,x)$ consisting of 
a $K_S$-homomorphism $\nu:\mathcal T_S \to \mathcal G$ and a 
$1$-cocycle $x$ of $\mathcal E_3(S)$ in $\mathcal G(K_S)$ satisfying 
the  two   conditions imposed in  subsection \ref{sub.NonAbH1Y}.

The map \eqref{eq.MapKSG} will be defined as a composition 
\begin{equation}\label{eq.MapKSGComp}
H^1_{\alg}(\mathcal E_3(S),\mathcal G(K_S)) \xrightarrow{BC} H^1_{\alg}(\mathcal E_3^K(S), G(K))
 \xrightarrow{p^*}H^1_{\alg}(\mathcal E_3(K/F),G(K)).
\end{equation}
Here $\mathcal E_3^K(S)$ is the Galois gerb (for $K/F$) defined by pushing out 
$\mathcal E_3(S)$ along the  map $\mathcal T_S(K_S) \to \mathbb T_S(K)$ induced 
by the inclusion of $K_S$ in $K$.  The map $BC$ appearing in \eqref{eq.MapKSGComp} 
is the obvious base change map: on the level of cocycles it  
sends $(\nu,x)$ to $(\nu,x')$, where $x'$ is the unique element of $Z^1(\mathcal E_3^K(S),G(K))$ 
that agrees with $\nu$ on $\mathbb T_S(K)$ and with $x$ 
 on $\mathcal E_3(S) \hookrightarrow \mathcal E_3^K(S)$. 
(So the map $BC$ is essentially an instance of \eqref{eq.AlgInf}, though we are now working 
with a slightly extended notion of $H^1_{\alg}$.) 

It remains to define the map $p^*$ appearing in \eqref{eq.MapKSGComp}.
Taking $i=3$ in the second equality in  Lemma \ref{lem.kSpS}, we see that the image of $\alpha_3(S)$ under 
\[
H^2(G(K/F),\mathcal T_S(K_S)) \to H^2(G(K/F),\mathbb T_S(K))
\] 
is equal to the image of $\alpha_3$ under 
\[
p^S_3:H^2(G(K/F),\mathbb T(K)) \to H^2(G(K/F),\mathbb T_S(K)).
\] 
Therefore there exists ${\tilde p}^S_3$ making the diagram 
 \begin{equation}
\begin{CD}
1 @>>> \mathbb T(K) @>>> \mathcal E_3(K/F) @>>> G(K/F) @>>> 1 \\
@. @V{p^S_3}VV @V{{\tilde p}^S_3}VV @| @. \\
1 @>>> \mathbb T_S(K) @>>> \mathcal E_3^K(S) @>>> G(K/F) @>>> 1
\end{CD}
\end{equation} 
commute. 
To lighten the notation we are now going to abbreviate $p^S_3$ to $p$. This should cause 
no confusion since we 
will have no further use for $p^S$, $p^S_1$, $p^S_2$. Similarly 
we will  abbreviate ${\tilde p}^S_3$ to $\tilde p$. 

As in subsection \ref{sub.ChBnd} $\tilde p$ induces a map 
\[
{\tilde p}^*:H^1_{\alg}(\mathcal E_3^K(S),G(K)) \to H^1_{\alg}(\mathcal E_3(K/F),G(K))  
\] 
that, by Lemma \ref{lem.VanS}(1),  is independent of the choice of $\tilde p$ extending $p$. 
Therefore we further lighten our notation by writing $p^*$ in place of ${\tilde p}^*$; this is the second 
map appearing in \eqref{eq.MapKSGComp}. 

The next lemma is the first step towards proving Proposition \ref{prop.BigDiag}. 
\begin{lemma}\label{lem.NeededForFinProp} 
Let $b \in H^1_{\alg}(\mathcal E_3(K/F),G(K))$. Then 
 there exists an adequate  subset $S$ of 
$V_F$ such that  $b$
 lies in the image of the map \eqref{eq.MapKSG}. 
\end{lemma}

\begin{proof}
 The  proof is similar to that of the
standard finiteness theorem in global Galois cohomology that we reviewed
earlier.  We choose an algebraic $1$-cocycle $(\nu,x)$ representing $b$. 
Thus $\nu$ is a $K$-homomorphism from $\mathbb T$ to $G$. 
Now $G$ is of
finite type over $F$ and $\mathbb T$ is the protorus obtained from the
projective system $\mathbb T_S$ of tori (with $S$ varying over all finite
subsets of $V_F$). So there exists an  adequate set $S$ such that $\nu$ comes from a
$K$-homomorphism  $\nu:\mathbb T_S \to G$. At this point we have refined
$(\nu,x)$ to an algebraic $1$-cocycle of $\mathcal E_3^K(S)$ in $G(K)$.  

Enlarging $S$, we may assume that
 $\nu$ is defined over $K_S$. The
restriction of the $1$-cocycle $x$ to $\mathcal T_S(K_S)$ agrees with $\nu$
and therefore takes values in $\mathcal G(K_S)$. Since $\mathcal T_S(K_S)$ is
of finite index in $\mathcal E_3(S)$, by enlarging $S$ further we may assume
that the restriction of $x$ to $\mathcal E_3(S) \hookrightarrow \mathcal
E_3^K(S)$ takes values in $\mathcal G(K_S)$. At this point we have refined
$b$ to an algebraic $1$-cocycle of $\mathcal E_3(S)$ in $\mathcal G(K_S)$,
and we are done with the proof of the lemma.  
\end{proof}

To deduce Proposition \ref{prop.BigDiag} from the lemma 
we just proved, it clearly suffices to construct, 
  for every  adequate set $S$  and every place $u \notin S$,  
 a localization map $\ell_u^S$ making the square  
 \begin{equation}\label{eq.SimpleFinSquare}
\begin{CD}
H^1_{\alg}(\mathcal E_3(S),\mathcal G(K_S))  @>{\eqref{eq.MapKSG}}>> 
H^1_{\alg}(\mathcal E_3(K/F),G(K)) \\
 @V{\ell_u^S}VV  @V{\ell^F_u}VV \\
H^1(G(K_v/F_u),\mathcal G(\mathcal O_v)) @>{\eqref{eqGOGK}}>>
H^1_{\alg}(\mathcal E(K_v/F_u),G(K_v)).
\end{CD}
\end{equation} 
commute. Moreover, it is enough to do this in the special case 
when $G(K_v/F_u)=G(K/F)$. Indeed, we may easily reduce to this 
special case by making use of the fixed field $E$ of $G(K_v/F_u)$ and  
 the restriction maps from the two sets in the top row of 
\eqref{eq.SimpleFinSquare} to their analogs for $K/E$. So, for the rest of this 
section we assume that $G(K_v/F_u)=G(K/F)$.

It remains only to explain how to construct $\ell_u^S$ making the 
square \eqref{eq.SimpleFinSquare} commute. To accomplish this 
we are going to construct a 
big commutative diagram 
 \begin{equation}\label{CD.NewMass}
\begin{CD}
H^1_{\alg}(\mathcal E_3(S),\mathcal G(K_S)) @>BC>> H^1_{\alg}(\mathcal
E^K_3(S),G(K)) @>p^*>> H^1_{\alg}(\mathcal E_3(K/F),G(K)) \\
 @V{\Loc}VV @V{\Loc}VV @V{\Loc}VV \\
H^1_{\alg}(\mathcal E_3^{\mathcal O_v}(S),\mathcal G(\mathcal O_v)) @>BC>>
H^1_{\alg}(\mathcal E^{K_v}_3(S),G(K_v)) @>p^*>> H^1_{\alg}(\mathcal
E^v_3(K/F),G(K_v))\\ 
@V{\mu_0^*}VV @V{\mu_0^*}VV @V{\mu_v'^*}VV\\
H^1(G(K_v/F_u),\mathcal G(\mathcal O_v)) @>>> H^1(G(K_v/F_u),G(K_v)) @>>>
H^1_{\alg}(\mathcal E(K_v/F_u),G(K_v)).
\end{CD}
\end{equation}  
The top row of the diagram is \eqref{eq.MapKSGComp}. 
 The composition of the two vertical maps at the right end is the
localization map $\ell^F_u$ (see section \ref{sec.Loc}).  
The two bottom horizontal arrows are the ones we composed to obtain 
the map \eqref{eqGOGK}. So, once we have constructed this big commutative 
diagram, the composition of the two vertical maps at its left end
will yield the desired map $\ell_u^S$.

There are only two sets in the diagram that have not yet been defined, namely, 
the two at the left end of the second row. These two sets bear the same 
relation to the ones above them as the third set in the second row does to 
the one above it. More precisely, for each of the two rings $R=\mathcal O_v, K_v$, 
we write $\mathcal E_3^R(S)$ for the extension of $G(K_v/F_u)$ obtained 
from $\mathcal E_3(S)$ by pushing forward along $\mathcal T_S(K_S) \to 
\mathcal T_S(R)$.   Now $\mathcal E_3^{K_v}(S)$ is 
a Galois gerb for $K_v/F_u$, so the set  $H^1_{\alg}(\mathcal E^{K_v}_3(S),G(K_v))$ 
is defined. We define $H^1_{\alg}(\mathcal E_3^{\mathcal O_v}(S),\mathcal G(\mathcal O_v))$ in the obvious way,   
as the set $H^1_Y(\mathcal E_3^{\mathcal O_v}(S),\mathcal G(\mathcal O_v))$ 
(see subsection \ref{sub.NonAbH1Y}) obtained from 
 $Y=\Hom_{\mathcal O_v}(\mathcal T_S,\mathcal G)$ together with 
the obvious map $\xi:Y \to\Hom(\mathcal T_S(\mathcal O_v),\mathcal G(\mathcal O_v))$.

As a first step towards constructing all the maps in 
the big commutative diagram above, we consider the
 commutative diagram 
\begin{equation}\label{CD.Tow2Big}
\begin{CD}
\mathcal T_S(K_S) @>>> \mathbb T_S(K) @<p<< \mathbb T(K) \\
@VVV @VVV @VVV \\
\mathcal T_S(\mathcal O_v) @>>> \mathbb T_S(K_v) @<p<< \mathbb T(K_v) \\
@A{\mu_0}AA @A{\mu_0}AA @A{\mu_v'}AA\\
\{1\} @>>> \{1\} @<q<< \mathbb G_m(K_v)
\end{CD}
\end{equation}
in which the maps are as follows. 
The unlabeled maps in the top two rows  of \eqref{CD.Tow2Big} 
 are induced by the obvious injective ring homomorphisms
\begin{equation}
\begin{CD}
\quad\qquad @. \quad K_S \quad@>>> \quad K \quad  @. \qquad\qquad\quad K \qquad\qquad\quad\\
@. @VVV @VVV  @VVV\\ 
\quad\qquad @. \quad \mathcal O_v \quad @>>> \quad K_v \quad @. \qquad\qquad\quad K_v \qquad\qquad\quad
\end{CD}
\end{equation}  
The map  $\mu_v'$ was defined in  section \ref{sec.Loc} on localization.  
The homomorphisms  $\mu_0$ 
and $q$ are of course trivial.  The diagram commutes; indeed, $p \mu_v'$  is
trivial because $u\notin S$, and it is clear that the other three squares commute.

{\bf Claim 1.}
We claim that $H^1(G(K_v/F_u),A)$ vanishes for each of the nine 
groups $A$ appearing in \eqref{CD.Tow2Big}. (Earlier we used $A$ as an abbreviation for the short exact sequence $A(V_F)$, but we no longer need to reserve the notation $A$ for this purpose.) 

For the group $A$ appearing in the lower right corner of \eqref{CD.Tow2Big}, 
Claim 1 follows from Hilbert's Theorem 90. 
For the remaining groups $A$ 
it follows from the 
various vanishing theorems proved in 
 Lemmas \ref{lem.TNcond2}, \ref{lem.VanS} and \ref{lem.LangHensTorus}.

It follows from Claim 1 that all the sets 
$H^1_{\alg}$ in the big commutative diagram \eqref{CD.NewMass}  
are well-defined up to canonical isomorphism (independent of the choice 
of extensions $\mathcal E$  having the right second cohomology classes). 

{\bf Claim 2.} 
The maps in  diagram \eqref{CD.Tow2Big} can be extended to 
 homomorphisms (of extensions of $G(K_v/F_u)$)  
\begin{equation}\label{CD.Tow2Big5}
\begin{CD}
\mathcal E_3^{}(S) @>>> \mathcal E_3^{K}(S) @<\tilde p<< \mathcal E_3(K/F) \\
@VVV @VVV @VVV \\
\mathcal E_3^{\mathcal O_v}(S) @>>> \mathcal E_3^{K_v}(S) @<\tilde p<<
\mathcal E_3^{v}(K/F)
\\ @A{\tilde{\mu}_0}AA @A{\tilde{\mu}_0}AA @A{\tilde{\mu}'_v}AA\\
G(K_v/F_u) @>>> G(K_v/F_u) @<\tilde q<< \mathcal E(K_v/F_u). 
\end{CD}
\end{equation}
(Of course the two arrows in the bottom row are necessarily the identity map on $G(K_v/F_u)$ and 
the canonical surjection $\mathcal E(K_v/F_u) \twoheadrightarrow G(K_v/F_u)$.)  
Moreover, for any choice of such a collection of extended homomorphisms,  the diagram 
\eqref{CD.Tow2Big5} is essentially
commutative, by which we mean that each of its four squares commutes up to conjugacy
under $\mathbb T_S(K_v)$. 

The second part of Claim 2 follows from Claim 1. 
The first part of Claim 2 will follow from Claim 3 below.

 Notice that there are canonical elements in 
$H^2(G(K_v/F_u),A)$ for each of the four modules $A$ appearing in the corners of \eqref{CD.Tow2Big}. They are as follows:
\begin{itemize}
\item  $\alpha_3(S)$ when $A=\mathcal T_S(K_S)$, 
\item  $\alpha_3$ when $A=\mathbb T(K)$,
\item the local fundamental class $\alpha(K_v/F_u)$ when $A=\mathbb G_m(K_v)$, 
\item $1$  when $A=\{1\}$.
\end{itemize} 

 {\bf Claim 3.} There is a unique collection of nine elements $\alpha_A \in H^2(G(K_v/F_u),A)$, one for each of the 
nine modules $A$ in diagram \eqref{CD.Tow2Big}, such that 
\begin{itemize}
\item $\alpha_A$ is the canonical element when $A$ is one of the four corners, and 
\item  each arrow $A \to A'$ in the diagram maps 
$\alpha_A$ to $\alpha_{A'}$.
\end{itemize} 
Moreover these nine elements $\alpha_A$ are the $2$-cohomology classes corresponding 
to the nine extensions of $G(K_v/F_u)$ 
 appearing in diagram \eqref{CD.Tow2Big5}. 

To verify the first part of Claim 3 it is enough to check that, for each of the four outer edges 
\[
A' \xrightarrow{} A \xleftarrow{} A''
\]  
in  diagram \eqref{CD.Tow2Big}, the image 
of $\alpha_{A'}$ under $A' \to A$ agrees with the image of $\alpha_{A''}$ under $A'' \to A$. 
This follows from 
 \begin{itemize}
\item Lemma \ref{lem.kSpS} for the top edge, 
\item Lemma \ref{lem.LocWorksWell}(2) for the right edge, 
\item Lemma \ref{lem.LangHensTorus} for the left edge, 
\end{itemize} 
and is trivially true for the bottom edge. The second part of Claim 3 is clear, once one 
remembers how the various extensions  were defined. 

Now we are going to define all the maps in \eqref{CD.NewMass}. 
 Each of the nine sets in that diagram is of the form $H^1_Y(E,M)$ (see subsection \ref{sub.NonAbH1Y}) 
for suitable $E$, $M$, $Y$ and $\xi:Y \to \Hom(A,M)$.   
The relevant groups $E$ and $M$ 
are shown explicitly in the diagram, and the relevant groups $A$ are the ones appearing in 
the corresponding locations in diagram \eqref{CD.Tow2Big}.  
The nine sets $Y$ (indicated only by the subscript ``$\alg$''  in the diagram), 
 together with certain natural maps linking them, are shown in the 
commutative diagram 
 \begin{equation}\label{CD.Massive2}
\begin{CD}
\Hom_{K_S}(\mathcal T_S,\mathcal G) @>>> \Hom_{K}(\mathbb T_S, G) @>p>>
\Hom_{K}(\mathbb T, G) \\
 @VVV @VVV @VVV \\
\Hom_{\mathcal O_v}(\mathcal T_S,\mathcal G) @>>>
\Hom_{K_v}(\mathbb T_S, G) @>p>>\Hom_{K_v}(\mathbb T, G) \\
@V{\mu_0}VV @V{\mu_0}VV @V{\mu_v'}VV\\
\Hom_{\mathcal O_v}(\{1\}, \mathcal G) @>>> \Hom_{K_v}(\{1\}, G) @>q>>
\Hom_{K_v}(\mathbb G_m, G).
\end{CD}
\end{equation} 
The nine $G(K_v/F_u)$-groups $M$ are linked by obvious homomorphisms 
\begin{equation}\label{CD.Massive5}
\begin{CD}
\mathcal G(K_S) @>>> G(K) @=  G(K) \\
 @VVV @VVV @VVV\\
\mathcal G(\mathcal O_v) @>>>
 G(K_v) @=  G(K_v) \\
@| @|  @|\\
 \mathcal G(\mathcal O_v) @>>>  G(K_v) @=  G(K_v).
\end{CD}
\end{equation} 
We use the maps shown in \eqref{CD.Massive5} \eqref{CD.Massive2}, \eqref{CD.Tow2Big5} 
to define the maps in \eqref{CD.NewMass}. All of them are instances of the maps $\Phi$ or $\Psi$
defined in  
subsections \ref{sub.MapPhifgh} and \ref{sub.DefnPsiMap}. The ones for which the relevant 
arrow in \eqref{CD.Massive5} is an equality are the ones for which the map is of type $\Psi$,
and  the others are of type $\Phi$. 

Some of the maps in \eqref{CD.NewMass} have already been 
defined earlier in this subsection, but it is easy to see that they agree with  
the ones we just described.  
The commutativity of the four squares in 
\eqref{CD.NewMass} follows from 
\begin{itemize}
\item Lemma \ref{lem.CompPhi} for the upper left square, 
\item Lemma \ref{lem.CompPsi} for the lower right square, 
\item Lemma \ref{lem.PhiPsi} for the remaining two squares.
\end{itemize}
The construction of diagram \eqref{CD.NewMass} is now complete,  so 
we are finally finished with the proof of  Proposition \ref{prop.BigDiag}.

\section{$B(F,G)_{\bsc}$ in the global case}\label{sec.BFGglobal} 

\subsection{Goal} 
Let $F$ be a global field. We fix a separable closure $\bar F$ of $F$ and
put $\Gamma:=\Gal(\bar F/F)$.  
Let $G$ be a connected reductive group over
$F$.  In this section we are going to study
$B(F,G)_{\bsc}$. The formulation of the main result is inspired by Borovoi's
Theorem 5.11 in \cite{B2}. 

\subsection{Main result} 

Let $S_{\infty}$ denote the set of infinite places of $F$. Consider the
commutative square 
\begin{equation}\label{CD.BAMain}
\begin{CD}
B(F,G)_{\bsc} @>>> \prod_{u \in S_{\infty}} B(F_u,G)_{\bsc}\\
@V{\kappa_G}VV @VVV \\
A(F,G) @>>> \prod_{u \in S_{\infty}} A(F_u,G),
\end{CD}
\end{equation}
in which the right vertical arrow is the product over $S_{\infty}$ of the
local maps $\kappa_G$, and the horizontal arrows have as $u$-components 
the  localization maps appearing in  
 Lemma \ref{lem.KapLocCompat}.

\begin{proposition}\label{prop.BAMain}
Diagram \eqref{CD.BAMain} is a cartesian square of sets.
\end{proposition} 

The proof will be given in the next three subsections. 

\begin{corollary}
In the function field case the map 
\[
\kappa_G:B(F,G)_{\bsc} \to A(F,G)
\] 
is bijective. 
\end{corollary}

\subsection{Proof of Proposition \ref{prop.BAMain} when $G_{\der}$ is simply
connected}  In this subsection we assume that the derived group of $G$ is
simply connected. We consider the short exact sequence 
\[
1 \to G_{\der} \to G \to D \to 1,
\] 
where $D$ denotes the quotient of $G$ by its derived group $G_{\der}$. 

Proposition \ref{prop.BAMain} is equivalent to the statement that 
\begin{equation}\label{CD.Gder}
\begin{CD}
B(F,G)_{\bsc} @>>> \prod_{u \in S_{\infty}} B(F_u,G)_{\bsc}\\
@VVV @VVV \\
B(F,D) @>>> \prod_{u \in S_{\infty}} B(F_u,D).
\end{CD}
\end{equation}
 is a cartesian square of sets, because $A(F,G)=B(F,D)$ and
$A(F_u,G)=B(F_u,D)$.

We need to show that the map from $B(F,G)_{\bsc}$ to the fiber product 
is bijective. We begin by showing that it  is injective. For this we
consider basic elements $b,b'$ in $B(F,G)$, and we assume that 
\begin{itemize}
\item $b,b'$ become equal in $B(F,D)$, and 
\item $b,b'$ become equal in $B(F_u,G)$ for every infinite place $u$ of $F$.
\end{itemize} 
We need to show that $b=b'$. 

Let $Z(G)$ denote the center of $G$, and let 
$\nu,\nu':\mathbb D_F
\to Z(G)$ be the Newton points of $b,b'$ respectively.  We claim that
$\nu=\nu'$. Indeed, this follows from the fact that $b,b'$ have the same
image in $B(F,D)$, since the map 
$\Hom_F(\mathbb D_F,Z(G)) \to \Hom_F(\mathbb D_F,D)$ is injective (any
homomorphism from a protorus  to the finite group 
$\ker[Z(G)\to D]$ is obviously trivial). 

Since $\nu=\nu'$,  the discussion in \ref{sub.1rstCpt}  shows that the
difference between $b$ and $b'$ is measured by an element $x$ in the pointed
set
$H^1(F,J_b)$. Here we have chosen an algebraic $1$-cocycle representing $b$
and used it to obtain the inner form $J_b$  of $G$. We need to show that $x$
is trivial. Our second assumption implies that $x$ is locally trivial at
every infinite place of $F$. Our first assumption, together with the fact
that $B(F_u,G)_{\bsc} \to B(F_u,D)$ is bijective for finite places $u$
(see Proposition \ref{prop.BscLocalMain}(1)), 
tells us that $x$ is locally trivial at every finite place of $F$. Therefore
$x$ is locally trivial everywhere. Again using the first assumption, we
conclude that $x$ is an element in  
\begin{equation}\label{eq.KerKer}
\ker[\ker^1(F,J_b) \to \ker^1(F,D)].
\end{equation}
To show that $x$ is trivial we just need to prove that 
the set \eqref{eq.KerKer} is trivial. 
In the number field case this follows from \cite[Lemma 4.3.1(b)]{CTT} 
(closely related to  Theorem 4.3 in \cite{Sa}). In the
function field case it is even true that the set 
\begin{equation}\label{eq.KerH1}
\ker[H^1(F,J_b) \to H^1(F,D)]
\end{equation} 
is trivial. 
Indeed, the set \eqref{eq.KerH1} is the image of 
$H^1(F,(J_b)_{\der})$, and this is 
 trivial \cite{HarderIII} because $(J_b)_{\der}$ is semisimple simply
connected.

We are done proving injectivity of the map from $B(F,G)_{\bsc}$ to the fiber
product. Now we prove surjectivity. For this we consider 
\begin{itemize}
\item an element $b_D \in B(F,D)$, and 
\item elements $b_u \in B(F_u,G)_{\bsc}$, one for each $u \in S_{\infty}$, 
\end{itemize}
such that
 \begin{itemize}
\item for all $u \in S_{\infty}$, the elements $b_D$ and $b_u$ become
equal in $B(F_u,D)$.
\end{itemize}
 We must show that there exists a basic element $b \in
B(F,G)$ such that 
\begin{itemize}
\item $b$ maps to $b_D$ under $B(F,G) \to B(F,D)$, and 
\item for every infinite place $u$ the element $b$ maps to $b_u$ under
$B(F,G)
\to B(F_u,G)$.
\end{itemize}

Choose a finite  set $S$ of places of $F$ such that 
\begin{itemize}
\item $S$ contains all infinite places of $F$, 
\item $S$ contains some finite place $u_0$ of $F$, and 
\item $b_D$ comes from $S$. 
\end{itemize}
Here we are using the following definition. 
\begin{definition}
Say that $b_D \in B(F,D)$ \emph{comes from $S$} if $b_D$ lies in the image
of 
\[
(X_*(D) \otimes X_3(K,S))_{G(K/F)} \to (X_*(D) \otimes X_3(K))_{G(K/F)}
\simeq B(F,D) 
\] 
for some (equivalently, every) finite Galois extension $K/F$ that splits
$D$, where $X_3(K,S)$ denotes $\mathbb Z[S_K]_0$.
\end{definition}
 
To see the equivalence of ``some'' and ``every'' in this definition, use
the surjectivity of  $X_3(L,S) \to X_3(K,S)$ when $L$ is a  
finite Galois extension of $F$ with $L \supset K$.

Next we choose (use \cite[Lemma 5.5.3]{H} in the number field case and 
\cite[Prop.~3.2]{Wort} in the function field case) a
maximal
$F$-torus
$T$ in
$G$ such that
$T$ is fundamental over $F_u$ for every $u \in S$. (In the nonarchimedean
case we are using fundamental as a synonym for elliptic, so $T$ is
fundamental when   
$T/Z(G)$ is anisotropic.) Finally, we choose a finite
Galois extension
$K/F$ that splits $T$. We then have an exact sequence 
\[
1 \to T_{\der} \to T \xrightarrow{q} D \to 1,
\] 
where $q$ denotes the restriction of $G \to D$ to $T$, and $T_{\der}=T\cap
G_{\der}$. Observe that $D$ also splits over $K$. Since we have chosen $S$
large enough that $b_D$ comes from $S$, we can express $b_D$ as the image of
some element 
\[
\sum_{v \in S_K} \mu_v \otimes v \in X_*(D) \otimes \mathbb Z[S_K]_0.
\] 
Thus $\mu_v \in X_*(D)$ for all $v \in S_K$, and $\sum_{v \in S_K} \mu_v=0$. 

For any place $u$ of $F$ we now write $V_u$ for the set of places of $K$
lying over $u$.  Let $u$ be an infinite place of $F$.  Because $b_u$ is basic and 
$T$ is fundamental over $F_u$, 
Lemma \ref{lem.BscReal} guarantees that
 there exists
$b_{T,u} \in B(F_u,T)$ such that $b_{T,u}$ maps to our given element $b_u$. 
We choose $\sum_{v \in  V_u} \mu'_v \otimes v \in X_*(T) \otimes \mathbb
Z[V_u]$ whose image in $(X_*(T) \otimes \mathbb Z[V_u])_{G(K/F)}\simeq
B(F_u,T)$ is equal to
$b_{T,u}$. 

Now $b_{T,u}$ and $b_D$ become equal in $B(F_u,D)$, so $\sum_{v \in
V_u}q(\mu'_v) \otimes v$ and $\sum_{v \in V_u}\mu_v \otimes v$ represent the
same class in 
$(X_*(D)\otimes \mathbb Z[V_u])_{G(K/F)}=B(F_u,D)$. Since $q:X_*(T) \to
X_*(D)$ is surjective, we may modify $\sum_{v \in  V_u} \mu'_v \otimes v$ in
such a way that 
\begin{itemize}
\item it still represents $b_{T,u}$, and 
\item $q(\mu'_v)=\mu_v$ for all $v \in V_u$. 
\end{itemize} 
We do this for every infinite place $u$. 

Recall that there exists a finite place $u_0$ in $S$. Choose some place
$v_0$ of $K$ lying over $u_0$. For every finite place $v \in S_K$ except
$v_0$ we choose $\mu'_v \in X_*(T)$ such that $q(\mu'_v)=\mu_v$. At this
point we have chosen elements $\mu'_v$ satisfying $q(\mu'_v)=\mu_v$ for
every $v \in S_K$ except for $v=v_0$. We now define $\mu'_{v_0}$ to be the
unique element of $X_*(T)$ such that $\sum_{v \in S_K} \mu'_v=0$. Applying
$q$ to this last equality, we see that $q(\mu'_{v_0})=\mu_{v_0}$. 

It is then clear that $\sum_{v \in S_K} \mu'_v \otimes v$ represents an
element
$b_T \in B(F,T)$ such that 
\begin{itemize}
\item $q(b_T) = b_D$, and 
\item $b_T \mapsto b_{T,u}$ under $B(F,T) \to B(F_u,T)$ for all infinite
places $u$ of $F$.
\end{itemize}
  Therefore the image $b$ of $b_T$ in $B(F,G)$ maps to
$b_D$ and to each $b_u$. To conclude the proof of surjectivity, it 
remains only to show that $b$ is basic. 
By Lemma \ref{lem.BscLocGlobal} it suffices   to show that 
 $b$ is locally basic everywhere. For
infinite places $u$, this follows from our assumption that $b_u$ is basic. 
For finite places  $u \in S$, it follows from the fact that $T$ is elliptic
at $u$. For finite places $u$ outside $S$, it is trivially true, since $b_T$
comes from $S$, hence is locally trivial outside of $S$.

\subsection{An elementary lemma} We are going to need the following very
easy lemma, whose proof is left to the reader. The lemma concerns the
following situation. Suppose that we are given a cartesian square 
\begin{equation}
\begin{CD}
A_1 @>f_{12}>> A_2 \\
@Vf_{13}VV @Vf_{24}VV \\
A_3 @>f_{34}>> A_4
\end{CD}
\end{equation}
in the category of groups, as well as a cartesian square
 \begin{equation}
\begin{CD}
X_1 @>g_{12}>> X_2 \\
@Vg_{13}VV @Vg_{24}VV \\
X_3 @>g_{34}>> X_4
\end{CD}
\end{equation}
in the category of sets. Suppose further that
\begin{itemize}
\item for $i=1,2,3,4$ we are given
an action of $A_i$ on $X_i$, and 
\item for $ij=12,13,24,34$ the map $g_{ij}$ is equivariant with respect to
$f_{ij}$, i.e., $g_{ij}(a_i x_i)=f_{ij}(a_i)g_{ij}(x_i)$. 
\end{itemize}
In this situation there is an obvious commutative square 
 \begin{equation} \label{CD.cart?}
\begin{CD}
A_1\backslash X_1 @>>> A_2\backslash X_2 \\
@VVV @VVV \\
A_3\backslash X_3 @>>> A_4\backslash X_4
\end{CD}
\end{equation}
of sets, with $A_i\backslash X_i$ denoting the quotient of $X_i$ by the
action of $A_i$. 
\begin{lemma} \label{lem.ElemCart}
The square \eqref{CD.cart?} is cartesian if 
\begin{itemize}
\item $A_4=f_{24}(A_2)f_{34}(A_3)$, and 
\item $A_4$ acts freely on $X_4$, i.e., if $a_4 \in A_4$ fixes some element
of $X_4$, then
$a_4$ is the identity.  
\end{itemize}
\end{lemma}

\subsection{Proof of the general case of Proposition \ref{prop.BAMain} using 
$z$-extensions}  We have proved  Proposition \ref{prop.BAMain} when 
$G_{\der}$ is simply connected. Now we use $z$-extensions to prove it in
general. 
So, we begin by choosing a finite Galois extension $K/F$ splitting  $G$ and
a $z$-extension 
\[
1 \to Z \to G' \to G \to 1
\] 
 with $Z=R_{K/F}(S)$ for some split $K$-torus $S$.
Because
$G'_{\der}$ is simply connected,   the square 
\begin{equation}\label{CD.Gder2}
\begin{CD}
B(F,G')_{\bsc} @>>> \prod_{u \in S_{\infty}} B(F_u,G')_{\bsc}\\
@VVV @VVV \\
A(F,G') @>>> \prod_{u \in S_{\infty}} A(F_u,G').
\end{CD}
\end{equation}
is cartesian.

The commutative diagram 
\begin{equation}\label{CD.Zinfty}
\begin{CD}
B(F,Z) @>>> \prod_{u \in S_{\infty}} B(F_u,Z)\\
@VVV @VVV \\
A(F,Z) @>>> \prod_{u \in S_{\infty}} A(F_u,Z).
\end{CD}
\end{equation} 
is trivially cartesian, because the vertical arrows are isomorphisms.

As in the lead-up to Lemma \ref{lem.ElemCart}, the groups in diagram
\eqref{CD.Zinfty} act on the sets in diagram \eqref{CD.Gder2}, and so we
obtain a commutative square of quotient sets, and this boils down to  
\begin{equation}\label{CD.Gder571}
\begin{CD}
B(F,G)_{\bsc} @>>> \prod_{u \in S_{\infty}} B(F_u,G)_{\bsc}\\
@VVV @VVV \\
A(F,G) @>>> \prod_{u \in S_{\infty}} A(F_u,G) ,
\end{CD}
\end{equation}
 by Proposition  \ref{prop.ZactB} and the fact that 
\[
A(F,Z) \to A(F,G') \to A(F,G) \to 0
\] 
is exact, locally as well as globally. (Use that 
$
0 \to\Lambda_Z \to \Lambda_{G'} \to \Lambda_G \to 0 
$ 
is exact.)

It follows from  Lemma \ref{lem.ElemCart}  that 
the square \eqref{CD.Gder571} is cartesian. The first hypothesis of
that lemma is trivially satisfied, because the right vertical arrow in
\eqref{CD.Zinfty} is an isomorphism. To show that the  second hypothesis is
 satisfied, we need to check that  $A(F_u,Z) \to A(F_u,G')$
is injective for every place $u$. 

So   
we need to check that 
\begin{equation}\label{eq.KerHomology}
 (\Lambda_Z)_{G(K_v/F_u)} \to (\Lambda_{G'})_{G(K_v/F_u)} 
\end{equation}
is injective for any  $v$ over $u$. From the long exact sequence of
group homology we see that the kernel of \eqref{eq.KerHomology}  is a
torsion group (killed by  $[K_v:F_u]$). But, because
$Z=R_{K/F}(S)$, the group  $(\Lambda_Z)_{G(K_v/F_u)}$ is torsion-free.
So the kernel of \eqref{eq.KerHomology} vanishes, and 
the proof of Proposition \ref{prop.BAMain} is now complete. 

\subsection{The image of $B(F,G)_{\bsc} \to A(F,G)$}
The next result involves the subset $A_0(F,G)$ of $A(F,G)$ defined in Remark
\ref{rem.im=A0}. 
\begin{proposition}\label{prop.ImKapGlob}
For any global field $F$ the image of $\kappa_G:B(F,G)_{\bsc} \to A(F,G)$ is 
$A_0(F,G)$.
\end{proposition}

\begin{proof} We already know from Remark \ref{rem.im=A0} 
that the image of $B(F,G)_{\bsc} \to A(F,G)$ is contained in $A_0(F,G)$. 
So we just need to check that any element in $A_0(F,G)$ lies in 
$\im[B(F,G)_{\bsc} \to A(F,G)]$. This follows easily from 
Propositions \ref{prop.LocImKap} and  \ref{prop.BAMain}.  
\end{proof}

\subsection{Analysis of the total localization map}
Let $K$ be a finite Galois extension of $F$ in $\bar F$ such that $\Gal(\bar
F/K)$ acts trivially on $\Lambda_G$.  For each finite place
$u$ of
$F$ we choose a place $v$ of $K$ lying over $u$. We  then have 
\[
B(F_u,G)_{\bsc}\simeq (\Lambda_G)_{G(K_v/F_u)}.
\]
So the fact that diagram \eqref{CD.BAMain} is cartesian can  be
reformulated as the fact that the diagram 
\begin{equation}\label{CD.Gder5712}
\begin{CD}
B(F,G)_{\bsc} @>>> \bigoplus_{u \in V_F} B(F_u,G)_{\bsc}\\
@VVV @VVV \\
\bigl(X_3(K)\otimes
\Lambda_G\bigr)_{G(K/F)} @>>>\bigoplus_{u \in V_F} (\Lambda_G)_{G(K_v/F_u)}
\end{CD}
\end{equation}
is cartesian. Here we are using direct-sum notation in a nonstandard way by 
writing
$\bigoplus_{u
\in V_F} B(F_u,G)_{\bsc}$ for the subset of $\prod_{u \in V_F}
B(F_u,G)_{\bsc}$ consisting of families  of elements  $b_u \in
B(F_u,G)_{\bsc}$ such that
$b_u$ is trivial for all but finitely many places $u$. In order to
understand the significance of this last cartesian diagram, we apply the
right-exact functor
$X
\mapsto (X\otimes \Lambda_G)_{G(K/F)}$ to the short exact sequence 
\[
0 \to X_3(K) \to X_2(K) \to X_1(K) \to 0,
\]
concluding that the cokernel of the bottom horizontal arrow in 
diagram \eqref{CD.Gder5712} can be identified with
$(\Lambda_G)_{G(K/F)}$. The fact that \eqref{CD.Gder5712} is cartesian then
yields the following result. 

\begin{proposition}\label{prop.CokGlobalG}
 An element in $\bigoplus_{u \in V_F} B(F_u,G)_{\bsc}$
lies in the image of the localization map 
\begin{equation}\label{eq.LocMp}
B(F,G)_{\bsc} \to\bigoplus_{u \in V_F} B(F_u,G)_{\bsc}
\end{equation}
if and only if its image under 
\[
\bigoplus_{u \in V_F} B(F_u,G)_{\bsc} \to \bigoplus_{u \in V_F}
(\Lambda_G)_{G(K_v/F_u)} \to (\Lambda_G)_{G(K/F)}
\]
is trivial. 
\end{proposition}

\begin{proof}
Clear. 
\end{proof}

The kernel of the localization map \eqref{eq.LocMp} is easily seen to
coincide with $\ker^1(F,G)$, and in \cite[\S 4]{CTT} this is described in
terms of $Z(\hat G)$. So we have a satisfactory understanding of the
localization map.

\appendix

\section{Rigidity of weak Tate-Nakayama triples}\label{App.A} 

\subsection{Review of the definition of Tate-Nakayama triple}
Let $X$, $A$ be $G$-modules, and let $\alpha \in H^2(G,\Hom(X,A))$.   
Recall  from section \ref{sec.absTN} that $(X,A,\alpha)$ is a
\emph{Tate-Nakayama triple for
$G$} 
if the following two conditions hold for every subgroup $G'$ of $G$:   
\begin{itemize}
\item For all $r \in \mathbb Z$ cup product with   $\Res_{G/G'}(\alpha)$
induces isomorphisms 
\[ 
H^r(G', X) \to
H^{r+2}(G', A). 
\] 
\item $H^1(G',\Hom(X,A))$ is trivial. 
\end{itemize} 
Weak Tate-Nakayama triples are ones for which the first condition holds (but
possibly not the second). The second condition is referred to as 
rigidity. In this appendix our main goal is to show 
 that  weak Tate-Nakayama triples of a certain kind 
 are automatically rigid.

\subsection{Review of Nakayama's theorem} 

Let $G$ be a finite group. For any $G$-module $M$ and any $r \in \mathbb Z$
the
 Tate cohomology
group $H^r(G,M)$ is defined. For a subgroup $G'$ of $G$ there are
restriction maps 
$\Res_{G/G'}:H^r(G,M) \to H^r(G',M)$ (see \cite{Se} for all this).

We now recall a special case of a result of Nakayama (see \cite{N,Se}).

\begin{theorem} Let $(X,A,\alpha)$ be a weak Tate-Nakayama triple. 
Then cup product with $\alpha$ is an isomorphism 
\[
H^r(G,M \otimes X) \to H^{r+2}(G,M \otimes A)
\] 
for every $r \in \mathbb Z$ and 
 every $G$-module $M$ that is torsion-free as abelian group. 
\end{theorem} 

For any $G$-module there is an  obvious pairing 
\[
\Hom(X,A) 
\otimes 
\Hom(M,X) \to \Hom(M,A),
\] 
given by composition of mappings. So cup product with $\alpha \in 
H^2(G,\Hom(X,A))$ also yields maps  
\begin{equation}\label{eq.cupTN}
H^r(G,\Hom(M,X)) \xrightarrow{\alpha\smile} H^{r+2}(G,\Hom(M, A)).
\end{equation}

\begin{definition}\label{def.calC}  
 Let $(X,A,\alpha)$ be a weak Tate-Nakayama triple. 
Let $\mathcal C=\mathcal C(X,A,\alpha)$ be the class    of
$G$-modules for which \eqref{eq.cupTN} is an isomorphism for all $r
\in \mathbb Z$. 
\end{definition} 
The next lemma gives some simple observations about the
class
$\mathcal C$, the fourth of which is a standard corollary of Nakayama's
theorem.

\begin{lemma} \label{lem.calC} 
\hfill 
\begin{enumerate} 
\item 
The class $\mathcal C$ is closed under arbitrary direct sums. 
\item Let $0 \to M''' \to M'' \to M' \to 0$ be a short exact sequence of
$G$-modules, and assume that $M'$ is free as abelian group. If two of
 $M',M'',M'''$ lie in the class $\mathcal C$, then so does the third one. 
\item 
The class $\mathcal C$ contains all $\mathbb Z$-free $G$-modules $M$ which
admit a chain 
\[
M_1 \subset M_2 \subset M_2 \subset \dots
\] 
of submodules such that 
(i) each $M_n$ lies in $\mathcal C$, and 
(ii) $M =\cup_{n=1}^{\infty} M_n$. 
\item The class $\mathcal C$ contains all $G$-modules $M$ that are free
of finite rank as abelian groups. 
\item The class $\mathcal C$ contains all $G$-modules $M$ that are free as
abelian groups and have a $\mathbb Z$-basis that is permuted by the
action of $G$. 
\item The class $\mathcal C$ contains all $G$-modules $M$ that are
free of countable rank as abelian groups.
\end{enumerate}
\end{lemma}
\begin{proof}

(1)
$\Hom(\cdot,X)$ and $\Hom(\cdot,A)$ convert direct sums into direct
products, and these are preserved by Tate cohomology. 

(2) 
Our assumption on $M'$ ensures that the sequences  
\[
0 \to \Hom(M',X) \to\Hom(M'',X) \to \Hom(M''',X) \to 0
\]
\[
0 \to \Hom(M',A) \to\Hom(M'',A) \to \Hom(M''',A) \to 0
\]
are short exact. Now consider their long exact sequences of Tate cohomology
 and apply
the $5$-lemma. 

(3) 
It follows from (1) that the module
$N:=\oplus_{n=1}^{\infty} M_n$ lies in
$\mathcal C$. We write elements $x \in N$ as sequences
$(x_1,x_2,x_3,\dots)$ with
$x_n \in M_n$ and $x_n=0$ for all but finitely many $n$. There is an
obvious surjection $g:N \to M$, defined by $g(x_1,x_2,x_3,\dots) =
\sum_{n=1}^{\infty} x_n$. We  define an endomorphism $f$ of $N$ by the
rule 
\[
f(x_1,x_2,x_3,\dots)=(x_1,x_2-x_1,x_3-x_2,\dots,x_{n+1}-x_n,\dots). 
\] 
It is easy to check that the sequence 
$
0 \to N \xrightarrow{f} N \xrightarrow{g} M \to 0 
$
is short exact. Since $N$ lies in $\mathcal C$ and $M$ is $\mathbb Z$-free,
we conclude from (2) that $M$ lies in $\mathcal C$.

(4) Apply the theorem of Nakayama to the $\mathbb Z$-dual of $M$. 

(5) This follows from (1) and (4). 

(6) This follows from (3) and (4). 
\end{proof}

\subsection{A sufficient condition for  a weak Tate-Nakayama triple to be 
rigid} 

We  consider a weak Tate-Nakayama triple $(X,A,\alpha)$. We are going to
give  some simple conditions on $X$ that imply the  rigidity of
$(X,A,\alpha)$. Before doing so we introduce some notation. Let $S$ be any
$G$-set. Then we write
$\mathbb Z[S]$ for the free abelian group on $S$. There is an obvious
$G$-module structure on $\mathbb Z[S]$, for which the $G$-action permutes
the basis elements $s \in S$ according to the given action on
$S$. There is an obvious $G$-map $f$ from $\mathbb Z[S]$ to the trivial
$G$-module $\mathbb Z$, 
defined by $f(\sum_{s \in S} n_s s)=\sum_{s\in S} n_s$. We denote by
$\mathbb Z[S]_0$ the $G$-module obtained as the kernel of $f$. Thus there is
a short exact sequence of $G$-modules 
\begin{equation}\label{eq.if0}
0 \to \mathbb Z[S]_0 \xrightarrow{i} \mathbb Z[S] \xrightarrow{f} \mathbb Z
\to 0. 
\end{equation}

\begin{lemma}\label{lem.rigid}
Consider a weak Tate-Nakayama triple $(X,A,\alpha)$. Assume that   $X$
satisfies one of the following two conditions: 
\begin{itemize}
\item There exists a $G$-set $S$ such that $X$ is isomorphic to $\mathbb
Z[S]$. 
\item   
   There exists a $G$-set $S$ such that $X$ is isomorphic to $\mathbb
Z[S]_0$, and, in addition,   $H^{-1}(G',X)$ vanishes for every subgroup $G'$
of
$G$.  
\end{itemize} 
Then cup product induces an isomorphism 
\begin{equation}\label{eq.XXiso}
H^r(G,\Hom(X,X)) \to H^{r+2}(G,\Hom(X,A))
\end{equation}  
for all $r \in \mathbb Z$. Moreover, $(X,A,\alpha)$ is rigid. 
\end{lemma} 
\begin{proof} 
The statement that \eqref{eq.XXiso} is an isomorphism for all $r \in
\mathbb Z$ is just the statement $X$ lies in the class
$\mathcal C(X,A,\alpha)$ of Definition \ref{def.calC}. When $X=\mathbb
Z[S]$, this follows from part (5) of Lemma \ref{lem.calC}. When $X=\mathbb
Z[S]_0$, it follows from parts (2) and  (5) of that lemma.  

It remains to prove that $(X,A,\alpha)$ is rigid. So, for every subgroup
$G'$ of $G$, we must show that $H^1(G',\Hom(X,A))$ vanishes. In fact, we may
as well  take $G'$ to be $G$, since  all the hypotheses of the
lemma also hold for the Tate-Nakayama triple
$(X,A,\Res_{G/G'}(\alpha))$ for $G'$.  Because of  the isomorphism
\eqref{eq.XXiso},  we just need to check that $H^{-1}(G,\Hom(X,X))$
vanishes.

When $X=\mathbb Z[S]$,   
this vanishing is a special case of Corollary \ref{lem.STvan}, and, when
$X=\mathbb Z[S]_0$, it is a special case of Lemma \ref{lem.H-10S}(3). (In
that lemma, when $S=T$, we may take $\epsilon$ to be the identity map.) 
\end{proof} 

In the rest of this appendix we make the calculations with Tate
cohomology that  were invoked in the proof of  the previous lemma. Along
the way we review some basic facts about Tate cohomology and prove some other
technical results needed in the body of the text.  

\subsection{Standard facts about Tate cohomology} 
Let $G$ be a finite group. The following lemma reviews some of the most
basic facts about Tate cohomology.
 
\begin{lemma} \label{lem.StdFacts} \hfill
\begin{enumerate}
\item $H^{-1}(G,\mathbb Z)=0$. 
\item $H^r(G, \prod_{i \in I} M_i)=\prod_{i \in I} H^r(G,M_i)$.
\item $H^r(G, \bigoplus_{i \in I} M_i)=\bigoplus_{i \in I} H^r(G,M_i)$. 
\end{enumerate}
\end{lemma} 

\begin{proof}
(1) is clear. (2)  follows formally from the fact that 
 $H^r(G,M)$
is computed as the cohomology of the complex $\Hom_G(P_n,M)$, where $P_n$ is
the standard complete resolution of $G$. (3) follows formally from the fact
that each of the $G$-modules $P_n$ in the standard resolution is finitely
generated as abelian group, so that the functor $\Hom(P_n,\cdot)$ preserves
direct sums. 
\end{proof}

\subsection{Modules induced from subgroups} 
Again let $G$ be a finite group. 
Let $H$ be a subgroup of $G$, and let $M$ be a $G$-module. Suppose that
there is a family $M_x$ of subgroups of $M$, one for each $x \in G/H$, such
that 
\begin{itemize}
\item $M=\bigoplus_{x \in G/H} M_x$, and  
\item $M_{gx}=gM_x$ for all $g \in G$, $x \in G/H$.
\end{itemize} 
Let $x_0$ denote the base-point in $G/H$ (in other words, $x_0$ is the
trivial coset
$H$ of $H$). The stabilizer $H$ of $x_0$ then acts on $M_0:=M_{x_0}$,  
and  $M$ is both induced and coinduced from the $H$-module $M_0$. 

Let us
denote by $\pi_0$ the projection of $M$ onto the direct summand $M_0$. 
It is evident that $\pi_0$ is an $H$-map, and  
 Shapiro's lemma states that the composed map 
\[
H^r(G,M) \xrightarrow{\Res_{G/H}} H^r(H,M) \xrightarrow{\pi_0} H^r(H,M_0) 
\]
 is an isomorphism.

 Now consider an arbitrary  $G$-set $S$. Of course $S$ decomposes as 
\begin{equation}\label{eq.DisjUnion}
S =\coprod_{s \in G\backslash S} Gs, 
\end{equation}  
where $Gs$ denotes the orbit of $s \in S$ under $G$. For $s \in S$ the
map $g \mapsto gs$ identifies $G/G_s$ with $Gs$. As a consequence of 
Lemma \ref{lem.StdFacts} (2) and Shapiro's lemma, one obtains the following
lemma (see page 714 of \cite{T}). 

\begin{lemma}\label{lem.ShapVar}
For any $G$-set $S$ and any $G$-module $M$ there is a canonical isomorphism 
\[
\pi: H^r(G,\Hom(\mathbb Z[S],M)) \to \prod_{s \in G\backslash S}
H^r(G_s,M), 
\] 
in which, for any $s \in S$, the $s$-component of $\pi$ is given by 
the composed map 
\[
H^r(G,\Hom(\mathbb Z[S],M)) \xrightarrow{\Res_{G/G_s}} H^r(G_s,\Hom(\mathbb
Z[S],M)) \xrightarrow{\pi_s} H^r(G_s,M),
\] 
where $\pi_s$ is the map sending $f \in \Hom(\mathbb Z[S],M)$ to its value
at
$s$. 
\end{lemma} 
The groups
$H^r(G_s,M)$ and $H^r(G_t,M)$ are \emph{canonically} isomorphic when $s,t
\in S$ lie in the same $G$-orbit. Indeed, this isomorphism is induced by
$\Int(g):G_s \to G_t$ and $m \mapsto gm$ for any $g
\in G$ such that $gs=t$. The choice of $g$ is immaterial because inner
automorphisms act trivially on Tate cohomology. This is why it is
reasonable to write the target of the isomorphism in the previous lemma
as a product over $G\backslash S$ (rather than going to the trouble of
choosing a set of representatives for the orbits of $G$ on $S$). 

\begin{lemma}\label{lem.HrZS}
Let $S$ be any $G$-set. Then 
\[
H^r(G,\mathbb Z[S])=\bigoplus_{s \in G\backslash S} H^r(G_s,\mathbb Z).
\]
\end{lemma}
\begin{proof}
 Using the decomposition  \eqref{eq.DisjUnion}, we deduce this from Lemma
\ref{lem.StdFacts}(3) and Shapiro's lemma. 
\end{proof}

When $S$ and $T$ are finite $G$-sets, the $G$-module  
 $\Hom(\mathbb Z[S],\mathbb Z[T])$ is canonically isomorphic to $\mathbb Z[S
\times T]$, so there is a canonical isomorphism 
\[
H^r(G,\Hom(\mathbb Z[S],\mathbb Z[T])) =\bigoplus_{(s,t) \in G\backslash (S
\times T)} H^r(G_{s,t},\mathbb Z),
\]  
where $G_{s,t}$ denotes the stabilizer of $(s,t)$ in $G$. The next lemma
gives a similar result in the case that $S$ and $T$ are arbitrary $G$-sets. 

\begin{lemma}
Let $S$, $T$  be any $G$-sets. Then 
\[
H^r(G,\Hom(\mathbb Z[S],\mathbb Z[T]))=\prod_{s \in G\backslash S} \,
\bigoplus_{t \in G_s\backslash T} H^r(G_{s,t},\mathbb Z). 
\]
\end{lemma}
\begin{proof}
This follows from the previous two lemmas.  
\end{proof}

\begin{remark}
The righthand side of the canonical isomorphism in the previous lemma can
be viewed as a subgroup of 
\[
\prod_{(s,t) \in G\backslash (S
\times T)} H^r(G_{s,t},\mathbb Z).
\] 
The projection of $H^r(G,\Hom(\mathbb
Z[S],\mathbb Z[T]))$ onto $H^r(G_{s,t},\mathbb Z)$ is then given by  the
composed map 
\[
H^r(G,\Hom(\mathbb Z[S],\mathbb Z[T])) \xrightarrow{\Res_{G/G_{s,t}}}
H^r(G_{s,t},\Hom(\mathbb Z[S],\mathbb Z[T]))
\xrightarrow{\pi_{s,t}} H^r(G_{s,t},\mathbb Z)
\] 
where $\pi_{s,t}:\Hom(\mathbb Z[S],\mathbb Z[T]) \to \mathbb Z$ is the
$G_{s,t}$-map sending $f$ to the $t$-component of $f(s)$. 
\end{remark} 

\begin{corollary}\label{lem.STvan}
Let $S$, $T$  be any $G$-sets. Then 
\[
H^{-1}(G,\Hom(\mathbb Z[S],\mathbb Z[T]))=0. 
\]
\end{corollary}
\begin{proof}
This follows from the previous lemma together with 
Lemma \ref{lem.StdFacts} (1). 
\end{proof} 

The next result again involves the short exact sequence (see
\eqref{eq.if0}) 
\begin{equation} \label{eq.if0'}
0 \to \mathbb Z[S]_0 \xrightarrow{i} \mathbb Z[S] \xrightarrow{f} \mathbb Z
\to 0. 
\end{equation}
\begin{lemma} \label{lem.H-10S}
Let $S$, $T$  be any $G$-sets, and suppose that there exists a $G$-map
$\epsilon:T \to S$. Then the following conclusions hold. 
\begin{enumerate}
\item The map 
\begin{equation}\label{eq.ZSS} 
H^{0}(G,\Hom(\mathbb Z,\mathbb Z[T])) \xrightarrow{f} 
H^{0}(G,\Hom(\mathbb Z[S],\mathbb Z[T]))
\end{equation}
is injective. 
\item The group $H^{-1}(G,\Hom(\mathbb Z[S]_0,\mathbb Z[T]))$ vanishes. 
\item Suppose further that $H^{-1}(G',\mathbb Z[T]_0)=0$ for every subgroup
$G'$ of $G$. Then the group 
$H^{-1}(G,\Hom(\mathbb Z[S]_0,\mathbb Z[T]_0))$ vanishes. 
\end{enumerate} 
\end{lemma}
\begin{proof} 
To prove (1) we do the following. For each $t \in T$ we write
$G_t$ for the stabilizer of $t$ in $G$, and we define a $G_t$-map 
\[
g_t:\Hom(\mathbb Z[S],\mathbb Z[T]) \to \mathbb Z
\]
by sending $h$ to the $t$-component of $h(\epsilon(t))$ (thinking of $h$ as
an
$S
\times T$-matrix satisfying a certain finiteness condition, we are sending 
$h$ to its  entry $h_{\epsilon(t),t}$). It is clear  that the
composed map 
\[
\mathbb Z[T]=\Hom(\mathbb Z,\mathbb Z[T]) \xrightarrow{f}  \Hom(\mathbb
Z[S],\mathbb Z[T]) \xrightarrow{g_t} \mathbb Z
\] 
is nothing but the $G_t$-map $\pi_t$ projecting an element in $\mathbb Z[T]$
onto its
$t$-component. 
From this we see that any element $x$ in the kernel of the
map \eqref{eq.ZSS} has trivial image under 
\[
H^0(G,\mathbb Z[T]) \xrightarrow{\Res_{G/G_t}} H^0(G_t,\mathbb Z[T])
\xrightarrow{\pi_t} H^0(G_t,\mathbb Z)
\]
for all $t \in T$. It then follows from Lemma \ref{lem.HrZS} that $x=0$.

Now we prove (2). From the long exact sequence of Tate cohomology for the
short exact sequence 
\[
0 \to \Hom(\mathbb Z,\mathbb Z[T]) \xrightarrow{f} 
\Hom(\mathbb Z[S],\mathbb Z[T]) \xrightarrow{i} 
\Hom(\mathbb Z[S]_0,\mathbb Z[T])
\to 0 
\]
we see that the vanishing of $H^{-1}(G,\Hom(\mathbb Z[S]_0,\mathbb Z[T]))$ 
follows from the first part of this lemma, together with the vanishing of 
$H^{-1}(G,\Hom(\mathbb Z[S],\mathbb Z[T]))$ (see Corollary
\ref{lem.STvan}). 

Finally, we prove (3). 
We use the long exact cohomology  sequence for the short exact sequence 
\[
0 \to \Hom(\mathbb Z,\mathbb Z[T]_0) \xrightarrow{f} 
\Hom(\mathbb Z[S],\mathbb Z[T]_0)
\xrightarrow{i} \Hom(\mathbb Z[S]_0,\mathbb Z[T]_0) \to 0.
\] 
To prove that $H^{-1}(G,\Hom(\mathbb Z[S]_0,\mathbb Z[T]_0))$ vanishes, it
is enough to show that 
\begin{equation}\label{eq.H-1SX} 
H^{-1}(G,\Hom(\mathbb Z[S],\mathbb Z[T]_0))=0
\end{equation}
and that 
\begin{equation}\label{eq.H0inj}
H^0(G,\Hom(\mathbb Z,\mathbb Z[T]_0))\xrightarrow{f} H^0(G,\Hom(\mathbb
Z[S],\mathbb Z[T]_0))
\end{equation}
is injective. 

The vanishing of the group in \eqref{eq.H-1SX} follows from Lemma
\ref{lem.ShapVar} together with our assumption that $H^{-1}(G', \mathbb
Z[T]_0)=0$ for every subgroup $G'$ of $G$. To prove that the map
\eqref{eq.H0inj} is injective, we consider the commutative
square 
\begin{equation}\label{CD.H0SXf}
\begin{CD}
H^0(G,\Hom(\mathbb Z,\mathbb Z[T]_0)) @>f>>  H^0(G,\Hom(\mathbb
Z[S],\mathbb Z[T]_0))
\\  @VVV @VVV \\
H^0(G,\Hom(\mathbb Z,\mathbb Z[T])) @>f>>  H^0(G,\Hom(\mathbb Z[S],\mathbb
Z[T])), 
\end{CD}
\end{equation} 
in which the two vertical maps are induced by the inclusion $\mathbb Z[T]_0
\hookrightarrow \mathbb Z[T]$. We want to prove that the top horizontal
arrow is injective, and for this  it will suffice to show that the left
vertical arrow and bottom horizontal arrow are both injective.  
 The vanishing of 
$H^{-1}(G,\Hom(\mathbb Z,\mathbb Z))=H^{-1}(G,\mathbb Z)$ 
(see Lemma \ref{lem.StdFacts}(1))
implies 
the injectivity of the left vertical arrow. The 
injectivity of the
bottom horizontal arrow was established in the first part of this lemma.  
\end{proof}

\section{Review of corestriction} 

In section \ref{sec.KRAS} we used  a number of simple results
concerning corestriction in group cohomology. They are probably all standard,
but I was not able to find a  textbook reference that had everything I
needed.  For that reason I am including a reasonably complete exposition  of
corestriction in this appendix (with no claim of  originality).  

\subsection{Notation} 
In this appendix $G$ is an arbitrary group, so that  
Tate cohomology is no longer defined, and we are free to write $H^r(G,M)$
($r \ge 0$) for the ordinary cohomology groups of a $G$-module $M$. These
can be computed using any
$\mathbb Z[G]$-free resolution
$P$  of the trivial $G$-module $\mathbb Z$. Indeed, $H^r(G,M)$ is the $r$-th
cohomology group of the complex $\Hom(P,M)^G=\Hom_G(P,M)$. If one takes 
$P$ to be the
standard resolution $\mathbb P$, one is led to standard cochains. 

\subsection{Automorphisms}
Let $\theta$ be an automorphism of $G$. By a $\theta$-automorphism of a
$G$-module $M$ we will mean an automorphism $\theta_M$ of the abelian group
$M$ such that $\theta_M(gm)=\theta(g)\theta_M(m)$ for all $g \in G$, $m \in
M$. Similarly for complexes of $G$-modules. Any $\theta$-automorphism of
$M$ preserves the $G$-invariants $M^G$ in $M$. 

There is an obvious $\theta$-automorphism $\theta_{\mathbb P}$ 
of the standard resolution for
$G$ (take the automorphism of $\mathbb P_r=\mathbb
Z[G^{r+1}]$ induced by the automorphism $(g_0,\dots,g_r) \mapsto
(\theta(g_0),\dots,\theta(g_r))$ of $G^{r+1}$). 

Let $M$ be a $G$-module, and suppose that we are given a
$\theta$-automorphism $\theta_M$ of $M$.  There is then an obvious
$\theta$-automorphism
$\theta$ of the complex  $\Hom(\mathbb P,M)$ 
(sending $f:\mathbb P_r \to M$ to $\theta_M
\circ f \circ \theta_{\mathbb P}^{-1}$). The induced automorphism on the cohomology
of the complex  $\Hom(\mathbb P,M)^G$ then provides an
automorphism $\theta$ of $H^r(G,M)$. However, it is not essential to use
the standard resolution. It works equally well to take any free resolution 
$P$ (of $\mathbb Z$) equipped with a $\theta$-automorphism (of complexes) that induces 
the identity map on $\mathbb Z=H^0(P)$.  

Now let $x \in G$, and consider the inner automorphism $\theta_x=\Int(x)$
of $G$. Any $G$-module $M$ then admits a canonical $\theta_x$-automorphism,
namely $\theta_M(m)=xm$. It is well-known (see \cite{Se}) that the induced
automorphism $\theta_x$ of  $H^r(G,M)$ is trivial. 

However, when one is given a normal subgroup $K$ of $G$, there are some
interesting automorphisms (needed for the Hochschild-Serre
spectral sequence). Again fix $x \in G$, but now write $\theta_x$ for the
automorphism $k \mapsto xkx^{-1}$ of $K$. On any $G$-module $M$ we have the 
canonical $\theta_x$-automorphism $\theta_M(m)=xm$. So there is an induced
automorphism $\theta_x$ on $H^r(K,M)$, and it is often non-trivial. This construction 
yields an action of $G$ on $H^r(K,M)$, and the normal subgroup $K$
acts trivially. We will refer to the resulting action of $G/K$ on 
$H^r(K,M)$ as the \emph{Hochschild-Serre action}. 

\subsection{Restriction} 
Let $K$ be a subgroup of $G$. 
One way to think about  restriction homomorphisms in group cohomology is as
follows. Let $P$ be a $\mathbb Z[G]$-free resolution of $\mathbb Z$. Then 
$P$ is also a $\mathbb Z[K]$-free resolution of $\mathbb Z$.   

For any $G$-module $A$ there is an obvious inclusion $A^G \subset A^K$.
Applying this simple observation to the $G$-modules $\Hom(P_r,M)$ ($M$ being
some $G$-module), we obtain inclusions $\Hom(P_r,M)^G \hookrightarrow
\Hom(P_r,M)^K$, and these give rise to restriction maps 
\[
\Res_{G/K}: H^r(G,M) \to H^r(K,M)
\]
for any $G$-module $M$. 

\subsection{Corestriction} Let $K$ be a subgroup of $G$, and assume that
the index $[G:K]$ is finite. For any $G$-module $A$ there is a norm map 
$N_{G/K}:A^K \to A^G$, defined by 
\[
N_{G/K} (a) =\sum_{g \in G/K} ga. 
\] 
Applying this construction to the $G$-modules $\Hom(P_r,M)$, we obtain
induced maps 
\[
 H^r(K,M) \to H^r(G,M), 
\] 
called corestriction maps, and denoted by $\Cor_{G/K}$. 

It is a standard result that $\Cor_{G/K} \Res_{G/K} = [G:K]$. When $K$ is
normal in $G$, we have another standard result.  

\begin{lemma} \label{lem.CorRes}
Assume that $K$ is a normal subgroup of finite index in $G$, and let $M$
be a $G$-module.    We then have
the Hochschild-Serre action of $G/K$ on $H^r(K,M)$. 
\begin{enumerate}
\item The composed  map $\Res_{G/K} \Cor_{G/K}$
coincides with  the norm map $N_{G/K}$ formed using the action of
$G/K$ on $H^r(G,M)$.  
\item The corestriction map 
$\Cor_{G/K}:H^r(K,M) \to H^r(G,M)$ factors through the canonical
surjection 
\[
H^r(K,M) \twoheadrightarrow H^r(K,M)_{G/K}
\]  
from $H^r(K,M)$ to the group of $G/K$-coinvariants for the
Hochschild-Serre action of $G/K$.  
\end{enumerate}
\end{lemma}
\begin{proof}
For (1) see Corollary 9.2 on page 257 in
Cartan-Eilenberg, though they  are
treating Tate cohomology and are therefore assuming that $G$ is finite.
This makes no real difference. 

For (2) we can reason as follows. Corestriction is functorial in the
following sense. Suppose that we are given an automorphism $\theta$ of $G$
that preserves  a subgroup $K$ of finite index in $G$. Suppose too that
we are given a
$\theta$-automorphism $\theta_M$ of some $G$-module $M$. We then obtain
induced automorphisms (denoted by $\theta$) on the cohomology groups
$H^r(G,M)$ and $H^r(K,M)$, and the square 
\[
\begin{CD}
H^r(K,M) @>{\Cor_{G/K}}>> H^r(G,M) \\
@V{\theta}VV @V{\theta}VV\\
H^r(K,M) @>{\Cor_{G/K}}>> H^r(G,M)
\end{CD}
\] 
commutes. When $K$ is normal, we may take $\theta$ to be the inner
automorphism $\Int(x)$ obtained from some element $x \in G$. 
We may take $\theta_M$ to be $m \mapsto xm$. 
Then the
induced automorphism of $H^r(G,M)$ is trivial, and the induced automorphism
of $H^r(K,M)$ is precisely the Hochschild-Serre action of $x \in G/K$. The
commutativity of our square for all $x \in G/K$ then tells us that
$\Cor_{G/K}$ factors through the coinvariants of $G/K$ on $H^r(K,M)$. 
\end{proof} 

\subsection{Corestriction for coinduced modules} 
Again let $K$ be a subgroup of finite index in $G$. 
For any $K$-module $M$ we write $R(M)$ for the $G$-module coinduced from
$M$. Thus an element in $R(M)$ is a function $f:G \to M$ such that
$f(kx)=kf(x)$ for all $k \in K$, $x \in G$. The group $G$ acts by right
translations. The adjunction morphism $\epsilon:R(M)
\twoheadrightarrow M$ is the $K$-map given by $\epsilon(f)=f(1_G)$.  
There is a canonical $K$-map 
$j:M \hookrightarrow R(M)$ such that $\epsilon j=\id_M$. It sends $m \in M$
to the element
$f_m
\in R(M)$ defined by 
\[
f_m(x):= 
\begin{cases}
xm &\text{ if $x \in K$,} \\
0 &\text{ otherwise.}
\end{cases}
\] 
We denote by $Sh:H^i(G,R(M)) \to H^i(K,M)$ the Shapiro isomorphism.

\begin{lemma} \label{lem.IndCor}
For any $K$-module $M$ and any $i \ge 0$ 
the  composed map 
\begin{equation}\label{map.RCor}
H^i(K,M) \xrightarrow{j} H^i(K,R(M)) \xrightarrow{\Cor_{G/K}} H^i(G,R(M))
\xrightarrow{Sh} H^i(K,M)
\end{equation}
is equal to the identity map on $H^i(K,M)$. 
\end{lemma}
\begin{proof} 
Observe that the functor $R$ is exact and that 
all three arrows in \eqref{map.RCor} are actually morphisms of
cohomological $\partial$-functors. Since the initial cohomological
$\partial$-functor is universal, in order to check that \eqref{map.RCor} is
always the identity map, it suffices to do so when $i=0$. This is a simple
computation. 
\end{proof}

\subsection{Corestriction of homogeneous cochains} 
Again let $K$ be a subgroup of finite index in $G$. We want to give a
cocycle-level formula for corestriction $\Cor_{G/K}$. Of course this
involves a choice. As we have seen, corestriction is most easily understood
when one views a free $G$-resolution of $\mathbb Z$ as also being a free
$K$-resolution of $\mathbb Z$. For this purpose we use the standard
resolution  $P_r(G)=\mathbb Z[G^{r+1}]$ for $G$. We will also need the
standard resolution $P_r(K)$ for $K$, given by $P_r(K)=\mathbb Z[K^{r+1}]$. 
Both $\Hom(P(G),M)^K$ and $\Hom(P(K),M)^K$ compute the $K$-cohomology of
$M$. To relate them, we can use any morphism $P(G) \to
P(K)$ of complexes of
$K$-modules that induces the identity on the module $\mathbb Z$ that is
being resolved.  The most obvious way to get such a morphism is to choose a
 map $p:G \to K$ satisfying $p(kg)=kp(g)$ (for all $k \in K$, $g \in G$),
and then to define
$\tilde p:P_r(G)
\to P_r(K)$ as the $\mathbb Z$-linear map induced by $(g_0,\dots,g_r)
\mapsto (p(g_0),\dots,p(g_r))$.  

Let us use homogeneous cochains. So, we start with a standard homogeneous 
cochain for $K$ with values in $M$. This is simply a map $f:K^{r+1} \to M$
satisfying $f(xk_0,\dots,xk_r)=xf(k_0,\dots,k_r)$ for all $x \in K$. From
this we obtain a map $f_1:G^{r+1} \to M$, defined by
$f_1(g_0,\dots,g_r)=f(p(g_0),\dots,p(g_r))$. The corestriction $f_2$ of
$f$ is obtained by applying $N_{G/K}$ to $f_1$. Thus 
\begin{equation}
f_2(g_0,\dots,g_r)=\sum_{x \in G/K}
x\bigl(f(p(x^{-1}g_0),\dots,p(x^{-1}g_r)\bigr).
\end{equation}  
This too is a homogeneous cochain. 

\subsection{Compatibility of corestriction with pullback maps} 

Suppose we are given a finite group $G$ and a commutative diagram 
 \begin{equation*}
\begin{CD}
1 @>>> A' @>>> E'     @>>> G @>>> 1 \\
 @. @VhVV  @V{\tilde h}VV @| @. \\
1 @>>> A @>>> E     @>>> G @>>> 1 
\end{CD}
\end{equation*} 
with exact rows. The homomorphisms $h$, $\tilde h$ induce pullback 
maps $h^*$, $\tilde h^*$ on group cohomology. These are inflation maps  
when $h$, $\tilde h$ are surjective. 

\begin{lemma}\label{lem.CorInf}
Let $M$ be an $E$-module. 
Then the square 
\begin{equation}
\begin{CD}
H^r(A,M) @>{\Cor_{E/A}}>> H^r(E,M) \\
@V{h^*}VV @V{\tilde h^*}VV \\
H^r(A',M) @>{\Cor_{E'/A'}}>> H^r(E',M)
\end{CD}
\end{equation}
commutes for all $r \ge 0$. 
\end{lemma}
\begin{proof}
This follows easily from the explicit formula we gave for corestriction of
homogeneous cochains. 
\end{proof}

\subsection{Explicit formula for corestriction of inhomogeneous $1$-cochains} 
In the case of $1$-cochains let us now rewrite  the   
 formula for corestriction in terms of 
inhomogeneous cochains.  So, we start with an inhomogeneous 
$1$-cochain for
$K$, i.e.~a map
$\phi:K
\to M$. The corresponding homogeneous $1$-cochain $f$ is given by
$f(k_0,k_1)=k_0\bigl(\phi(k_0^{-1}k_1)\bigr)$. Corestriction sends this to
the homogeneous $1$-cochain $f_2$. The corresponding inhomogeneous
$1$-cochain $\psi$ (which represents the corestriction of $\phi$) is given 
by  
\[
\psi(g)=f_2(1,g)=\sum_{y \in G/K} y f\bigl(p(y^{-1}),p(y^{-1}g)
\bigr)=\sum_{x \in K\backslash G}
x^{-1}p(x)\bigl(\phi(p(x)^{-1}p(xg)) \bigr). 
\] 

Having made this computation, we no longer have any use for homogeneous cochains, and 
we revert to our usual practice of referring to inhomogeneous cochains simply 
as cochains. The same goes for cocycles.  

Let us apply the computation above to the following very special case. We consider an
extension 
\[
1 \to A \to E \to G \to 1
\] 
of a finite group $G$ by an abelian group $A$. For any $G$-module $M$ we 
are interested in the corestriction map $H^1(A,M) \to H^1(E,M)$. 
Because $A$ is abelian, $H^1(A,M)$ is equal to $\Hom(A,M)$.

As usual, inner automorphisms
by elements in $G$ make $A$ into a $G$-module. Choose a set-theoretic
section
$s:G \to E$, and define a $2$-cocycle $\alpha$ of $G$ in $A$ by
the rule $s(\sigma)s(\tau)=\alpha_{\sigma,\tau}s(\sigma\tau)$. Our cochain level
version of corestriction requires the choice of a map $p:E \to A$ such that
$p(aw)=ap(w)$ for all $a \in A$, $w \in E$. The obvious way to get such a
map $p$ is to put $p(as(\sigma)):=a$ for all $a\in A$, $\sigma \in G$. 

A $1$-cocycle of $A$ in $M$ is a homomorphism $\mu:A \to M$ of abelian
groups. An easy computation shows that the corestriction of $\mu$ to $E$ is
represented by the
$1$-cocycle
$b$ of
$E$ in $M$ defined by 
\begin{align*}
b_{as(\sigma)}&=\sum_{\tau \in G} \tau^{-1}\bigl(
\mu(\tau(a))+\mu(\alpha_{\tau,\sigma}) \bigr) \\ 
&= (N_G\mu)(a) + \sum_{\tau \in G} \tau^{-1}\bigl(\mu
(\alpha_{\tau,\sigma})\bigr)
\end{align*} 
for all $a \in A$, $\sigma \in G$. (This provides a nice illustration of
the principle that $\Res \circ \Cor$ is $N_G$ when we are dealing with a
normal subgroup and the quotient group is $G$.) 

In the special case when $N_G\mu=0$, the formula above shows  that  $b$ 
is inflated from 
 the $1$-cocycle $b'$ of $G$ in $M$ given by 
\begin{equation}
b'_\sigma=\sum_{\tau \in G} \tau^{-1}\bigl(\mu
(\alpha_{\tau,\sigma})\bigr).
\end{equation} 
Now Tate cohomology makes sense for the finite group $G$, and, still 
assuming that $N_G\mu=0$,
 we may view $\mu$ as a $(-1)$-cocycle of $G$ in  
 $\Hom(A,M)$, so we can also form the 
cup-product $c:=\alpha \smile \mu \in Z^1(G,M)$. 

\begin{lemma}\label{lem.CorCup}
The cocycles $b'$ and $c$ are  are
cohomologous.
\end{lemma} 
\begin{proof}
This follows from the lemma in the next subsection. 
\end{proof}

\subsection{Some formulas for cup products}
Let $G$ be a finite group and let $A$, $B$ be $G$-modules. We then have
Tate cohomology groups and  cup product pairings 
\[
H^p(G,A) \otimes  H^q(G,B) \to  H^{p+q}(G,A \otimes B). 
\]
We need a cochain level formula for this cup product when $p=2$ and $q=-1$. 
\begin{lemma}\label{lem.cup} 
Let $a_{\sigma,\tau}$ be a $2$-cocycle of $G$ in $A$ and let
$b$ be a $(-1)$-cocycle of $G$ in $B$. Thus $b$ is an element of $B$ such
that $Nb=0$, where $Nb:=\sum_{\sigma \in G} \sigma b$. Then the cup product
$c=a \smile b$ is represented by the $1$-cocycle 
\[
c_\sigma=\sum_{\tau \in G} a_{\sigma,\tau} \otimes \sigma\tau b.
\] 
Moreover the $1$-cocycle 
\[
d_\sigma=\sum_{\tau \in G} \tau^{-1} a_{\tau,\sigma} \otimes \tau^{-1} b
\]
is cohomologous to $c_\sigma$, so it too represents   $a \smile b$. 
\end{lemma}
\begin{proof}
The  formula for $c$  comes from the article by 
Atiyah-Wall in \cite{CF}. It remains to prove that $c$ and $d$ are
cohomologous.  
Now we will certainly get a $1$-cocycle $c'_\sigma$ cohomologous to
$c_\sigma$ if we replace $a_{\sigma,\tau}$ by a $2$-cocycle
$a'_{\sigma,\tau}$ cohomologous to
$a_{\sigma,\tau}$, and in fact we will see that $d_\sigma=c'_\sigma$ for a
suitable choice of $a'_{\sigma,\tau}$. 

The right choice for $a'_{\sigma,\tau}$ turns out to be  
\[
a'_{\sigma,\tau}=-\sigma\tau a_{\tau^{-1},\sigma^{-1}}.
\]
 Why is $a'$ cohomologous to $a$? Use the $2$-cocycle $a$ to build an
extension
$p:E\twoheadrightarrow G$ of $G$ by $A$, equipped with a set-theoretic
section
$s:G
\to E$ of $p$. By construction we have the multiplication rule $s_\sigma
s_\tau=a_{\sigma,\tau}s_{\sigma\tau}$ in $E$. Any other set-theoretic
section $s'$ gives rise to a cohomologous $2$-cocycle.  The $2$-cocycle  
$a'$ is obtained in this way from the section $s'$  defined by
$s'_\sigma=(s_{\sigma^{-1}})^{-1}$.  

For this choice of $a'$ we find that $c'_\sigma$ is given by 
\[
c'_\sigma=-\sum_{\tau \in G} \sigma\tau a_{\tau^{-1},\sigma^{-1}} \otimes
\sigma\tau b.
\]
The $1$-cocycle property for $c'$ implies that $c'_\sigma+\sigma
c'_{\sigma^{-1}}=0$. Therefore we have 
\[
c'_\sigma=-\sigma c'_{\sigma^{-1}}=\sum_{\tau \in G} 
\tau a_{\tau^{-1},\sigma} \otimes \tau b.
\]
Replacing $\tau$ by $\tau^{-1}$ in this last sum, we find that 
$c'_\sigma=d_\sigma$, as claimed. 
\end{proof}

\bibliographystyle{amsalpha} 
\providecommand{\bysame}{\leavevmode\hbox to3em{\hrulefill}\thinspace}

\end{document}